\documentclass[10pt,a4paper,american]{amsart}
\usepackage[left=3.5cm, right=3.5cm]{geometry}
\usepackage[american]{babel}
\usepackage{amsthm}
\usepackage{amsbsy}
\usepackage{amstext}
\usepackage{amsfonts}
\usepackage{mathcomp}
\usepackage{dsfont}
\usepackage{latexsym}
\usepackage{amssymb}
\usepackage{esint}
\usepackage[colorlinks,pdfpagelabels,pdfstartview = FitH,bookmarksopen = true,bookmarksnumbered = true,linkcolor = blue,plainpages = false,hypertexnames = false,citecolor = red,pagebackref=true]{hyperref}
\usepackage{cite}
\usepackage{stmaryrd}
\usepackage{comment}

\newtheorem{theo}{Theorem}[section]

\newtheorem{lem}[theo]{Lemma}
\theoremstyle{definition}
\newtheorem{defin}[theo]{Definition}
\newtheorem*{lem*}{Lemma}
\newtheorem{rem}[theo]{Remark}

\newtheorem*{cor*}{Corollary}
\newtheorem*{theo*}{Theorem}

\DeclareMathOperator*{\osc}{osc}
\newcommand{\norm}[1]{\lVert#1\rVert}

\DeclareMathOperator*{\esssup}{ess\,sup}
\DeclareMathOperator*{\essinf}{ess\,inf}
\DeclareMathOperator*{\essosc}{ess\,osc}
\DeclareMathOperator*{\supp}{supp}

\def\XXint#1#2#3{{\setbox0=\hbox{$#1{#2#3}{\int}$ }
\vcenter{\hbox{$#2#3$ }}\kern-.6\wd0}}

\def\XXiint#1#2#3{{\setbox0=\hbox{$#1{#2#3}{\iint}$ }
\vcenter{\hbox{$#2#3$ }}\kern-.55\wd0}}

\renewcommand{\d}{\:\:\!\!\mathrm{d}}
\newcommand{\N}{\ensuremath{\mathbb{N}}}
\newcommand{\R}{\ensuremath{\mathbb{R}}}

\newcommand{\Q}{\mathbb{Q}}

\renewcommand{\b}{\mathfrak{b}}

\numberwithin{equation}{section}
\allowdisplaybreaks

\begin{document}
\renewcommand{\refname}{References}
\renewcommand{\abstractname}{Abstract}

\title[ Regularity of solutions to doubly singular equations]{An extensive study of the regularity of solutions to doubly singular equations}
\date{\today}

\author[V. Vespri]{Vincenzo Vespri}
\address{Vincenzo Vespri\\
Universit\`a degli Studi di Firenze, Dipartimento di Matematica ed Informatica "Ulisse Dini" \\
Viale Morgagni 67/a, 50134 Firenze, Italy \\
Member of G.N.A.M.P.A. (I.N.d.A.M.) }
\email{vincenzo.vespri@unifi.it} 

\author[M. Vestberg]{Matias Vestberg}
\address{Matias Vestberg\\
Department of Mathematics and Systems Analysis, Aalto University\\
P.~O.~Box 11100, FI-00076 Aalto University, Finland}
\email{matias.vestberg@aalto.fi}

\begin{abstract}
In recent years, many papers have been devoted to the regularity of doubly nonlinear singular evolution equations. Many of the proofs are unnecessarily complicated, rely on superfluous assumptions or follow an inappropriate approximation procedure. This makes the theory unclear and quite chaotic to a nonspecialist. The aim of this paper is to fix all the misprints, to follow correct procedures, to exhibit, possibly, the shortest and most elegant proofs and to give a complete and self-contained overview of the theory.
\end{abstract}
\maketitle
 
\section{Introduction}
This work is concerned with the regularity properties of weak solutions to doubly nonlinear equations whose model case is
\begin{align}\label{DNPE}
\partial_t u - \nabla\cdot (u^{m-1} |\nabla u|^{p-2}\nabla u) = 0 \quad \text{ in } \Omega_T:=\Omega\times (0,T),
\end{align}
where $\Omega\subset \R^n$ is an open bounded set, and the parameters $m$ and $p$ are restricted to the range
\begin{align}\label{parameter-range}
p\in (1,2), \qquad \qquad m > 1, \qquad\text{ and } \qquad 2<m+p<3.
\end{align}
The term doubly nonlinear refers to the fact that the diffusion part depends nonlinearly both on the gradient and the solution itself. Such kind of equations describe several physical phenomena and were introduced by \cite{Li} (see also the nice survey by Kalashnikov \cite{Ka}). Moreover, these equations have an intrinsic mathematical interest because they represent a natural bridge between the more natural generalisations of the heat equation: the parabolic $p$-Laplace and the Porous Medium equations. 

Especially in recent years, many papers have been devoted to this topic. The approaches are sometimes not rigorous, sometimes not with sharp assumptions or with unnecessarily long proofs.
The natural definition of weak solutions is obtained from \eqref{DNPE} by a formal application of the chain rule and requires that a certain power of $u$ (rather than $u$ itself) has a weak gradient. This is perhaps the most delicate point: too many papers devoted to this topic do not take this aspect into account carefully, and use incorrect approximations or non-admissible test-functions. For more details, we refer the reader to Section \ref{weaksolsect}.

Analogously, some results presented below, such as the $L^1$-Harnack inequality and the expansion of positivity have been obtained previously under the assumption that the function $u$ itself has weak gradient, see \cite{FoSoVe} and \cite{FoSoVe2}. Since this is not necessarily true in our setting we have included detailed proofs showing that the strategies developed in \cite{FoSoVe} and \cite{FoSoVe2} are applicable also without assuming the existence of $\nabla u$.
But we do not limit ourselves to fix this aspect. We go through the regularity theory and we use a unified approach giving shorter and different proofs with respect to the ones known in literature. In this way, a reader can have a self-contained overview of the theory of doubly nonlinear singular parabolic equations.
We obtain different results under various ranges for the parameters. The time continuity, mollified weak formulation, energy estimates, expansion of positivity and $L^1$-Harnack inequality are obtained in the full range \eqref{parameter-range}. Local boundedness of weak solutions is shown in the smaller range
\begin{align}\label{nicerange}
m+p>3-\frac{p}{n-(\frac{n-p}{p})}.
\end{align}
We recall that this range is sharp. In the special case $m=1$, \eqref{DNPE} becomes the singular parabolic $p$-Laplace equation. Then the condition \eqref{nicerange} and the integrability required of $u$ in Definition \ref{weakdef} below reduce to $p>\frac{2n}{n+2}$ and $u\in L^2$ respectively, which are well-known sharp conditions to guarantee local boundedness for this equation, see for example Chapter V of \cite{DiBene}.
 
The local H\"older continuity will be proven only in the so-called supercritical range
\begin{align}\label{supercritical}
m+p>3-\frac{p}{n}.
\end{align}
Note that \eqref{supercritical} is a stricter condition than \eqref{nicerange}. We decided that it was too much dispersive for the reader to prove H\"older continuity also in the sub-critical case because requires a slighty different approach (and assumptions).
In the last section, we prove Harnack estimates in the supercritical range. Note that, as proven in \cite{DiGiaVe1} for the p-Laplacian, this result is sharp.
\medskip

\noindent
{\bf Acknowledgments.} M. Vestberg wants to express gratitude to the Academy of Finland. Moreover, we thank Juha Kinnunen for useful discussions and feedback during the writing of this article.

\section{Setting and main result}\label{weaksolsect}
In order to motivate the natural definition of weak solutions, we reformulate \eqref{DNPE}. Formally applying the chain rule, we can write the equation in the form
\begin{align}\label{reformulated}
\partial_t u-\nabla \cdot (\beta^{1-p}|\nabla u^\beta |^{p-2}\nabla u^\beta )= 0,
\end{align}
where
\begin{align}\label{def:beta}
\beta:=1 + \frac{m-1}{p-1} >1.
\end{align}
For later reference we note that \eqref{nicerange} can be expressed conveniently in terms of $\beta$, $p$ and $n$ as
\begin{align}\label{nicerange-rephrased}
\frac{p(\beta+1)}{1-\beta(p-1)} > n.
\end{align}
We will prove our result not only for solutions to \eqref{reformulated}, but for all equations of the form
\begin{align}\label{general}
\partial_t u-\nabla \cdot A(x,t, u,\nabla u^\beta)=0,
\end{align}
where $A(x,t, u,\xi)$ is a vector field satisfying
\begin{align}
\label{structcond1}|A(x,t, u,\xi)|&\leq C_1 |\xi|^{p-1}
\\
\label{structcond2}A(x,t, u,\xi)\cdot \xi &\geq C_0|\xi|^p
\end{align}
An example of an equation that satisfies these conditions is
\begin{align}\label{DNPEG}
\partial_t u - \sum_{i,j=1}^n(a_{ij}(x,t)\beta^{1-p} |\nabla u^\beta|^{p-2} u^{\beta}_{x_i})_{x_j} = 0 \quad \text{ in } \Omega_T:=\Omega\times (0,T),
\end{align}
where the coefficients $a_{ij} $ are bounded and measurable and where the matrix $(a_{ij}(x,t))^n_{i,j=1}$ is positive definite uniformly in $(x,t)$. 
We arrive at the definition of weak solutions by multiplying \eqref{general} by a smooth test function and integrating formally by parts.
\begin{defin}\label{weakdef}
A function $u\colon\Omega_T\to \R $ is a weak solution to \eqref{general} if and only if $u\geq 0$, $u^\beta \in L^p(0,T;W^{1,p}(\Omega))$, $u\in L^{\beta+1}(\Omega_T)$ and
\begin{align}\label{weakform2}
&\iint_{\Omega_T} A(x,t,u,\nabla u^\beta)\cdot \nabla \varphi- u\partial_t \varphi\d x\d t=0,
\end{align}
for all $\varphi \in C^\infty_0(\Omega_T)$.
\end{defin}
\begin{rem}
The extra integrability condition $u\in L^{\beta+1}(\Omega_T)$ is made to justify a test function containing $u^\beta$. The condition is needed since we are considering the fast diffusion case, in which $\beta p< \beta +1$. By contrast, in the slow diffusion case $m+p>3$ which is not considered in this article, the inequality holds in the reverse direction, which means that no additional integrability is needed. For explicit calculations illustrating this point, consider the flat case of the equation studied in \cite{SiVe} and \cite{SiVe2}. Earlier works treating the slow diffusion case (although not necessarily with the same definition) are \cite{PoVe} and \cite{Iv3}.
\end{rem}

\section{Preliminaries}
\label{Preliminaries}
Here we introduce some notation and present auxiliary tools that will be useful in the course of the paper.
\subsection{Notation}
With $B_\rho(x_o)$ we denote the open ball in $\R^n$ with radius $\rho$ at center $x_o$, and the corresponding closed ball is denoted $\bar B_\rho(x_o)$. Furthermore, we use the notation $Q_{\rho,\theta}(z_o):=B_\rho(x_o)\times (t_o-\theta,t_o)$ for space-time cylinders, where $z_o:=(x_o,t_o)\in \Omega_T$.
For $w,v \geq 0$ we define
\begin{align}
\label{definition:b}
&\b[v,w]:= \tfrac{1}{\beta+1} (v^{\beta+1}-w^{\beta+1})- w^\beta (v-w)
\\
\notag &\hphantom{\b[v,w]:}= \tfrac{\beta}{\beta+1} (w^{\beta+1}-v^{\beta+1})- v (w^\beta-v^\beta),
\\ \notag
\\
&\b[v,w]^+:=\b[v,w]\chi_{(w,\infty)}(v),
\end{align}
where $\beta$ is defined by \eqref{def:beta}. For any real-valued essentially bounded function $g$ defined on a measurable set $E\subset \R^{n+1}$ we define its essential oscillation in $E$ as
\begin{align*}
\essosc_E g:= \esssup_E g-\essinf_E g.
\end{align*}
The oscillation $\osc_E g$ of a bounded function $g$ is defined analogously, using the ordinary supremum and infimum. The parameters $C_0, C_1,m,n,p$ will collectively be referred to as the data. 
\subsection{Auxiliary tools}
We now recall some elementary lemmas that will be used later, and start by defining a mollification in time as in \cite{KiLi}, see also \cite{BoeDuMa}. For $T>0$, $t\in [0,T]$, $h\in (0,T)$ and $v\in L^1(\Omega_T)$ we set
\begin{align}
\label{def:moll}
v_h(x,t):=\frac{1}{h}\int^t_0 e^\frac{s-t}{h}v(x,s)\d s.
\end{align}
Moreover, we define the reversed analogue by
\begin{align*}
v_{\overline h}(x,t) :=\frac{1}{h}\int^T_t e^\frac{t-s}{h}v(x,s)\d s.
\end{align*}
For details regarding the properties of the exponential mollification we refer to \cite[Lemma 2.2]{KiLi}, \cite[Lemma 2.2]{BoeDuMa}, \cite[Lemma 2.9]{St}. The properties of the mollification that we will use have been collected for convenience into the following lemma:
\begin{lem}
\label{expmolproperties} Suppose that $v \in L^1(\Omega_T)$, and let $p\in[1,\infty)$. Then the mollification $v_h$ defined in \eqref{def:moll} has the following properties:
\begin{enumerate}
\item[(i)]
If $v\in L^p(\Omega_T)$ then $v_h\in L^p(\Omega_T)$,
$$
\norm{v_h}_{L^p(\Omega_T)}\leq \norm{v}_{L^p(\Omega_T)},
$$
and $v_h\to v$ in $L^p(\Omega_T)$.
\item[(ii)]
In the above situation, $v_h$ has a weak time derivative $\partial_t v_h$ on $\Omega_T$ given by
\begin{align*}
\partial_t v_h=\tfrac{1}{h}(v-v_h),
\end{align*}
whereas for $v_{\overline h}$ we have
\begin{align*}
\partial_t v_{\overline h}=\tfrac{1}{h}(v_{\overline h}-v).
\end{align*}
\item[(iii)]
If $v\in L^p(0,T;W^{1,p}(\Omega))$ then $v_h\to v$ in $L^p(0,T;W^{1,p}(\Omega))$ as $h\to 0$.
\item[(iv)] If $v\in L^p(0,T;L^{p}(\Omega))$ then $v_h \in C([0,T];L^{p}(\Omega))$.
\end{enumerate}
\end{lem}

The next Lemma provides us with some useful estimates for the quantity $\b[v,w]$ that was defined in \eqref{definition:b}. The proof can be found in \cite[Lemma 2.3]{BoDuKoSc}.
\begin{lem}
\label{estimates:boundary_terms}
Let $v,w \geq 0$ and $\beta> 1$. Then there exists a constant $c$ depending only on $\beta$ such that:
\begin{enumerate}
\item[(i)] $\tfrac 1 c\big| w^{\frac{\beta+1}{2}}-v^{\frac{\beta+1}{2}} \big|^2 \leq \b[v,w] \leq c \big| w^{\frac{\beta+1}{2}}-v^{\frac{\beta+1}{2}} \big|^2$ \vspace{2mm}
\item[(ii)] $\tfrac 1 c | w^{\beta}-v^{\beta}|^2 \leq \left( w^{\beta-1}+v^{\beta-1}\right) \b[v,w] \leq c |w^{\beta}-v^{\beta}|^2 $
\item[(iii)]$\b[v,w] \leq c |v^{\beta}-w^{\beta}|^{\frac{\beta+1}{\beta}}$
\end{enumerate}
\end{lem}
 
Next, we recall a well-known parabolic Sobolev inequality, which can be found for example in \cite{DiBene}. For the proof, we refer to \cite[Lemma 3.2]{Sche}.
\begin{lem}
\label{lemma:Gagliardo}
Let $z_o=(x_o,t_o)\in \R^{n+1}$ and $\theta>0$. Suppose that $q>0$, $p>1$. Then for every
$$
u\in L^\infty (t_o-\theta,t_o;L^q(B_r(x_o))) \cap L^p(t_o-\theta,t_o;W^{1,p}_0(B_r(x_o)))
$$
we have
\begin{align*}
\iint_{Q_{r,\theta}(z_o)} |u|^{p(1+\frac qn)} \d x\d t& \leq c\bigg(\esssup_{t\in (t_o-\theta,t_o)} \int_{B_r(x_o)\times \{t\}} |u|^q \d x\bigg)^{\frac pn} \iint_{Q_{r,\theta}(z_o)} |\nabla u|^p \d x\d t
\end{align*}
for a constant $c=c(n,p,q)$.
\end{lem}
The following lemma can be proven using an inductive argument, see for example \cite[Lemma 7.1]{Gi}.
\begin{lem}
\label{fastconvg} Let $(Y_j)^\infty_{j=0}$ be a positive sequence such that
\begin{equation*}
Y_{j+1}\leq C b^j Y^{1+\delta}_j,
\end{equation*}
where $C, b >1$ and $\delta>0$. If
\begin{equation*}
Y_0\leq C^{-\frac{1}{\delta}}b^{-\frac{1}{\delta^2}},
\end{equation*}
then $(Y_j)$ converges to zero as $j\to\infty$.
\end{lem}
A form of the following lemma was originally proven by De Giorgi\cite{DeGi}, see also \cite{DiBene}.
\begin{lem}\label{isoperim}
Let $v\in W^{1,1}(B_\rho(x_o))$ for some $\rho>0$ and $x_o\in \R^n$. Let $k$ and $l$ be real numbers such that $k<l$. Then there exists a constant $c$ depending only on $n$ (and thus independent of $k,l,v,x_o$ and $\rho$) such that for any representative of $v$, we have
\begin{align*}
(l-k)|\{x\in B_\rho(x_o) : v(x)>l\}|\leq \frac{c \rho^{n+1}}{|\{x\in B_\rho(x_o) : v(x)<k\}|}\int_{\{k<v<l\}\cap B_\rho(x_o)}|\nabla v|\d x.
\end{align*}
\end{lem}

The following lemma is a special case of Theorem 1.1 in section IV.1 of \cite{DiBene}.
\begin{lem}\label{lem:p-parabolic_hold_cont}
Let $1<p<2$ and suppose that $v\in L^p(0,T;W^{1,p}(\Omega))\cap L^\infty(\Omega_T)$ is a weak solution to the equation
\begin{align*}
\partial_t v-\nabla \cdot \big(\tilde{A}(x,t,v,\nabla v)\big) =0,
\end{align*}
where $\tilde{A}$ satisfies the structure conditions
\begin{align*}
|\tilde{A}(x,t,v,\xi)|&\leq \tilde{C}_1|\xi|^{p-1}
\\
\tilde{A}(x,t,v,\xi)\cdot \xi &\geq \tilde{C}_0|\xi|^p.
\end{align*}
Then $v$ is locally H\"older continuous in $\Omega_T$ and there are constants $c>1$ and $\nu \in (0,1)$ depending only on $n,p, \tilde{C}_0, \tilde{C}_1$ such that for any subset $K\subset \Omega_T$, compactly contained in $\Omega\times (0,T]$, we have for all $(x,t),(y,s)\in K$ that
\begin{align*}
|v(x,t)-v(y,s)|\leq c\norm{v}_{L^\infty(\Omega_T)}\bigg( \frac{\norm{v}_{L^\infty(\Omega_T)}^\frac{2-p}{p}|x-y|+|t-s|^\frac{1}{p}}{d_p(K)} \bigg)^\nu,
\end{align*}
where
\begin{align*}
d_p(K):=\inf_{\substack{ (x,t)\in K \\ (y,s)\in \partial_p \Omega_T}}\Big( \norm{v}_{L^\infty(\Omega_T)}^\frac{2-p}{p}|x-y|+|t-s|^\frac{1}{p} \Big).
\end{align*}
\end{lem}

The next lemma shows that weak solutions to \eqref{general} which are bounded from below and above by positive constants are in fact also solutions to an equation of parabolic $p$-Laplace type (in the case $M=1$). It also investigates how solutions are affected by re-scaling.
\begin{lem}\label{re-scaling}
Let $A$ satisfy the structure conditions \eqref{structcond1} and \eqref{structcond2} and suppose that $u$ is a weak solution to \eqref{general} in the cylinder $B_R(x_o)\times (0,M^{3-m-p}\tau)$. Suppose furthermore that
\begin{align}\label{ubounds}
\beta_0 M \leq u \leq \beta_1 M,
\end{align}
for some positive constants $\beta_0,\beta_1$. Then the function
\begin{align*}
v(x,t)=M^{-1}u(x, M^{3-m-p}t), \hspace{5mm} (x,t)\in B_R(x_o)\times (0,\tau),
\end{align*}
has a weak $p$-integrable gradient, and is a weak solution in $B_R(x_o)\times (0,\tau)$ to the equation
\begin{align}\label{tildeweakform}
\partial_t v-\nabla \cdot \big(\tilde{A}(x,t,\nabla v)\big)=0,
\end{align}
where
\begin{align*}
\tilde{A}(x,t,\xi):= M^{2-m-p}A\big(x, M^{3-m-p}t, Mv(x,t), \beta M^\beta v^{\beta-1}(x,t)\xi\big).
\end{align*}
The vector field $\tilde{A}$ satisfies the structure conditions
\begin{align*}
|\tilde{A}(x,t,\xi)|&\leq C_1\beta^{p-1}\beta_1^{(\beta-1)(p-1)}|\xi|^{p-1}
\\
\tilde{A}(x,t,\xi)\cdot \xi &\geq C_0\beta^{p-1}\beta_0^{(\beta-1)(p-1)}|\xi|^p,
\end{align*}
where $C_0$ and $C_1$ are the constants appearing in the structure conditions \eqref{structcond1} and \eqref{structcond2}.
\end{lem}
\begin{proof}[Proof]
The bounds on $u$ show that the chain rule holds in the following form:
\begin{align}\label{nabla_u}
\nabla u=\nabla (u^\beta)^\frac{1}{\beta}=\beta^{-1} u^{1-\beta}\nabla u^\beta.
\end{align}
Note especially that the lower bound on $u$ guarantees that $u^{1-\beta}$ stays bounded despite the negative exponent. From these observations it follows that also $v$ has a weak gradient which is $p$-integrable. By a change of variables in the time variable in the weak formulation \eqref{weakform2}, and by taking note of \eqref{nabla_u}, one can see that $v$ satisfies \eqref{tildeweakform} weakly. The structure conditions for $\tilde{A}$ follow from the corresponding conditions satisfied by $A$, and the bounds \eqref{ubounds}.
\end{proof}
 
\subsection{Continuity in time and mollified weak formulation}\label{subsec:time-cont}
In this subsection we show that weak solutions are continuous in time as maps into $L^{\beta+1}_{\textrm{loc}}(\Omega)$. The proof is adapted from \cite{St}. We start with a lemma.
\begin{lem}\label{lem:time-cont}
Suppose that $u$ is a weak solution in the sense of Definition \ref{weakdef} and define
\begin{align*}
\mathcal{V}:=\big\{ w\in L^{\beta+1}(\Omega_T)\, |\, w^\beta\in L^p(0,T;W^{1,p}(\Omega)), \, \partial_t w^\beta \in L^{\frac{\beta+1}{\beta}}(\Omega_T) \big\}.
\end{align*}
Then, for every $\zeta\in C^\infty_0(\Omega_T,\R_{\geq 0})$ and $w\in \mathcal V$ we have
\begin{align}
\label{eq:time_1}
\iint_{\Omega_T} \partial_t \zeta \b[u,w] \d x\d t = \iint_{\Omega_T} A(x,t,u,\nabla u^\beta) \cdot \nabla[\zeta(u^\beta-w^\beta)] + \zeta \partial_t w^\beta (u-w) \d x\d t.
\end{align}
\end{lem}
\begin{proof}[Proof]
Let $w\in \mathcal V$, $\zeta \in C^\infty_0(\Omega_T,\R_{\geq 0})$ and choose
$$
\varphi =\zeta \left( w^\beta-[u^\beta]_h \right)
$$
as test function in \eqref{weakform2}. Our goal is to pass to the limit $h\to 0$. It follows from Lemma \ref{expmolproperties} (iii) that
\begin{align*}
\iint_{\Omega_T}A(x,t,u,\nabla u^\beta)\cdot \nabla \varphi \d t \d t \xrightarrow[h\to 0]{} \iint_{\Omega_T} A(x,t,u,\nabla u^\beta)\cdot \nabla [\zeta (w^\beta-u^\beta)]\d x \d t.
\end{align*} Note that Lemma \ref{expmolproperties} (ii) implies
$$
\big( [u^\beta]_h^{\frac 1 \beta} -u\big) \partial_t [u^\beta]_h\leq 0,
$$
which shows that we can treat the parabolic part as follows.
\begin{align*}
\iint_{\Omega_T} u\partial_t \varphi \d x\d t&=\iint_{\Omega_T} \zeta u \partial_t w^\beta \d x\d t -\iint_{\Omega_T} \zeta [u^\beta]_h^{\frac 1 \beta} \partial_t [u^\beta]_h \d x\d t
\\
&\quad+ \iint_{\Omega_T} \zeta \big( [u^\beta]_h^{\frac 1 \beta} -u\big) \partial_t [u^\beta]_h \d x\d t+ \iint_{\Omega_T} \partial_t \zeta u \big( w^\beta-[u^\beta]_h \big) \d x\d t
\\
&\leq \iint_{\Omega_T} \zeta u \partial_t w^\beta\d x\d t+ \iint_{\Omega_T} \tfrac{\beta}{\beta+1} \partial_t \zeta [u^\beta]_h^{\frac{\beta+1}{\beta}} \d x\d t
\\
&\quad +\iint_{\Omega_T} \partial_t \zeta u \big(w^\beta-[u^\beta]_h \big) \d x\d t
\\
&\xrightarrow[h\to 0]{} \iint_{\Omega_T} \zeta u \partial_t w^\beta \d x\d t+ \iint_{\Omega_T} \partial_t \zeta \big(\tfrac{\beta}{\beta+1} u^{\beta+1}+u(w^\beta-u^\beta) \big) \d x\d t
\\
&\quad =\iint_{\Omega_T} \zeta \partial_t w^\beta (u-w) \d x\d t -\iint_{\Omega_T} \partial_t \zeta \b[u,w] \d x \d t,
\end{align*}
This shows ``$\leq$'' in \eqref{eq:time_1}. The reverse inequality can be derived in the same way by taking
$$
\varphi =\zeta \left( w^\beta-[v^\beta]_{\overline h} \right)
$$
as test function. \end{proof}

\begin{theo}\label{cont_into_Lbetaplusone}
Let $u$ be a weak solution in the sense of Definition \ref{weakdef}. Then
\\
$u\in C([0,T];L^{\beta+1}_\textrm{loc}(\Omega))$.
\end{theo}
\begin{proof}[Proof]
We prove continuity on the interval $[0,\tfrac12 T]$ and describe later how the argument can be modified to show continuity also on $[\tfrac12 T,T]$, thus completing the proof. We first note that due to Lemma \ref{expmolproperties}, $w:=([u^\beta]_{\bar{h}} )^\frac{1}{\beta}$ belongs to the set of admissible comparison functions $\mathcal{V}$ of Lemma \ref{lem:time-cont}. Furthermore, since Lemma \ref{expmolproperties} (iv) guarantees that $w^\beta$ is continuous $[0,T]\to L^\frac{\beta+1}{\beta}(\Omega)$ and since
\begin{align*}
|w(x,s)-w(x,t)|^{\beta+1}\leq |w^\beta (x,s)-w^\beta (x,t)|^\frac{\beta+1}{\beta}=|[u^\beta]_{\bar{h}}(x,s) - [u^\beta]_{\bar{h}}(x,t)|^\frac{\beta+1}{\beta},
\end{align*}
we see that $w$ is continuous $[0,T]\to L^{\beta+1}(\Omega)$. We will show that $u$ is essentially the uniform limit on the time interval $[0,\tfrac12 T]$ of the functions $w$ as $h\to 0$, and the continuity will follow from this.
For a compact set $K\subset \Omega$ we take $\eta \in C^\infty_0(\Omega;[0,1])$ such that $\eta=1$ on $K$ and $|\nabla \eta|\leq C_K$. Furthermore, take $\psi\in C^\infty([0,T];[0,1])$ with $\psi=1$ on $[T,\tfrac12 T]$, $\psi=0$ on $[\tfrac34 T,T]$ and $|\psi'|\leq \tfrac8T$. For $\tau\in (0,\tfrac12 T)$ and $\varepsilon>0$ so small that $\tau+\varepsilon< \tfrac12 T$ we define
\begin{align*}
\chi^\tau_\varepsilon(t)=
\begin{cases}
0, & t<\tau \\
\varepsilon^{-1}(t-\tau), & t\in [\tau, \tau+\varepsilon] \\
1, & t> \tau+\varepsilon.
\end{cases}
\end{align*}
We use \eqref{eq:time_1} with $\zeta=\eta\chi^\tau_\varepsilon\psi$ and $w=([u^\beta]_{\bar{h}} )^\frac{1}{\beta}$ to obtain
\begin{align*}
\varepsilon^{-1} \int^{\tau+\varepsilon}_\tau\int_\Omega \b[u,w]\eta \d x\d t &= \iint_{\Omega_T} A(x,t,u,\nabla u^\beta) \cdot \nabla[\eta(u^\beta-w^\beta)]\chi^\tau_\varepsilon\psi \d x \d t
\\
&\quad + \iint_{\Omega_T}\eta\chi^\tau_\varepsilon\psi \partial_t w^\beta (u-w) \d x\d t - \iint_{\Omega_T} \b[u,w]\eta \psi'\d x \d t
\\
&\leq \iint_{\Omega_T} |A(x,t,u,\nabla u^\beta)| (|\nabla u^\beta -\nabla [u^\beta]_{\bar{h}}| + |\nabla\eta| |u^\beta - [u^\beta]_{\bar{h}}|)\d x \d t
\\
& \quad + \frac{8}{T}\iint_{\Omega_T} \b[u,w]\d x \d t.
\end{align*}
Here we were able to drop the term involving $\partial_t w^\beta$ since Lemma \ref{expmolproperties} (ii) shows that the factors $\partial_t w^\beta$ and $(u-w)$ are of opposite sign, and hence their product is nonpositive. Passing to the limit $\varepsilon\to 0$ we see that
\begin{align}\label{gs}
\int_K \b[u,w](x,\tau) \d x & \leq C_K \iint_{\Omega_T} |A(x,t,u,\nabla u^\beta)| (|\nabla u^\beta -\nabla [u^\beta]_{\bar{h}}| + |u^\beta - [u^\beta]_{\bar{h}}|)\d x \d t
\\
\notag & \quad + \frac{8}{T}\iint_{\Omega_T} \b[u,w]\d x \d t
\end{align}
for all $\tau\in [0,\tfrac12 T]\setminus N_h$, where $N_h$ is a set of measure zero. Note that the integrand on the left-hand side can be estimated using Lemma \ref{estimates:boundary_terms} (ii) and the fact that $\beta>1$ as follows
\begin{align*}
|u-w|^{\beta+1}=(|u-w|^\frac{\beta+1}{2})^2\leq \big| u^{\frac{\beta+1}{2}}-w^{\frac{\beta+1}{2}} \big|^2 \leq c \b[u,w].
\end{align*}
For the term on the last line of \eqref{gs} we can use Lemma \ref{estimates:boundary_terms} (iii) to make the estimate
\begin{align*}
\b[u,w] &\leq c |u^\beta-[u^\beta]_{\bar{h}} \big|^{\frac{\beta+1}{\beta}} =c |u^\beta-[u^\beta]_{\bar{h}} \big|^\frac{1}{\beta}|u^\beta-[u^\beta]_{\bar{h}} \big|
\leq c(u + ([u^\beta]_{\bar{h}})^\frac{1}{\beta})|u^\beta - [u^\beta]_{\bar{h}}|.
\end{align*}
The first factor stays bounded in $L^{\beta+1}$ as $h\to 0$ and the second factor converges to zero in $L^\frac{\beta+1}{\beta}$ as $h\to 0$. The fact that $|A(u,\nabla u^\beta)|\in L^{p'}(\Omega_T)$ combined with Lemma \ref{expmolproperties} (iii) show that also the first integral on the right-hand side of \eqref{gs} converges to zero as $h\to 0$. Picking now a sequence $h_j\to 0$ and $w_j=([u^\beta]_{\bar{h}_j} )^\frac{1}{\beta}$ and $N:= \cup N_{h_j}$ (which has measure zero) we see that \eqref{gs} combined with the previous observations implies
\begin{align}\label{unif_limit}
\lim_{j\to \infty} \sup_{\tau\in [0,\frac{1}{2}T]\setminus N} \int_K |u-w_j|^{\beta+1}(x,\tau) \d x = 0.
\end{align}
As noted earlier, each $w_j$ is continuous as a map $[0,T]\to L^{\beta+1}(K)$ so the uniform limit \eqref{unif_limit} shows that $u$ has a representative which is continuous on $[0,\tfrac12T]\setminus N$. By the completeness of $L^{\beta+1}(K)$ we find a representative of $u$ which is continuous $[0,\tfrac12T]\to L^{\beta+1}(K)$. The continuity on $[\tfrac12T,T]$ follows from a similar argument with $w=([u^\beta]_h)^\frac{1}{\beta}$ and with $\psi$ and $\chi^\tau_\varepsilon$ mirrored on the interval $[0,T]$ under the map $t\mapsto T-t$.\end{proof}

Now that we have established the continuity in time it is possible to show that weak solutions in the sense of Definition \ref{weakdef} satisfy a mollified weak formulation.
\begin{lem}\label{lem:mollified}
Let $u$ be a weak solution in the sense of Definition \ref{weakdef}. Then we have
\begin{align}\label{h-averaged-form}
\iint_{\Omega_T} [A(x,\cdot,u, \nabla u^\beta)]_h\cdot \nabla \phi+\partial_t u_h \phi\d x\d t- \int_\Omega u(x,0) \phi_{\bar{h}}(x,0) \d x = 0 
\end{align}
for all $\phi \in C^\infty(\Omega \times [0,T])$ with support contained in $K\times [0,\tau]$ ,where $K\subset \Omega$ is compact and $\tau \in (0,T)$. Here $u(x,0)$ refers to the value at time zero of the continuous representative of $u$ as a map $[0,T]\to L^{\beta+1}(K)$.
\end{lem}
\begin{proof}[Proof]
Consider the piecewise smooth function
\begin{align*}
\eta_\varepsilon (t):=
\begin{cases}
\, \frac{t}{\varepsilon}, & t\in[0,\varepsilon]
\\
\, 1, & t\in (\varepsilon, T],
\end{cases}
\end{align*}
and use \eqref{weakform2} with the test function $\varphi=\eta_\varepsilon\phi_{\bar{h}}$. Taking the limit $\varepsilon \to 0$ and using Fubini's theorem we see that the elliptic term will converge to the integral of $[A(x,\cdot,u,\nabla u^\beta)]_h \cdot \nabla \phi$. Note now that
\begin{align*}
\iint_{\Omega_T} u\partial_t(\eta_\varepsilon\phi_{\bar{h}})\d x \d t = \iint_{\Omega_T}u\eta_\varepsilon\frac{\phi_{\bar{h}}-\phi}{h}\d x \d t +\varepsilon^{-1} \int^\varepsilon_0\int_\Omega u \phi_{\bar{h}} \d x \d t.
\end{align*}
In the first term we can pass to the limit $\varepsilon\to 0$, use Fubini's theorem and Lemma \ref{expmolproperties} (ii) to obtain the integral of $\partial_t u_h \varphi$. It remains to investigate what happens to the last term in the limit $\varepsilon\to 0$. Note that we can write this term as
\begin{align*}
\varepsilon^{-1} \int^\varepsilon_0\int_K u \phi_{\bar{h}} \d x \d t
=\varepsilon^{-1} \int^\varepsilon_0 \int_K u(x,t) \phi_{\bar{h}}(0) \d x \d t + \varepsilon^{-1} \int^\varepsilon_0 \int_K u(x,t)[ \phi_{\bar{h}}(t)- \phi_{\bar{h}}(0) ]\d x \d t.
\end{align*}
The second term on the right-hand side converges to zero since $\phi_{\bar{h}}$ is uniformly continuous and $\norm{u(t)}_{L^{\beta+1}(K)}$ is bounded independent of $t$. The first term on the right-hand side converges to the second integral on the left-hand side of \eqref{h-averaged-form} since $u\in C([0,T]; L^{\beta+1}(K))$ and $\phi_{\bar{h}}(0)\in L^\frac{\beta+1}{\beta}(\Omega)$.
\end{proof}

\section{Energy Estimates}\label{sec:energy}
Here we discuss various energy estimates. We begin by showing that the assumptions on $u$ made in Definition \ref{weakdef} allow suitable choices of test functions in the mollified weak formulation. This is a crucial step in obtaining a rigourous proof for the energy estimates.
 
We want to use test functions involving $(u^\beta-k^\beta)_\pm$ for some $k\geq0$. Since these functions have a $p$-integrable gradient, they automatically fit with the elliptic term in \eqref{h-averaged-form}. The minimal integrability of $u$ which justifies the test function becomes apparent from the diffusive part of the mollified weak formulation: If $u\in L^q$ then $\partial_t u_h\in L^q$ and $(u^\beta-k^\beta)_\pm \in L^\frac{q}{\beta}$. These exponents should be at least dual exponents so we need
\begin{align*}
\frac{1}{q}+\frac{1}{q/\beta}\leq 1,
\end{align*}
which is equivalent to $q\geq \beta +1$. This is exactly the integrability we required in Definition \ref{weakdef}.

We now show the energy estimate for solutions according to Definition \ref{weakdef}.
\begin{lem}\label{lem:Energy_Est}
Let $u$ be a weak solution in the sense of Definition \ref{weakdef}. Then
\begin{align}\label{caccioppoli}
\iint_{\Omega_T} |\nabla (u^\beta-k^\beta)_\pm|^p\varphi^p\d x \d t + \esssup_{\tau\in [0,T]}\int_\Omega \b[u,k] \chi_{\{(u-k)_\pm > 0\}}\varphi^p(x,\tau)\d x
\\
\notag \leq C\iint_{\Omega_T}(u^\beta-k^\beta)^p_\pm |\nabla \varphi|^p\d x\d t + \iint_{\Omega_T} \b[u,k] \chi_{\{(u-k)_\pm > 0\}} |\partial_t \varphi^p|\d x \d t,
\end{align}
for all smooth $\varphi\geq 0$ defined on $\bar{\Omega}_T$, vanishing for $x$ outside a compact $K\subset \Omega$ and for all times less than some $\delta>0$. The constant $C$ only depends on the data.
\end{lem}
\begin{proof}[Proof]
We prove the case for the positive part. The case for the negative part is similar. We use the mollified weak formulation \eqref{h-averaged-form} with the test function

\noindent $\phi = (u^\beta-k^\beta)_+\varphi^p\xi_{\tau,\varepsilon}$ where $\varphi$ is as in the statement of the lemma and $\xi_{\tau,\varepsilon}$ is defined as
\begin{align}\label{def:xi}
\xi_{\tau,\varepsilon}(t):=\begin{cases}
\, 1, &t<\tau
\\
\, 1-\varepsilon^{-1}(\tau-t), &t\in [\tau, \tau+\varepsilon]
\\
\, 0, &t>\tau+\varepsilon.
\end{cases}
\end{align}
Even though $\phi$ is nonsmooth, it is still an admisssible test function since we can find a sequence of functions $\phi_j\in C^\infty_0(\Omega_T)$ converging to $\phi$ in $L^p(0,T;W^{1,p}_0(\Omega))\cap L^\frac{\beta+1}{\beta}(\Omega_T)$. Our goal is to make some estimates in \eqref{h-averaged-form} and pass to the limit $h\to 0$ and then $\varepsilon \to 0$. We first show that the term involving the initial value vanishes in this process. Taking into account the support of $\phi$ we have
\begin{align*}
\int_\Omega u(x,0) \phi_{\bar{h}}(x,0) \d x &= \iint_{\Omega_T} u(x,0) h^{-1}e^{-\frac{t}{h}}\phi(x,t)\d x\d t = \int^T_\delta \int_\Omega u(x,0) h^{-1}e^{-\frac{t}{h}}\phi(x,t)\d x\d t
\\
&\leq \int^T_\delta \int_\Omega u(x,0) \delta^{-1} (\tfrac{\delta}{h} e^{-\frac{\delta}{h}})\phi(x,t)\d x\d t\xrightarrow[h\to 0]{} 0,
\end{align*}
due to the dominated convergence theorem. The elliptic term can be treated using Lemma \ref{expmolproperties} (i) as
\begin{align*}
\iint_{\Omega_T} [A(x,\cdot,u, \nabla u^\beta)]_h\cdot \nabla \phi \d x \d t &\xrightarrow[h\to 0]{} \iint_{\Omega_T} A(x,t,u, \nabla u^\beta)\cdot \nabla \phi \d x \d t
\\
&\xrightarrow[\varepsilon\to 0]{} \iint_{\Omega_\tau} A(x,t,u, \nabla u^\beta)\cdot \nabla [(u^\beta-k^\beta)_+\varphi^p] \d x \d t
\end{align*}
We now calculate
\begin{align*}
\nabla \phi = \varphi^p\xi_{\tau,\varepsilon}\chi_{\{u>k\}}\nabla u^\beta + p(u^\beta-k^\beta)_+\xi_{\tau,\varepsilon}\varphi^{p-1}\nabla \varphi.
\end{align*}
From the properties of the vector field, here denoted only $A(u,\nabla u^\beta)$ for brevity, and Young's inequality we obtain
\begin{align*}
A(u,\nabla u^\beta) \cdot \nabla [(u^\beta-k^\beta)_+\varphi^p] &= A(u,\nabla u^\beta) \cdot \nabla u^\beta \chi_{\{u>k\}} \varphi^p + A(u,\nabla u^\beta) \cdot \nabla (\varphi^p)(u^\beta-k^\beta)_+
\\
&\geq c|\nabla u^\beta|^p \chi_{\{u>k\}} \varphi^p -|A(u,\nabla u^\beta)||\nabla \varphi|p\varphi^{p-1}(u^\beta-k^\beta)_+
\\
&\geq c|\nabla u^\beta|^p \chi_{\{u>k\}} \varphi^p -c|\nabla u^\beta|^{p-1}\varphi^{p-1}(u^\beta-k^\beta)_+|\nabla \varphi|
\\
&\geq c|\nabla u^\beta|^p \chi_{\{u>k\}} \varphi^p- c(u^\beta-k^\beta)_+^p|\nabla \varphi|^p.
\end{align*}
Using Lemma \ref{expmolproperties} (ii) and the fact that $s\mapsto (s^\beta-k^\beta)_+$ is increasing we can treat the diffusion term as
\begin{align*}
\partial_t u_h \phi &= \big(\frac{u-u_h}{h}\big)[(u^\beta-k^\beta)_+ - ([u_h]^\beta-k^\beta)_+]\varphi^p\xi_{\tau,\varepsilon} + \partial_t u_h([u_h]^\beta-k^\beta)_+\varphi^p\xi_{\tau,\varepsilon}
\\
&\geq \partial_t G(u_h) \varphi^p \xi_{\tau,\varepsilon},
\end{align*}
where
\begin{align}\label{G-def}
G(u):= \int^u_0 (s^\beta-k^\beta)_+ \d s= \b[u,k]\chi_{\{u>k\}}.
\end{align}
The chain rule works in our case since Lemma \ref{expmolproperties} guarantees that both $u_h$ and $\partial_t u_h$ are in $L^{\beta+1}(\Omega_T)$. Thus, we may estimate
\begin{align*}
\iint_{\Omega_T} \partial_t u_h \phi\d x \d t &\geq \iint_{\Omega_T} \partial_t G(u_h) \varphi^p \xi_{\tau,\varepsilon} \d x \d t= - \iint_{\Omega_T} G(u_h)\partial_t(\varphi^p\xi_{\tau,\varepsilon}) \d x \d t
\\
&\xrightarrow[h\to 0]{} - \iint_{\Omega_T} G(u)\partial_t(\varphi^p\xi_{\tau,\varepsilon}) \d x \d t
\\
&=- \iint_{\Omega_T} G(u)\partial_t\varphi^p\xi_{\tau,\varepsilon} \d x \d t + \varepsilon^{-1}\int^{\tau+\varepsilon}_\tau \int_\Omega G(u) \varphi^p\d x \d t
\\
&\xrightarrow[\varepsilon\to 0]{} - \iint_{\Omega_\tau} G(u)\partial_t\varphi^p\d x \d t + \int_\Omega G(u) \varphi^p(x,\tau)\d x,
\end{align*}
for a.e. $\tau$. Putting together the estimates for the elliptic and diffusion terms we have
\begin{align*}
c\iint_{\Omega_\tau}|\nabla u^\beta|^p \chi_{\{u>k\}} \varphi^p \d x \d t + \int_\Omega G(u) \varphi^p(x,\tau)\d x &\leq c \iint_{\Omega_\tau} (u^\beta-k^\beta)_+^p|\nabla \varphi|^p \d x \d t
\\
&\quad + \iint_{\Omega_\tau} G(u)|\partial_t\varphi^p|\d x \d t,
\end{align*}
for a.e. $\tau$. We obtain the desired estimate by using \eqref{G-def} and noting that the right-hand side can be estimated upwards by replacing $\tau$ by $T$.
\end{proof}
The following variant of the energy estimate will also be useful.
\begin{lem}\label{lem:alt-en-est}
Let $\varphi \in C^\infty_0(\Omega;\R_{\geq 0})$ and suppose that $[t_1,t_2]\subset (0,T)$. Then the time-continuous representative of $u$ satisfies
\begin{align}\label{klj}
c^{-1}\int^{t_2}_{t_1}\int_\Omega& |\nabla (u^\beta-k^\beta)_-|^p\d x \d t + \int_\Omega \b[u,k]\chi_{\{u<k\}}\varphi^p(x,t_2)\d x \d t
\\
\notag &\leq c \int^{t_2}_{t_1}\int_\Omega (u^\beta-k^\beta)_-^p|\nabla \varphi|^p\d x \d t+ \int_\Omega \b[u,k]\chi_{\{u<k\}}\varphi^p(x,t_1)\d x,
\end{align}
where $c>0$ is a constant depending only on $p,C_0,C_1$.
\end{lem}
\begin{proof}[Proof]
We use the mollified weak formulation \eqref{h-averaged-form} with the test function
\noindent $\phi=-(u^\beta-k^\beta)_-\varphi^p(x)\xi_\varepsilon(t)$, where
\begin{align*}
\xi_\varepsilon(t)=\begin{cases}
0, &t\leq t_1,
\\
\varepsilon^{-1}(t-t_1), & t\in (t_1,t_1+\varepsilon),
\\
1, &t\in [t_1+\varepsilon,t_2],
\\
\varepsilon^{-1}(t_2+\varepsilon-t), &t\in (t_2,t_2+\varepsilon), t\in (t_2,t_2+\varepsilon),
\\
0, & t\geq t_2+\varepsilon.
\end{cases}
\end{align*}
Reasoning as in the proof of Lemma \ref{lem:Energy_Est} leads to \eqref{klj}.
\end{proof}

\section{$L^1$-Harnack inequality}
In order to obtain the reduction of the oscillation we will use the fact that weak solutions satisfy a local $L^1$-Harnack inequality. Such a result was already obtained in \cite[Theorem 5.1]{FoSoVe} in a quite general setting, allowing for all $m>0$ and also a source term satisfying certain structure conditions. However, the proofs were made under the assumption that $u$ itself has a gradient, whereas in our case we only know that $u^\beta$ has a gradient. It turns out that the same strategy as in \cite{FoSoVe} works also in our case with some modifications. In this section we present the full proof in the case $m>1$ and without a source term.
 
\begin{theo}[Harnack inequality]\label{harnack}
Let $u$ be a nonnegative weak solution to \eqref{general} where the vector field $A(x,t,u,\xi)$ satisfies the structure conditions \eqref{structcond1} and \eqref{structcond2}, and the parameters satisfy the conditions \eqref{parameter-range}. Then there exists a positive constant $\gamma$ depending only on $m,n,p,C_0,C_1$ such that for all cylinders $\bar{B}_{2\rho}(y)\times[s,t]\subset \Omega\times [0,T)$,
\begin{align*}
\esssup_{\tau\in [s,t]}\int_{B_\rho(y)}u(x,\tau)\d x \leq \gamma \essinf_{\tau\in [s,t]}\int_{B_{2\rho}(y)}u(x,\tau)\d x + \gamma \left(\frac{t-s}{\rho^\lambda}\right)^\frac{1}{3-m-p},
\end{align*}
where $\lambda= n(p+m-3)+p$.
\end{theo}
Note that $\lambda$ can have any sign. If we use the time continuous representative of $u$ we can replace the essential infimum and supremum by the actual infimum and supremum.
Before proceeding we note that by translation we may assume that $s=0$. All of the calculations will be performed under this assumption, and the time interval $[s,t]$ will henceforth be labelled $[0,\tau]$, where $\tau \in (0,T)$.
The first step of the argument is a lemma corresponding to \cite[Lemma 5.2]{FoSoVe}.
\begin{lem}\label{femtvaa}
Let $u$ be a weak solution, $\tau \in (0,T)$, $\sigma\in (0,1)$ and $B_\rho(x_o)\subset \Omega$. Then
\begin{align}\label{est:harnack-lemma}
\iint_{B_{\sigma\rho}(x_o)\times(0,\tau)}&|\nabla u^\beta|^p (u^\beta+\varepsilon^\beta)^{\frac{m+p-3}{\beta p} -1}t^\frac{1}{p} \d x \d t + \iint_{B_{\sigma\rho}(x_o)\times(0,\tau)} F_\varepsilon(u) t^{\frac{1}{p}-1}\d x \d t
\\
\notag &\leq \frac{c\rho}{(1-\sigma)^p} \Big(\frac{\tau}{\rho^\lambda}\Big)^\frac{1}{p}\Big[ \sup_{t\in[0,\tau]} \int_{B_\rho(x_o)} u(x,t) \d x + \varepsilon\rho^n \Big]^\frac{2p+m-3}{p},
\end{align}
where $\lambda= n(p+m-3)+p$, $\varepsilon = (\tfrac{\tau}{\rho^p})^\frac{1}{3-m-p}$ and $F_\varepsilon$ is defined in \eqref{def:F} below. The constant $c$ depends only on $m,n,p, C_0,C_1$.
\end{lem}
\begin{proof}[Proof]
Consider the mollified weak formulation \eqref{h-averaged-form} with the test function
\begin{align*}
\phi(x,t) = -(u^\beta+\varepsilon^\beta)^\frac{m+p-3}{\beta p} t^\frac{1}{p}\varphi^p(x)\xi_{\tau,\delta}(t),
\end{align*}
where $\varepsilon > 0$, $\xi_{\tau,\delta}$ is defined as in \eqref{def:xi} and $\varphi\in C^\infty_0(B_\rho(x_o);[0,1])$ satisfies $\varphi=1$ on $B_{\sigma\rho}(x_o)$. We may thus choose $\varphi$ such that
\begin{align*}
|\nabla \varphi|\leq 2(1-\sigma)^{-1}\rho^{-1}.
\end{align*}
We have
\begin{align*}
\nabla \phi = \tfrac{(3-m-p)}{\beta p}(u^\beta+\varepsilon^\beta)^{\frac{m+p-3}{\beta p} -1}t^\frac{1}{p}\varphi^p(x)\xi_{\tau,\delta}(t)\nabla u^\beta - (u^\beta+\varepsilon^\beta)^\frac{m+p-3}{\beta p} t^\frac{1}{p}\xi_{\tau,\delta}(t) \nabla \varphi^p(x).
\end{align*}
We see that
\begin{align*}
\iint_{\Omega_T}[A(x,\cdot,u,\nabla u^\beta)]_h \cdot \nabla \phi \d x \d t \xrightarrow[h\to 0]{} \iint_{\Omega_T} A(x,t,u,\nabla u^\beta) \cdot \nabla \phi \d x \d t,
\end{align*}
and
\begin{align*}
A&(x,t,u,\nabla u^\beta) \cdot \nabla \phi
= \tfrac{(3-m-p)}{\beta p}(u^\beta+\varepsilon^\beta)^{\frac{m+p-3}{\beta p} -1}t^\frac{1}{p}\varphi^p \xi_{\tau,\delta}A(x,t,u,\nabla u^\beta)\cdot \nabla u^\beta
\\
&\hspace{31mm} - p(u^\beta+\varepsilon^\beta)^\frac{m+p-3}{\beta p} t^\frac{1}{p}\xi_{\tau,\delta}\varphi^{p-1} A(x,t,u,\nabla u^\beta)\cdot \nabla \varphi
\\
&\geq c_0|\nabla u^\beta|^p (u^\beta+\varepsilon^\beta)^{\frac{m+p-3}{\beta p} -1}t^\frac{1}{p}\varphi^p\xi_{\tau,\delta}
- c_1(u^\beta+\varepsilon^\beta)^\frac{m+p-3}{\beta p} t^\frac{1}{p}\xi_{\tau,\delta}\varphi^{p-1}|\nabla u^\beta|^{p-1}|\nabla \varphi|
\\
&\geq \hat{c}_0|\nabla u^\beta|^p (u^\beta+\varepsilon^\beta)^{\frac{m+p-3}{\beta p} -1}t^\frac{1}{p}\varphi^p\xi_{\tau,\delta} - \hat{c}_1 (u^\beta+\varepsilon^\beta)^{p-1+\frac{m+p-3}{\beta p}}t^\frac{1}{p}\xi_{\tau,\delta}|\nabla \varphi|^p.
\end{align*}
Here $c_0,c_1,\hat{c}_0,\hat{c}_1$ are constants depending only on $m,p,C_0,C_1$. For the initial value term we note that
\begin{align*}
\Big| \int_\Omega u(x,0) \phi_{\bar{h}}(x,0) \d x \Big| &= \Big| \iint_{\Omega_T} u(x,0) h^{-1}e^{-\frac{t}{h}}\phi(x,t)\d x\d t \Big|
\\
&\leq c \iint_{\supp \varphi \times [0,T]} u(x,0) (\tfrac{t}{h} e^{-\frac{t}{h}}) t^{\frac{1}{p}-1}\d x\d t \xrightarrow[h\to 0]{} 0,
\end{align*}
by the dominated convergence theorem. The diffusion part is treated as follows:
\begin{align*}
\phi\partial_t u_h &=
\big( \big[([u_h]^\beta+\varepsilon^\beta)^\frac{m+p-3}{\beta p} - (u^\beta+\varepsilon^\beta)^\frac{m+p-3}{\beta p}\big]\tfrac{(u-u_h)}{h}
-([u_h]^\beta+\varepsilon^\beta)^\frac{m+p-3}{\beta p}\partial_t u_h \big) t^\frac{1}{p}\varphi^p\xi_{\tau,\delta}
\\
&\geq -([u_h]^\beta+\varepsilon^\beta)^\frac{m+p-3}{\beta p}\partial_t u_h t^\frac{1}{p}\varphi^p\xi_{\tau,\delta}
\\
&= -\partial_t[ F(u_h)] t^\frac{1}{p}\varphi^p\xi_{\tau,\delta},
\end{align*}
where
\begin{align}\label{def:F}
F_\varepsilon(s):=\int^s_0 (t^\beta+\varepsilon^\beta)^\frac{m+p-3}{\beta p}\d t \leq \int^s_0 t^\frac{m+p-3}{p}\d t = \tfrac{p}{2p+m-3} s^\frac{2p+m-3}{p}.
\end{align}
From this we see that
\begin{align*}
\iint_{\Omega_T} \phi\partial_t u_h\phi \d x \d t \geq &\iint_{\Omega_T} F_\varepsilon(u_h)\partial_t( t^\frac{1}{p}\varphi^p\xi_{\tau,\delta}) \d x \d t
\\
\xrightarrow[h\to 0]{} &\iint_{\Omega_T} F_\varepsilon(u)\partial_t( t^\frac{1}{p}\varphi^p\xi_{\tau,\delta}) \d x \d t
\\
=&\frac{1}{p}\iint_{\Omega_T} F_\varepsilon(u)\varphi^p t^{\frac{1}{p}-1}\xi_{\tau,\delta}\d x \d t - \delta^{-1}\int^{\tau+\delta}_\tau \int_\Omega F_\varepsilon(u)\varphi^p t^\frac{1}{p} \d x \d t
\\
\xrightarrow[\delta\to 0]{}& \frac{1}{p}\iint_{\Omega_\tau} F_\varepsilon(u)\varphi^p t^{\frac{1}{p}-1}\d x \d t - \tau^\frac{1}{p}\int_\Omega F_\varepsilon(u)\varphi^p (x,\tau)\d x.
\end{align*}
To conclude the limit in the last term we use the Lipschitz continuity of $F$ and the time-continuity of $u$. Combining these estimates we have
\begin{align}\label{sutevensu}
&\iint_{\Omega_\tau}|\nabla u^\beta|^p (u^\beta+\varepsilon^\beta)^{\frac{m+p-3}{\beta p} -1}t^\frac{1}{p}\varphi^p \d x \d t + \iint_{\Omega_\tau} F_\varepsilon(u)\varphi^p t^{\frac{1}{p}-1}\d x \d t
\\
\notag &\quad \leq c\iint_{\Omega_\tau} (u^\beta+\varepsilon^\beta)^{p-1+\frac{m+p-3}{\beta p}}t^\frac{1}{p}|\nabla \varphi|^p\d x \d t + c\tau^\frac{1}{p}\int_\Omega F_\varepsilon(u)\varphi^p (x,\tau)\d x.
\end{align}
Taking into account the estimate in \eqref{def:F} and the support of $\varphi$, and applying H\"older's inequality we see that
\begin{align*}
\tau^\frac{1}{p}\int_\Omega F_\varepsilon(u)\varphi^p (x,\tau)\d x &\leq c \tau^\frac{1}{p}\int_{B_\rho(x_o)} u^\frac{2p+m-3}{p}(x,\tau) \varphi^p(x) \d x
\\
&\leq \tau^\frac{1}{p}\Big[ \int_{B_\rho(x_o)} u(x,\tau) \d x \Big]^\frac{2p+m-3}{p} |B_\rho(x_o)|^\frac{3-m-p}{p}
\\
&\leq c\tau^\frac{1}{p}\Big[ \sup_{t\in[0,\tau]}\int_{B_\rho(x_o)} u(x,t) \d x \Big]^\frac{2p+m-3}{p}\rho^\frac{n(3-m-p)}{p}
\\
&=c\rho\Big(\frac{\tau}{\rho^\lambda}\Big)^\frac{1}{p} \Big[ \sup_{t\in[0,\tau]}\int_{B_\rho(x_o)} u(x,t) \d x \Big]^\frac{2p+m-3}{p}.
\end{align*}
Using the bound on the gradient of $\varphi$ we may now estimate the other term on the right-hand side of \eqref{sutevensu} as
\begin{align}\label{gottagofast}
\iint_{\Omega_\tau} &(u^\beta+\varepsilon^\beta)^{p-1+\frac{m+p-3}{\beta p}}t^\frac{1}{p}|\nabla \varphi|^p\d x \d t
\\
\notag &\leq \tfrac{c}{(1-\sigma)^p\rho^p}\iint_{B_\rho(x_o)\times(0,\tau)} (u^\beta+\varepsilon^\beta)^\frac{m+p-3}{\beta} (u^\beta+\varepsilon^\beta)^\frac{2p+m-3}{\beta p} t^\frac{1}{p}\d x \d t
\\
\notag &\leq \tfrac{c}{(1-\sigma)^p\rho^p}\varepsilon^{m+p-3} \int^\tau_0\int_{B_\rho(x_o)} (u^\beta+\varepsilon^\beta)^\frac{2p+m-3}{\beta p} \d x\, t^\frac{1}{p} \d t
\\
\notag &\leq \tfrac{c}{(1-\sigma)^p\rho^p}\varepsilon^{m+p-3} \tau^{\frac{1}{p}+1} \sup_{t\in [0,\tau]} \int_{B_\rho(x_o)} (u^\beta+\varepsilon^\beta)^\frac{2p+m-3}{\beta p}(x,t) \d x ,
\end{align}
In the second step we use the fact that the exponent $\beta^{-1}(m+p-3)$ is negative. In the last step we estimate the integral over the ball by the supremum in time of such integrals, leaving only an integral in time of the factor $t^\frac{1}{p}$. The integral appearing in the last expression may be estimated further using H\"older's inequality and the definition of $\lambda$ as
\begin{align}\label{mellansteg}
\int_{B_\rho(x_o)} (u^\beta+\varepsilon^\beta)^\frac{2p+m-3}{\beta p} \d x &\leq \Big[ \int_{B_\rho(x_o)} (u^\beta+\varepsilon^\beta)^\frac{1}{\beta} \d x \Big]^\frac{2p+m-3}{p}|B_\rho(x_o)|^\frac{3-m-p}{p}
\\
\notag &\leq c \Big[ \int_{B_\rho(x_o)} u \d x + \varepsilon\rho^n \Big]^\frac{2p+m-3}{p}\rho^{1-\frac{\lambda}{p}}.
\end{align}
Since the exponent $p^{-1}(2p+m-3)$ is positive, we can combine \eqref{gottagofast} and \eqref{mellansteg} taking the supremum inside the square brackets to obtain
\begin{align*}
\iint_{\Omega_\tau} (u^\beta+\varepsilon^\beta)^{p-1+\frac{m+p-3}{\beta p}}t^\frac{1}{p}|\nabla \varphi|^p\d x \d t &\leq \frac{c\rho}{(1-\sigma)^p}\Big(\frac{\tau}{\rho^p}\Big)\varepsilon^{m+p-3} \Big(\frac{\tau}{\rho^\lambda}\Big)^\frac{1}{p}
\\
&\quad \times\Big[ \sup_{t\in[0,\tau]} \int_{B_\rho(x_o)} u(x,t) \d x + \varepsilon\rho^n \Big]^\frac{2p+m-3}{p}
\end{align*}
Combining the estimate for the two terms on the right-hand side of \eqref{sutevensu} we end up with
\begin{align}\label{almostthere}
\iint_{\Omega_\tau}&|\nabla u^\beta|^p (u^\beta+\varepsilon^\beta)^{\frac{m+p-3}{\beta p} -1}t^\frac{1}{p}\varphi^p \d x \d t + \iint_{\Omega_\tau} F_\varepsilon(u)\varphi^p t^{\frac{1}{p}-1}\d x \d t
\\
\notag &\leq \frac{c\rho}{(1-\sigma)^p} \Big(\frac{\tau}{\rho^\lambda}\Big)^\frac{1}{p}\Big[1+\varepsilon^{m+p-3}\Big(\frac{\tau}{\rho^p}\Big)\Big]\Big[ \sup_{t\in[0,\tau]} \int_{B_\rho(x_o)} u(x,t) \d x + \varepsilon\rho^n \Big]^\frac{2p+m-3}{p}.
\end{align}
Choosing now $\varepsilon=(\tfrac{\tau}{\rho^p})^\frac{1}{3-m-p}$ confirms \eqref{est:harnack-lemma}.
\end{proof}

Because of the somewhat more complicated calculations in our setting, we also need the following result, which does not appear in \cite{FoSoVe}.
\begin{lem}\label{lem:somestuff}
Let $F_\varepsilon$ be defined by \eqref{def:F} and let $\varepsilon>0$. Then there is a constant $c=c(m,p)$ such that
\begin{align}\label{lem:u-eps-est}
(u^\beta+\varepsilon^\beta)^{[\frac{3-m-p}{\beta p}+1](p-1)} \leq cF_\varepsilon(u)+ c \varepsilon^{\frac{m+p-3}{p}+1},
\end{align}
for all $u\geq 0$.
\end{lem}
\begin{proof}[Proof]
Assume first $u>2 \varepsilon$. Then since $m+p-3<0$,
\begin{align*}
F_\varepsilon(u)&=\int^u_0 (t^\beta+\varepsilon^\beta)^\frac{m+p-3}{\beta p} \d t \geq \int^u_\varepsilon (t^\beta+\varepsilon^\beta)^\frac{m+p-3}{\beta p} \d t\ \geq \int^u_\varepsilon (2t^\beta)^\frac{m+p-3}{\beta p} \d t
\\
&=c \big(u^{\frac{m+p-3}{p}+1}-\varepsilon^{\frac{m+p-3}{p}+1}\big) \geq \tilde{c}u^{\frac{m+p-3}{p}+1},
\end{align*}
where in the last step we used the assumption $u>2\varepsilon$ and the fact that the exponent of $\varepsilon$ is positive. On the other hand, since $u>2\varepsilon$ we also have
\begin{align*}
(u^\beta+\varepsilon^\beta)^{[\frac{3-m-p}{\beta p}+1](p-1)}\leq c u^{[\frac{3-m-p}{p}+\beta](p-1)}=cu^{\frac{m+p-3}{p}+1},
\end{align*}
and combining the two estimates we have verified the claim in the case $u>2\varepsilon$. Suppose now $u\leq 2\varepsilon$. Then
\begin{align*}
F_\varepsilon(u) &= \int^u_0 (t^\beta+\varepsilon^\beta)^\frac{m+p-3}{\beta p} \d t \geq \int^u_0 ((1+2^\beta)\varepsilon^\beta )^\frac{m+p-3}{\beta p} \d t = c \varepsilon^{\frac{m+p-3}{p}}u \geq c u^{\frac{m+p-3}{p}+1}
\\
&=c (u^\beta)^{[\frac{3-m-p}{\beta p}+1](p-1)} \geq c_1 (u^\beta+\varepsilon^\beta)^{[\frac{3-m-p}{\beta p}+1](p-1)} - c_2 \varepsilon^{\frac{m+p-3}{p}+1},
\end{align*}
where in the last step we used the fact that for positive $\alpha$ and nonnegative $a,b$ we have $a^\alpha\geq 2^{-\alpha}(a+b)^\alpha-b^\alpha$. Thus, we have verified the claim also in the case $u\leq 2\varepsilon$.
\end{proof}

The next lemma corresponds to \cite[Lemma 5.3]{FoSoVe}. A formal application of the chain rule shows that the the integrands on the left-hand side in both lemmas are essentially the same, although in our case the gradient of $u$ need not exist. The proof in our case is somewhat more complicated as we need also to use Lemma \ref{lem:somestuff}.
\begin{lem}\label{drphil} Let $u$ be a weak solution and $\delta\in (0,1)$. Then there is a constant $c$ depending only on $m,n,p,C_0,C_1$ such that
\begin{align*}
\frac{1}{\rho}\int^\tau_0 \int_{B_{\sigma\rho}(x_o)} |\nabla u^\beta|^{p-1} \d x \d t \leq \delta \sup_{t\in[0,\tau]} \int_{B_\rho(x_o)} u(x,t) \d x + \frac{c \delta^\frac{3-2p-m}{3-m-p}}{(1-\sigma)^\frac{p^2}{3-m-p}}\Big(\frac{\tau}{\rho^\lambda}\Big)^\frac{1}{3-m-p}.
\end{align*}
\end{lem}
\begin{proof}[Proof]
Choose $\varepsilon$ as in Lemma \ref{femtvaa}. By H\"older's inequality and the previous lemma, we have
\begin{align}\label{pppp}
&\int^\tau_0 \int_{B_{\sigma\rho}(x_o)} |\nabla u^\beta|^{p-1}\d x \d t
\\
\notag &= \int^\tau_0\int_{B_{\sigma\rho}(x_o)} \Big[ |\nabla u^\beta|^{p-1}(u^\beta+\varepsilon^\beta)^{[\frac{m+p-3}{\beta p}-1]\frac{(p-1)}{p}}t^\frac{p-1}{p^2}\Big]
\Big[ (u^\beta+\varepsilon^\beta)^{[\frac{3-m-p}{\beta p}+1]\frac{(p-1)}{p}} t^\frac{1-p}{p^2} \Big] \d x \d t
\\
\notag & \leq \Big[ \int^\tau_0\int_{B_{\sigma\rho}(x_o)} |\nabla u^\beta|^p (u^\beta+\varepsilon^\beta)^{\frac{m+p-3}{\beta p}-1} t^\frac{1}{p}\d x \d t \Big]^\frac{p-1}{p}
\\
\notag &\quad \times \Big[ \int^\tau_0\int_{B_{\sigma\rho}(x_o)} (u^\beta+\varepsilon^\beta)^{[\frac{3-m-p}{\beta p}+1](p-1)}t^\frac{1-p}{p}\d x \d t\Big]^\frac{1}{p}.
\end{align}
The second integral in the last expression can be estimated by combining \eqref{lem:u-eps-est} and \eqref{est:harnack-lemma}:
\begin{align*}
\int^\tau_0&\int_{B_{\sigma\rho}(x_o)} (u^\beta+\varepsilon^\beta)^{[\frac{3-m-p}{\beta p}+1](p-1)}t^\frac{1-p}{p}\d x \d t \leq c \int^\tau_0\int_{B_{\sigma\rho}(x_o)} (F_\varepsilon(u) + \varepsilon^{\frac{m+p-3}{p}+1})t^\frac{1-p}{p}\d x \d t
\\
&\leq c \int^\tau_0\int_{B_{\sigma\rho}(x_o)} F_\varepsilon(u)t^\frac{1-p}{p}\d x \d t + c \rho^n\varepsilon^{\frac{m+p-3}{p}+1}\tau^\frac{1}{p}
\\
&\leq \frac{c\rho}{(1-\sigma)^p} \Big(\frac{\tau}{\rho^\lambda}\Big)^\frac{1}{p}\Big[ \sup_{t\in[0,\tau]} \int_{B_\rho(x_o)} u(x,t) \d x + \varepsilon\rho^n \Big]^\frac{2p+m-3}{p} + c\rho \Big(\frac{\tau}{\rho^\lambda}\Big)^\frac{1}{p}(\varepsilon\rho^n)^\frac{2p+m-3}{p}
\\
&\leq \frac{c\rho}{(1-\sigma)^p} \Big(\frac{\tau}{\rho^\lambda}\Big)^\frac{1}{p}\Big[ \sup_{t\in[0,\tau]} \int_{B_\rho(x_o)} u(x,t) \d x + \varepsilon\rho^n \Big]^\frac{2p+m-3}{p}.
\end{align*}
Since also the other integral appearing in the last expression of \eqref{pppp} can be estimated using \eqref{est:harnack-lemma}, we have
\begin{align*}
\int^\tau_0 \int_{B_{\sigma\rho}(x_o)} |\nabla u^\beta|^{p-1}\d x \d t \leq \frac{c\rho}{(1-\sigma)^p} \Big(\frac{\tau}{\rho^\lambda}\Big)^\frac{1}{p}\Big[ \sup_{t\in[0,\tau]} \int_{B_\rho(x_o)} u(x,t) \d x + \varepsilon\rho^n \Big]^\frac{2p+m-3}{p}.
\end{align*}
Dividing by $\rho$ and applying Young's inequality to the right-hand side yields the claim.
\end{proof}

Now we can finally prove the Harnack inequality.
\begin{proof}[Proof of Theorem \ref{harnack}].
For $j\in \N$ we choose
\begin{align*}
\rho_j&:=2(1-2^{-j})\rho,\hspace{4mm} \tilde{\rho}_j:=\frac{1}{2}(\rho_j+\rho_{j+1})
\\
B_j&:=B_{\rho_j}(x_o),\hspace{9mm} \tilde{B}_j :=B_{\tilde{\rho}_j}(x_o)
\end{align*}
Pick $\zeta_j\in C^\infty_0(B_{\tilde{\rho}_j}(x_o);[0,1])$ such that $\zeta_j=1$ on $B_{\rho_j}(x_o)$ and We use the weak formulation \eqref{weakform2} with the test function $\varphi=\zeta_j\xi^r_{\tau_1,\tau_2}$ where $r>0$, $\tau_1<\tau_2<\tau$ and
\begin{align*}
\xi^r_{\tau_1,\tau_2}(t)=
\begin{cases}
0, &t<\tau_1,
\\
r^{-1}(t-\tau_1), &t\in [\tau_1,\tau_1+r]
\\
1, & t\in (\tau_1+r,\tau_2)
\\
r^{-1}(\tau_2+r-t), &t\in [\tau_2,\tau_2+r],
\\
0, &t>\tau_2+r.
\end{cases}
\end{align*}
This implies
\begin{align*}
\frac{1}{r}\int^{\tau_2}_{\tau_1} \int_\Omega u\zeta_j\d x\d t = \iint_{\Omega_T} A(u,\nabla u^\beta) \cdot \nabla \zeta_j \xi^r_{\tau_1,\tau_2}\d x \d t + \frac{1}{r}\int^{\tau_2}_{\tau_1} \int_\Omega u\zeta_j\d x\d t.
\end{align*}
Passing to the limit $r\to 0$ and using the structure conditions and properties of $\zeta_j$ we have
\begin{align}\label{psk}
\int_{B_j} u(x,\tau_1)\d x \leq \int_\Omega u\zeta_j(x,\tau_1) \d x &= \int^{\tau_2}_{\tau_1}\int_\Omega A(u,\nabla u^\beta) \cdot \nabla \zeta_j \d x \d t + \int_\Omega u\zeta_j(x,\tau_2)\d x
\\
\notag &\leq \int^{\tau_2}_{\tau_1}\int_\Omega |A(u,\nabla u^\beta)| |\nabla \zeta_j| \d x \d t + \int_\Omega u\zeta_j(x,\tau_2)\d x
\\
\notag &\leq c\frac{2^j}{\rho} \int^{\tau}_0\int_{\tilde{B}_j}|\nabla u^\beta|^{p-1}\d x \d t + \int_{B_{j+1}}u(x,\tau_2)\d x,
\end{align}
for all $\tau_1,\tau_2$ due to the time-continuity of $u$. Although we assumed $\tau_1<\tau_2$, we see by a similar calculation that the estimate remains valid for $\tau_1\geq \tau_2$. We want to estimate the double integral in the last expression using Lemma \ref{drphil} with $\rho$ replaced by $\rho_{j+1}$, and consequently with $\sigma:=\tilde{\rho}_j/\rho_{j+1}$. Directly from the definition it follows that
\begin{align*}
\frac{1}{1-\sigma}< 2^{j+2}.
\end{align*}
Taking this into account, Lemma \ref{drphil} shows that
\begin{align*}
\int_{B_j} u(x,\tau_1)\d x \leq c 2^j \delta \sup_{t\in[0,\tau]} \int_{B_{j+1}} u(x,t) \d x + c \frac{2^{\frac{jp^2}{3-m-p}+j}}{\delta^\frac{2p+m-3}{3-m-p}} \Big(\frac{\tau}{\rho^\lambda}\Big)^\frac{1}{3-m-p} + \int_{B_{j+1}}u(x,\tau_2)\d x,
\end{align*}
for all $\delta\in (0,1)$. Here we also used the fact that all the elements of the sequence $(\rho_j)$ are comparable in size to $\rho$. Taking now $\delta=c^{-1}2^{-1}\varepsilon_o$ where $\varepsilon_o\in (0,1)$ and $c\geq 1$ is the constant from the previous estimate, we see that
\begin{align*}
\int_{B_j} u(x,\tau_1)\d x \leq \varepsilon_o \sup_{t\in[0,\tau]} \int_{B_{j+1}} u(x,t) \d x + cb^j \Big(\frac{\tau}{\rho^\lambda}\Big)^\frac{1}{3-m-p} + \int_{B_{2\rho}(x_o)}u(x,\tau_2)\d x,
\end{align*}
where $b=b(m,n,p, C_0, C_1)$ and $c=c(m,n,p, C_0, C_1, \varepsilon_o)$. We also used the fact that $B_{j+1}\subset B_{2\rho}(x_o)$. Recalling that the inequality holds for a.e. $\tau_1,\tau_2\in (0,\tau)$ we see that it implies
\begin{align*}
S_j\leq \varepsilon_o S_{j+1} + c b^j \Big(\frac{\tau}{\rho^\lambda}\Big)^\frac{1}{3-m-p} + I,
\end{align*}
where
\begin{align}\label{iterative_esssup_ineq}
S_j:= \sup_{t\in[0,\tau]} \int_{B_j} u(x,t) \d x, \hspace{5mm} I:= \inf_{t\in[0,\tau]} \int_{B_{2\rho}(x_o)}u(x,t)\d x.
\end{align}
Iterating \eqref{iterative_esssup_ineq} we have
\begin{align}\label{iterated}
\sup_{t\in[0,\tau]} \int_{B_\rho(x_o)} u(x,t) \d x = S_1 \leq \varepsilon_o^M S_{M+1}+ c b \Big(\frac{\tau}{\rho^\lambda}\Big)^\frac{1}{3-m-p} \sum^{M-1}_{j=0}(b\varepsilon_o)^j+ I\sum^{M-1}_{j=0}\varepsilon_o^j.
\end{align}
choose now for example $\varepsilon_o= \tfrac{1}{2b}$ so that both of the sums in \eqref{iterated} converge in the limit $M\to \infty$. Then, since
\begin{align*}
S_{M+1}\leq \sup_{t\in[0,\tau]} \int_{B_{2\rho}(x_o)} u(x,t) \d x,
\end{align*}
where the right-hand side finite due to the time-continuity of $u$, we see that we can pass to the limit $M\to \infty$ which yields the claim.
\end{proof}

\section{Expansion of Positivity}
In this section we show that weak solutions exhibit expansion of positivity. This type of result was already obtained in \cite{FoSoVe2}, but the calculations were made under the assumption that $u$ has a gradient, which is not necessarily true in our case. We demonstrate that the same strategy as in \cite{FoSoVe2} can nevertheless be applied with some modifications. For the reader's convenience detailed proofs are provided. We start with a lemma corresponding to Lemma 3.1 of \cite{FoSoVe}.
\begin{lem}[General De Giorgi type lemma]\label{lem:GenDeGio}
Suppose that $v:\Omega_T\to \R_{\geq 0}$ satisfies $v^\beta \in L^p(0,T;W^{1,p}(\Omega))$ and the energy estimate
\begin{align}\label{general_en-est}
c_g\iint_{\Omega_T} | \nabla & v^\beta |^p\varphi^p \chi_{\{v<k\}}\d x\d t + c_e\esssup_{t\in [0,T]} \int_\Omega \b[v,k]\chi_{\{v<k\}}\varphi^p(x,t) \d x
\\
\notag & \leq \iint_{\Omega_T} |\nabla \varphi|^p(v^\beta-k^\beta)^p_- + \big(v^{\frac{\beta+1}{2}}-k^{\frac{\beta+1}{2}} \big)_-^2 \varphi^{p-1}|\partial_t \varphi | \d x \d t,
\end{align}
for some positive constants $c_g$ and $c_e$ and all $k\geq 0$ and functions $\varphi \in C^\infty(\bar{\Omega}\times [0,T];[0,1])$ vanishing in a neighborhood of $\partial_p \Omega_T$. Suppose $K>0$, $a\in (0,1)$ and that $Q_{\rho,\theta\rho^p}(z_o)\subset \Omega_T$ Then there is a constant $c>0$ depending only on $m,n,p$ such that if
\begin{align}\label{meas_cond_for_v}
|Q_{\rho,\theta\rho^p}(z_o)\cap \{ v < K\}|\leq c c_e c_g^{\frac{n}{p}} (1-a^\beta)^{n+2} \frac{ (\theta K^{m+p-3})^\frac{n}{p}}{\big[ 1 + \theta K^{m+p-3}\big]^{\frac{(n+p)}{p}}} |Q_{\rho,\theta\rho^p}(z_o)|,
\end{align}
then $v\geq a K$ a.e. in $Q_{\frac{\rho}{2},\theta(\frac{\rho}{2})^p}(z_o)$.
\end{lem}
\begin{proof}[Proof]
Define
\begin{align*}
&\rho_j :=\frac{\rho}{2}+\frac{\rho}{2^{j+1}}, \hspace{5mm} k_j^\beta := \Big(a^\beta+\frac{(1-a^\beta)}{2^j}\Big)K^\beta, \hspace{5mm} B_j:=B_{\rho_j}(y_o), \hspace{5mm} T_j:=(t_o-\theta \rho_j^p,t_o),
\\
&Q_j:= B_j\times T_j = Q_{\rho_j,\theta \rho_j^p}(y_o,t_o), \hspace{5mm} A_j:=Q_j\cap \{ v<k_j\}, \hspace{5mm} Y_j:=|A_j|/|Q_j|.
\end{align*}
Pick $\varphi_j \in C^\infty(Q_j;[0,1])$ such that $\varphi_j=1$ on $Q_{j+1}$ and $\varphi_j=0$ in a neighborhood of $\partial_p Q_j$, and
\begin{align*}
|\nabla \varphi_j|\leq 2^{j+3}\rho^{-1}, \hspace{5mm} |\partial_t\varphi|\leq c_p\theta^{-1}2^j\rho^{-p}.
\end{align*}
In the set where $v<k_{j+1}$ we have
\begin{align*}
(v^\beta-k_j^\beta)_-\geq k_j^\beta-k_{j+1}^\beta = \frac{(1-a^\beta)}{2^{j+1}}K^\beta,
\end{align*}
so
\begin{align}\label{ffff}
\frac{(1-a^\beta)^p}{2^{(j+1)p}}K^{\beta p}|A_{j+1}| &\leq \iint_{A_{j+1}} (v^\beta-k_j^\beta)_-^p\d x \d t
\\
\notag &\leq \Big( \iint_{A_{j+1}} (v^\beta-k_j^\beta)_-^{p \frac{(n+p)}{n}} \d x \d t \Big)^\frac{n}{n+p} |A_{j+1}|^\frac{p}{n+p}.
\end{align}
We treat the integral inside the brackets by applying H\"older's inequality to the integral over the space variables. One of the resulting integrals is then estimated by taking the essential supremum over the time interval, and the Gagliardo-Nirenberg inequality provides an upper bound for the other integral. All in all, we have
\begin{align*}
\iint_{A_{j+1}} & (v^\beta-k_j^\beta)_-^{p \frac{(n+p)}{n}} \d x \d t = \int_{T_{j+1}}\int_{B_{j+1}} (v^\beta-k_j^\beta)_-^{p \frac{p}{n}}\chi_{A_{j+1}} (v^\beta-k_j^\beta)_-^p \d x \d t
\\
&\leq \int_{T_{j+1}}\Big[\int_{B_{j+1}} (v^\beta-k_j^\beta)_-^p \chi_{A_{j+1}}\d x \Big]^\frac{p}{n} \Big[\int_{B_{j+1}} (v^\beta-k_j^\beta)_-^{p^*} \d x \Big]^\frac{p}{p^*} \d t
\\
&\leq \Big[\esssup_{T_{j+1}} \int_{B_{j+1}} (v^\beta-k_j^\beta)_-^p \chi_{A_{j+1}}\d x \Big]^\frac{p}{n} \int_{T_j} \Big[\int_{B_j} \big((v^\beta-k_j^\beta)_-\varphi_j\big)^{p^*} \d x \Big]^\frac{p}{p^*} \d t
\\
&\leq c\Big[\esssup_{T_{j+1}} \int_{B_{j+1}} (v^\beta-k_j^\beta)_-^{p-2} \chi_{A_{j+1}}(v^\beta-k_j^\beta)_-^2\d x \Big]^\frac{p}{n} \iint_{Q_j} |\nabla \big((v^\beta-k_j^\beta)_-\varphi_j\big)|^p \d x \d t
\\
&\leq c (1-a^\beta)^{(p-2)\frac{p}{n}} K^{\beta(p-2)\frac{p}{n}}2^{j(2-p)\frac{p}{n}} k_j^{(\beta-1)\frac{p}{n}} \Big[\esssup_{T_{j+1}} \int_{B_{j+1}} k_j^{1-\beta}(v^\beta-k_j^\beta)_-^2\d x \Big]^\frac{p}{n}
\\
&\quad \times \iint_{Q_j} |\nabla \big((v^\beta-k_j^\beta)_-\varphi_j\big)|^p \d x \d t
\\
&\leq c(1-a^\beta)^{(p-2)\frac{p}{n}} K^{\frac{p}{n}(m+p-3)}2^{j(2-p)\frac{p}{n}} \Big[\esssup_{T_j} \int_{B_j} \b[v,k_j]\chi_{\{v<k_j\}} \varphi_j^p \d x \Big]^\frac{p}{n}
\\
&\quad \times \iint_{Q_j} |\nabla \big((v^\beta-k_j^\beta)_-\varphi_j\big)|^p \d x \d t
\\
&\leq c c_e^{-\frac{p}{n}} c_g^{-1}(1-a^\beta)^{(p-2)\frac{p}{n}} K^{\frac{p}{n}(m+p-3)}2^{j(2-p)\frac{p}{n}}
\\
&\quad \times \Big[\iint_{Q_j} |\nabla \varphi_j|^p (v^\beta-k_j^\beta)^p_- + \big(v^{\frac{\beta+1}{2}}-k_j^{\frac{\beta+1}{2}} \big)_-^2 \varphi_j^{p-1}|\partial_\tau \varphi_j| \d x \d t\Big]^\frac{p+n}{n},
\end{align*}
where $c=c(m,n,p)$. We have also used \eqref{estimates:boundary_terms} (ii) and the fact that $k_j\leq K$. In the last step we use \eqref{general_en-est}. Taking also into account the bounds on the derivatives of $\varphi_j$ and the bound for $k_j$ we end up with
\begin{align*}
\Big[ \iint_{A_{j+1}} (v^\beta-k_j^\beta)_-^{p \frac{(n+p)}{n}} \d x \d t \Big]^\frac{n}{n+p} &\leq c c_e^{-\frac{p}{n+p}}c_g^{-\frac{n}{n+p}}(1-a^\beta)^{(p-2)\frac{p}{n+p}}K^\frac{p(m+p-3)}{n+p}2^{j\frac{p(2-p)}{n+p}}
\\
&\quad \times \big[K^{\beta p} 2^{j p}\rho^{-p} + K^{\beta+1}\theta^{-1}2^{-j}\rho^{-p}\big]|A_j|
\\
&\leq c c_e^{-\frac{p}{n+p}}c_g^{-\frac{n}{n+p}}(1-a^\beta)^{\frac{(p-2)p}{n+p}} K^{\frac{p(m+p-3)}{n+p} + \beta p}2^{j\frac{p(n+2)}{n+p}}\rho^{-p}
\\
&\quad \times \big[ 1 + K^{3-m-p}\theta^{-1}\big]|A_j|.
\end{align*}
Combining the last estimate with \eqref{ffff} we end up with
\begin{align*}
|A_{j+1}| &\leq c c_e^\frac{-p}{n+p}c_g^\frac{-n}{n+p}(1-a^\beta)^{\frac{(p-2)p}{n+p}-p} K^{\frac{p(m+p-3)}{n+p}} 2^{j[\frac{p(n+2)}{n+p}+p]}\rho^{-p} \big[ 1 + K^{3-m-p}\theta^{-1}\big]|A_j|^{1+\frac{p}{n+p}}
\end{align*}
Dividing by $|Q_j|$ gives us the desired iterative estimate
\begin{align*}
Y_{j+1}\leq c c_e^\frac{-p}{n+p}c_g^\frac{-n}{n+p}(1-a^\beta)^{\frac{(p-2)p}{n+p}-p} \big[\theta K^{m+p-3}\big]^\frac{p}{n+p} 2^{j[\frac{p(n+2)}{n+p}+p]} \big[ 1 + K^{3-m-p}\theta^{-1}\big]Y_j^{1+\frac{p}{n+p}}.
\end{align*}
Thus Lemma \ref{fastconvg} shows that if
\begin{align*}
Y_0 \leq c c_e c_g^{\frac{n}{p}} (1-a^\beta)^{n+2} \frac{ (\theta K^{m+p-3})^\frac{n}{p}}{\big[ 1 + \theta K^{m+p-3}\big]^{\frac{(n+p)}{p}}}
\end{align*}
for a suitable constant $c$ depending only on $m,n,p$, then $Y_j\to 0$, which means that $v\geq a K$ in $Q_{\frac{\rho}{2},\theta(\frac{\rho}{2})^p}(z_o)$.
\end{proof}
The following variant of the De Giorgi lemma will also be useful. The extra assumption \eqref{ularge}, regarding the values of $u$ at the initial time of the space-time cylinder, allows us to get a lower bound which holds on a cylinder which has only been reduced in the spatial dimensions. It is understood that we consider the time-continuous representative of $u$, so that \eqref{ularge} makes sense.
\begin{lem}[Variant of the general De Giorgi type lemma]\label{lem:VarGenDeGio}
Let $u$ be a weak solution in the sense of Definition \ref{weakdef}. Suppose that $Q_{\rho,\theta\rho^p}(z_o)\subset \Omega_T$ and that
\begin{align}\label{ularge}
u(x,t_o-\theta\rho^p)\geq K,
\end{align}
for a.e. $x\in B_\rho(x_o)$.
Then there is a constant $c$ depending only on $m,n,p,C_0,C_1$ such that if
\begin{align}\label{assumption:meas}
|Q_{\rho,\theta\rho^p}(z_o)\cap \{u< K\}| \leq c\frac{(1-a^\beta)^{n+2}}{\theta K^{m+p-3}}|Q_{\rho,\theta\rho^p}(z_o)|,
\end{align}
then $u\geq aK$ a.e. in $Q_{\frac{\rho}{2},\theta\rho^p}(z_o)$.
\end{lem}
\begin{proof}[Proof] Define $k_j$, $\rho_j$ and $B_j$ as in Lemma \ref{lem:GenDeGio}, but choose
\begin{align*}
Q_j:=B_j\times \Delta=B_j\times (t_o-\theta\rho^p,t_o)=Q_{\rho_j,\theta\rho^p}.
\end{align*}
As before, we denote $A_j=Q_j\cap \{u<k_j\}$ and $Y_j=|A_j|/|Q_j|$. Choose $\varphi_j\in C^\infty_0(B_j;[0,1])$ such that $\varphi_j=1$ on $B_{j+1}$ and
\begin{align*}
|\nabla \varphi_j|\leq \rho^{-1}2^{j+3}.
\end{align*}
We use the energy estimate \eqref{klj} of Lemma \ref{lem:alt-en-est} with $\varphi=\varphi_j$, $k=k_j$, $t_1=t_o-\theta\rho^p$ and $t_2\in \Delta$. The assumption \eqref{ularge} guarantees that the second term on the right-hand side of \eqref{klj} vanishes and we end up with
\begin{align*}
\iint_{Q_j} |\nabla (u^\beta-k_j^\beta)_-|^p\varphi_j^p\d x \d t + \esssup_\Delta \int_{B_j}\b[u,k_j] \chi_{\{u<k_j \}} \varphi_j^p\d x
\\ \leq c \iint_{Q_j} (u^\beta-k_j^\beta)_-^p|\nabla \varphi_j|^p\d x \d t,
\end{align*}
where $c=c(p,C_0,C_1)$. As in Lemma \ref{lem:GenDeGio} we see that
\begin{align}\label{jli}
\frac{(1-a^\beta)^p}{2^{(j+1)p}}K^{\beta p}|A_{j+1}|
\leq \Big( \iint_{A_{j+1}} (u^\beta-k_j^\beta)_-^{p \frac{(n+p)}{n}} \d x \d t \Big)^\frac{n}{n+p} |A_{j+1}|^\frac{p}{n+p}.
\end{align}
Similarly as in the proof of Lemma \ref{lem:GenDeGio} we may estimate the integral inside the brackets as
\begin{align*}
\iint_{A_{j+1}}& (u^\beta-k_j^\beta)_-^{p \frac{(n+p)}{n}} \d x \d t
\\
&\leq \Big[\esssup_\Delta \int_{B_j} (u^\beta-k_j^\beta)_-^p \chi_{A_{j+1}}\varphi_j^p \d x \Big]^\frac{p}{n} \int_\Delta \Big[\int_{B_j} \big((u^\beta-k_j^\beta)_-\varphi_j\big)^{p^*} \d x \Big]^\frac{p}{p^*} \d t
\\
&\leq c (1-a^\beta)^{(p-2)\frac{p}{n}} K^{\beta(p-2)\frac{p}{n}}2^{j(2-p)\frac{p}{n}} k_j^{(\beta-1)\frac{p}{n}} \Big[\esssup_\Delta \int_{B_j} (u^\beta-k_j^\beta)_-^2 k_j^{1-\beta} \varphi_j^p \d x \Big]^\frac{p}{n}
\\
& \quad \times \iint_{Q_j} |\nabla \big( (u^\beta-k^\beta)_- \varphi_j\big)|^p\d x \d t
\\
& \leq c (1-a^\beta)^{(p-2)\frac{p}{n}} K^{(m+p-3)\frac{p}{n}}2^{j(2-p)\frac{p}{n}} \Big[ \iint_{Q_j} (u^\beta-k_j^\beta)_-^p|\nabla \varphi_j|^p\d x \d t \Big]^\frac{p+n}{n}
\\
& \leq c (1-a^\beta)^{(p-2)\frac{p}{n}} K^{(m+p-3)\frac{p}{n}}2^{j(2-p)\frac{p}{n}} \big( K^{\beta p} \rho^{-p}2^{jp}|A_j| \big)^\frac{p+n}{n},
\end{align*}
where the constant $c$ only depends on $m,n,p,C_0,C_1$. Combining this estimate with \eqref{jli} we have
\begin{align*}
|A_{j+1}|\leq c (1-a^\beta)^{\frac{(p-2)p}{n+p}-p} K^{\frac{p(m+p-3)}{n+p}} 2^{j[\frac{p(n+2)}{n+p}+p]}\rho^{-p}|A_j|^{1+\frac{p}{n+p}}
\end{align*}
Dividing by $|Q_j|$ we obtain
\begin{align*}
Y_{j+1} \leq c (1-a^\beta)^{\frac{(p-2)p}{n+p}-p} (\theta K^{m+p-3})^\frac{p}{n+p} 2^{j[\frac{p(n+2)}{n+p}+p]} Y_j^{1+\frac{p}{n+p}}.
\end{align*}
In light of Lemma \ref{fastconvg}, this means that there exists a constant $c=c(m,n,p,C_0,C_1)$ such that if
\begin{align*}
Y_0\leq c\frac{(1-a^\beta)^{n+2}}{\theta K^{m+p-3}},
\end{align*}
then $Y_j\to 0$.
\end{proof}

A version of the following result was proven in Lemma 1.1 of Chapter 4 of \cite{DiGiaVe} for the parabolic $p$-Laplace equation. We use the same strategy.
\begin{lem}\label{snails}
Let $u$ be a weak solution on $\Omega_T$. Suppose that $B_\rho(y)\times\{s\} \subset \Omega_T$ and that
\begin{align}\label{exposlemma-assumption1}
|B_\rho(y)\cap \{u(\cdot,s)\geq M\}| \geq \alpha |B_\rho(y)|.
\end{align}
Then there are $\delta=\delta(m, n, p, C_0, C_1, \alpha)$ and $\epsilon=\epsilon(\alpha)$ such that
\begin{align}\label{exposlemma-consequence}
|B_\rho(y)\cap \{u(\cdot,t)\geq \epsilon M\}|\geq \frac{1}{2}\alpha |B_\rho(y)|,
\end{align}
for all $t \in (s, \min\{T, s+\delta M^{3-m-p}\rho^p \})$.
\end{lem}
\begin{proof}[Proof] Let $\tau < \min\{T, s+\delta M^{3-m-p}\rho^p\}$, where $\delta$ is a positive number which is yet to be chosen and consider \eqref{klj} of Lemma \ref{lem:alt-en-est} with $t_1=s$, $t_2=\tau$ and $k=M$. Discarding the first term on the left-hand side, which is non-negative we end up with
\begin{align}\label{kultaa}
\int_\Omega \big[\b[u,M]\chi_{\{u<M\}} \varphi^p\big](x,\tau) \d x &\leq \int_\Omega \big[\b[u,M]\chi_{\{u<M\}} \varphi^p\big](x,s) \d x
\\
\notag &\quad + c\int_s^\tau\int_\Omega (u^\beta-M^\beta)_-^p |\nabla \varphi|^p\d x \d t,
\end{align}
where $\varphi\in C^\infty_0(\Omega;\R_{\geq 0})$ and the constant $c$ only depends on $p,C_0,C_1$. Taking $\sigma \in (0,1)$ and $\varphi \in C^\infty_0(B_\rho(y);[0,1])$ such that
$\varphi=1$ on $B_{(1-\sigma)\rho}(y)$ and $|\nabla \varphi|\leq \tfrac{2}{\sigma \rho}$, the estimate \eqref{kultaa} implies
\begin{align}\label{MOERKOE}
\int_{B_{(1-\sigma)\rho}(y)} \b[u,M]\chi_{\{u<M\}} (x,\tau) \d x &\leq \int_{B_\rho(y)} \b[u,M]\chi_{\{u<M\}} (x,s) \d x
\\
\notag &\quad + \frac{c}{\sigma^p\rho^p}\int^\tau_0\int_{B_\rho(y)}(u^\beta-M^\beta)_-^p \d x \d t.
\end{align}
From the properties of $\b$ it follows that when $u<M$ we have
\begin{align*}
\b[u,M]\leq \b[0,M]=\tfrac{\beta}{\beta+1}M^{\beta+1}.
\end{align*}
Using this result and the assumption \eqref{exposlemma-assumption1} we conclude that
\begin{align*}
\int_{B_\rho(y)} \b[u,M]\chi_{\{u<M\}} (x,s)\d x &\leq \tfrac{\beta}{\beta+1} M^{\beta+1}|B_\rho(y)\cap \{u(\cdot,s)<M\}|
\\
& \leq \tfrac{\beta}{\beta+1}M^{\beta+1}(1-\alpha)|B_\rho(y)|.
\end{align*}
Recall that $\tau\in (s,s+\delta M^{3-m-p}\rho^p)$ where $\delta$ is to be chosen so
\begin{align*}
\frac{c}{\sigma^p\rho^p}\int^\tau_0\int_{B_\rho(y)}(u^\beta-M^\beta)_-^p \d x \d t \leq \frac{c\delta}{\sigma^p}M^{\beta+1}|B_\rho(y)|.
\end{align*}
We estimate the term on the left-hand side of \eqref{MOERKOE} as
\begin{align*}
\int_{B_{(1-\sigma)\rho}(y)} \b[u,M]\chi_{\{u<M\}} (x,\tau) \d x &= \int_{B_\rho(y)} \b[u,M]\chi_{\{u<M\}} (x,\tau) \d x
\\
&\quad - \int_{B_\rho(y)\setminus B_{(1-\sigma)\rho}(y)} \b[u,M]\chi_{\{u<M\}} (x,\tau) \d x
\\
&\geq \int_{B_\rho(y)} \b[u,M]\chi_{\{u<M\}} (x,\tau) \d x - c_n^\beta\sigma M^{\beta+1}|B_\rho(y)|.
\end{align*}
Picking $\epsilon \in (0,1)$ we can estimate the last integral as
\begin{align*}
\int_{B_\rho(y)} \b[u,M]\chi_{\{u<M\}} (x,\tau) \d x &\geq \int_{B_\rho(y)\cap \{u<\epsilon M\}} \b[u,M] (x,\tau) \d x
\\
&\geq \b[\epsilon M, M] |B_\rho(y)\cap \{u(\cdot,\tau)<\epsilon M\}|
\\
&\geq \tfrac{\beta}{\beta+1}M^{\beta+1}(1-2\epsilon ) |B_\rho(y)\cap \{u(\cdot,\tau)<\epsilon M\}|.
\end{align*}
Combining all the estimates we have
\begin{align}\label{rrrrr}
|B_\rho(y)\cap \{u(\cdot,\tau)<\epsilon M\}|\leq \frac{|B_\rho(y)|}{(1-2\epsilon)}[\tilde{c}_n^\beta \sigma + (1-\alpha) + c\sigma^{-p}\delta].
\end{align}
where $\tilde{c}_n^\beta=\tilde{c}_n^\beta(\beta,n)$ and $c=c(m,n,p,C_0,C_1)$. Choose $\sigma=\sigma(\alpha,n,m,p)$ so small that $\tilde{c}_n^\beta\sigma < \alpha/8$. With this choice of $\sigma$, choose
\noindent$\delta=\delta(m,n,p,C_0,C_1,\alpha)$ so small that $c\sigma^{-p}\delta<\alpha/8$. Here $c$ denotes the constant in \eqref{rrrrr}. This leads to
\begin{align*}
|B_\rho(y)\cap \{u(\cdot,\tau)<\epsilon M\}|\leq \frac{|B_\rho(y)|}{(1-2\epsilon)}(1-3\alpha/4).
\end{align*}
From this it follows that \eqref{exposlemma-consequence} is true for any
\begin{align*}
0<\epsilon \leq \frac{\alpha}{4(2-\alpha)}.
\end{align*}
\end{proof}
 
We are now ready to prove the main result of this section.
\begin{theo}[Expansion of Positivity]\label{theo:exppos}
Suppose that $(x_o,s)\in \Omega_T$ and $u$ is a weak solution satisfying
\begin{align}\label{assumpt-exppos}
|B_\rho(x_o)\cap\{u(\cdot,s) \geq M\}|\geq \alpha |B_\rho(x_o)|,
\end{align}
for some $M>0$ and $\alpha \in (0,1)$. Then there exist $\varepsilon,\delta,\eta \in (0,1)$ depending only on $m,p,n,C_0,C_1,\alpha$ such that if $B_{16\rho}(x_o)\times (s,s+\delta M^{3-m-p}\rho^p) \subset \Omega_T$ then
\begin{align*}
u\geq \eta M \textrm{ in } B_{2\rho}(x_o)\times (s+(1-\varepsilon)\delta M^{3-m-p}\rho^p, s+\delta M^{3-m-p}\rho^p) .
\end{align*}
\end{theo}
\begin{proof}[Proof]
The proof is divided into several steps.

\noindent \textbf{Step 1: Change of variables, transformed equation and energy estimates.} Let $\delta=\delta(m,n,p,C_0,C_1,\alpha)\in (0,1)$ be the constant from Lemma \eqref{snails}. By translation we may assume that $(y,s)=(\bar{0},0)$. Furthermore, we assume that $B_{16\rho}(\bar{0})\times (0, \delta M^{3-m-p}\rho^p)\subset \Omega_T$, since otherwise there is nothing to prove. Introduce the new variables $(y,\tau)$ defined by the equations
\begin{align*}
y=\frac{x}{\rho}, \hspace{15mm} -e^{-\tau}= \frac{t-\delta M^{3-m-p}\rho^p}{\delta M^{3-m-p}\rho^p}.
\end{align*}
These coordinates transform the cylinder $B_{16\rho}(\bar{0})\times (0,\delta M^{3-m-p}\rho^p)$ into $B_{16}(\bar{0})\times (0,\infty)$, preserving the direction of time. Define the function $v: B_{16}(\bar{0})\times (0, \infty)\to \R$,
\begin{align*}
v(y,\tau)=\frac{e^\frac{\tau}{3-m-p}}{M}u(x,t)= \frac{e^\frac{\tau}{3-m-p}}{M} u\big(\rho y, \delta M^{3-m-p}\rho^p(1-e^{-\tau})\big).
\end{align*}
A routine calculation confirms that $v^\beta \in L^p(0,S; W^{1,p}(B_{16}(\bar{0})))$, for all $S>0$, and that $v$ is a weak solution to the equation
\begin{align*}
\partial_\tau v -\nabla\cdot \tilde{A}(y,\tau,v,\nabla v^\beta)=\tfrac{1}{3-m-p}v,
\end{align*}
where 
\begin{align*}
\tilde{A}(y,\tau,v,\xi)=\delta \rho^{p-1} \frac{e^{(\frac{m+p-2}{3-m-p})\tau}}{M^{m+p-2}} A\big(\rho y, \delta M^{3-m-p}\rho^p(1-e^{-\tau}), M e^{-\frac{\tau}{3-m-p}}v,\rho^{-1}M^\beta e^{-\frac{\beta \tau}{3-m-p}}\xi \big)
\end{align*} 
satisfies the structure conditions
\begin{align}
\label{tildestructcond0} \tilde{A}(y,\tau,v,\xi)\cdot \xi &\geq \delta C_0 |\xi|^p,
\\
\label{tildestructcond1} |\tilde{A}(y,\tau,v,\xi)| &\leq \delta C_1 |\xi|^{p-1},
\end{align}
where $C_0$ and $C_1$ are the constants appearing in the structure conditions \eqref{structcond1} and \eqref{structcond2}.
The time continuity of $u$ obtained in Subsection \ref{subsec:time-cont} implies that $v\in C([0,\infty);L^{\beta+1}_{\textrm{loc}}(B_{16}(\bar{0})) $. This allows us to reason as in the proof of Lemma \ref{lem:mollified}, to conclude that $v$ satisfies the mollified weak formulation
\begin{align}\label{molform_for_v}
\int^\infty_0\int_{B_{16}(\bar{0})} [\tilde{A}(y,\cdot,v, \nabla v^\beta)]_h\cdot \nabla \phi+\partial_\tau v_h \phi\d y\d \tau - \int_{B_{16}(\bar{0})} (v \phi_{\tilde{h}})(y,0) \d y
\\
\notag = \tfrac{1}{3-m-p}\int^\infty_0 \int_{B_{16}(\bar{0})} v_h\phi\d y\d \tau,
\end{align}
for all $\phi \in C^\infty_0(B_{16}(\bar{0})\times (0,\infty))$. The only difference is that we have replaced $\phi_{\bar{h}}$ by
\begin{align*}
\phi_{\tilde{h}}(y,\tau):=\frac{1}{h}\int^\infty_\tau e^\frac{\tau-s}{h}\phi(y,s)\d s,
\end{align*}
which in practice always can be written as a finite integral due to the support of $\phi$. This enables us to prove an energy estimate for $v$. Namely, we use \eqref{molform_for_v} with the test function $\phi = -(v^\beta-k^\beta)_-\varphi^p \xi_r(\tau)$, where $\varphi$ is a smooth function vanishing near $\partial_p (B_{16}(\bar{0})\times(0,\infty))$, and
\begin{align*}
\xi_r(\tau)= \begin{cases}
1, &\tau \leq \tilde{\tau},
\\
r^{-1}(\tilde{\tau}+r-t), &\tau \in [\tilde{\tau},\tilde{\tau}+r],
\\
0, &\tau>\tilde{\tau}+r.
\end{cases}
\end{align*}
Here $\tilde{\tau}>0$. We see that
\begin{align*}
\int^\infty_0\int_{B_{16}(\bar{0})} [\tilde{A}(y,\tau,v, \nabla v^\beta)]_h\cdot \nabla \phi\d y \d \tau \xrightarrow[h\to 0]{} \int^\infty_0\int_{B_{16}(\bar{0})} \tilde{A}(y,\tau,v, \nabla v^\beta) \cdot \nabla \phi\d y \d \tau
\\
\xrightarrow[r\to 0]{} -\int^{\tilde{\tau}}_0\int_{B_{16}(\bar{0})} \tilde{A}(y,\tau,v, \nabla v^\beta) \cdot \nabla [(v^\beta-k^\beta)_-\varphi^p ]\d y \d \tau
\\
= \int^{\tilde{\tau}}_0\int_{B_{16}(\bar{0})} \tilde{A}(y,\tau,v, \nabla v^\beta) \cdot \nabla v^\beta \chi_{\{v<k\}} \varphi^p - \tilde{A}(y,\tau,v, \nabla v^\beta) \cdot \nabla \varphi^p (v^\beta-k^\beta)_- \d y \d \tau
\\
\geq \int^{\tilde{\tau}}_0\int_{B_{16}(\bar{0})} \delta C_0|\nabla v^\beta|^p\varphi^p \chi_{\{v<k\}} - \delta \tilde{C}_1|\nabla v^\beta|^{p-1}\varphi^{p-1}|\nabla \varphi|(v^\beta-k^\beta)_- \d y \d \tau
\\
\geq \int^{\tilde{\tau}}_0\int_{B_{16}(\bar{0})} \delta \frac{C_0}{2}|\nabla v^\beta|^p\varphi^p \chi_{\{v<k\}} - \delta \hat{C}_1 |\nabla \varphi|^p(v^\beta-k^\beta)^p_- \d y \d \tau,
\end{align*}
where $\hat{C}_1=\hat{C}_1(C_1,p)$. Reasoning similarly as in the proof of Lemma \ref{lem:Energy_Est}, the parabolic term can be treated as
\begin{align*}
\int^\infty_0&\int_{B_{16}(\bar{0})}\partial_\tau v_h \phi\d y\d \tau \geq -\int^\infty_0\int_{B_{16}(\bar{0})} \b[v_h,k]\chi_{\{v_h<k\}} \partial_\tau( \varphi^p \xi_r )\d y\d \tau
\\
&\xrightarrow[h\to 0]{} -\int^\infty_0\int_{B_{16}(\bar{0})} \b[v,k]\chi_{\{v<k\}} \xi_r \partial_\tau \varphi^p\d y \d \tau + r^{-1} \int^{\tilde{\tau}+r}_{\tilde{\tau}} \int_{B_{16}(\bar{0})}\b[v,k]\chi_{\{v<k\}} \varphi^p\d y \d \tau.
\\
&\xrightarrow[r\to 0]{} -\int_0^{\tilde{\tau}}\int_{B_{16}(\bar{0})} \b[v,k]\chi_{\{v<k\}} \partial_\tau \varphi^p \d y \d \tau - \int_{B_{16}(\bar{0})}\b[v,k]\chi_{\{v<k\}} \varphi^p(y,\tilde{\tau}) \d y.
\end{align*}
As in the proof of Lemma \ref{lem:Energy_Est}, one can see that the second term on the left-hand side of \eqref{molform_for_v} vanishes in the limit $h\to 0$. Combining the estimates for all terms we end up with
\begin{align}
\notag \delta \frac{C_0}{2}& \int^{\tilde{\tau}}_0\int_{B_{16}(\bar{0})} |\nabla v^\beta|^p\varphi^p \chi_{\{v<k\}}\d y\d \tau + \int_{B_{16}(\bar{0})}\b[v,k]\chi_{\{v<k\}}\varphi^p(z,\tilde{\tau}) \d y
\\
\label{solstizio} &\leq \delta \hat{C}_1 \int^{\tilde{\tau}}_0\int_{B_{16}(\bar{0})} |\nabla \varphi|^p(v^\beta-k^\beta)^p_- \d y \d \tau + \int_0^{\tilde{\tau}}\int_{B_{16}(\bar{0})} \b[v,k]\chi_{\{v<k\}} \partial_\tau \varphi^p \d y \d \tau
\end{align}
Note that we were able to drop the term on the right-hand side of \eqref{molform_for_v} since it is non-positive. Using Lemma \ref{estimates:boundary_terms} (i) to estimate $\b[v,k]$ on the right-hand side and taking into account that $\delta \in (0,1)$ we finally obtain the desired energy estimate
\begin{align}\label{en_est_for_v}
\notag c\delta \int^{\tilde{\tau}}_0 &\int_{B_{16}(\bar{0})} |\nabla v^\beta|^p\varphi^p \chi_{\{v<k\}}\d y\d \tau + c\esssup_{\tau\in [0,\tilde{\tau}]} \int_{B_{16}(\bar{0})}\b[v,k]\chi_{\{v<k\}}\varphi^p(y,\tau) \d y
\\
& \leq \int^{\tilde{\tau}}_0\int_{B_{16}(\bar{0})} |\nabla \varphi|^p(v^\beta-k^\beta)^p_- + \big(v^{\frac{\beta+1}{2}}-k^{\frac{\beta+1}{2}} \big)_-^2 \varphi^{p-1} |\partial_\tau \varphi | \d y \d \tau,
\end{align}
where $c=c(C_0,C_1,p,m)$ and $\tilde{\tau}$ is any positive number.

\noindent \textbf{Step 2: Measure estimates of sublevel sets.} From the assumption \eqref{assumpt-exppos} and Lemma \ref{snails} it follows that there is an $\epsilon=\epsilon(\alpha)$ such that
\begin{align}\label{measure_est_for_v}
|B_1(\bar{0}) \cap \{v(\cdot,\tau)>\epsilon e^\frac{\tau}{3-m-p}\}|\geq \frac{\alpha}{2}|B_1(\bar{0})|,
\end{align}
for all $\tau\in [0,\infty)$. Pick $\tau_o>0$ to be determined later and define
\begin{align}\label{k_j-def}
k_o:=\epsilon e^\frac{\tau_o}{3-m-p}, \hspace{7mm} k_j:=\frac{k_o}{(2^\frac{1}{\beta})^j}, \hspace{5mm} j\in \N_0.
\end{align}
With these definitions, \eqref{measure_est_for_v} implies that
\begin{align}\label{est_on_B8}
|B_8(\bar{0}) \cap \{v(\cdot,\tau)> k_j\}|\geq \frac{\alpha}{2}8^{-n}|B_8(\bar{0})|,
\end{align}
for all $\tau\in [\tau_o,\infty)$ and $j\in \N_0$.
We introduce the cylinders
\begin{align*}
Q_{\tau_o} :=B_8(\bar{0})\times (\tau_o+k_o^{3-m-p}, \tau_o+2k_o^{3-m-p}),
\hspace{5mm}
Q_{\tau_o}' :=B_{16}(\bar{0})\times (\tau_o, \tau_o+2k_o^{3-m-p})
\end{align*}
Pick $\zeta_1\in C^\infty_0(B_{16}(\bar{0}))$ such that $\zeta_1=1$ on $B_8(\bar{0})$ and $|\nabla \zeta_1|\leq \frac{1}{4}$. Pick $\zeta_2 \in C^\infty(\R)$ such that $\zeta_2(\tau)=0$ for $\tau <\tau_o$, $\zeta_2(\tau)=1$ for $\tau \geq \tau_o + k_o^{3-m-p}$ and $0\leq \zeta_2'(\tau)\leq \frac{2}{k_o^{3-m-p}}$. Using the energy estimate \eqref{en_est_for_v} with $\varphi(y,\tau)=\zeta_1(y)\zeta_2(\tau)$, $k=k_j$ and $\tau_1 = \tau_o+2k_o^{3-m-p}$ yields
\begin{align}\label{midsommar}
\iint_{Q_{\tau_o}} |\nabla v^\beta|^p \chi_{\{v<k_j\}}\d y\d \tau &\leq \frac{c}{\delta} \iint_{Q_{\tau_o}'}(v^\beta-k_j^\beta)_-^p + \frac{2}{ k_o^{3-m-p}} \big(v^{\frac{\beta+1}{2}}-k_j^{\frac{\beta+1}{2}} \big)_-^2 \d y \d \tau
\\
\notag &\leq c\delta^{-1}\Big(k_j^{\beta p} + \frac{k_j^{\beta+1}}{k_o^{3-m-p}}\Big)|Q_{\tau_o}|
\\
\notag &\leq c \delta^{-1}k_j^{\beta p} |Q_{\tau_o}|,
\end{align}
where in the second step we used the fact that the measures of $Q_{\tau_o}$ and $Q_{\tau_o}'$ are comparable. In the last step we used that $k_j\leq k_o$. The constant $c$ still depends only on $C_0,C_1,p,m$. We define the sets
\begin{align*}
A_j:= Q_{\tau_o}\cap \{v<k_j\}, \hspace{5mm} A_j(\tau):=B_8(\bar{0}) \cap \{v(\cdot, \tau) <k_j\}.
\end{align*}
By the isoperimetric inequality \eqref{isoperim} and \eqref{est_on_B8} we have
\begin{align*}
\frac{k_j^\beta}{2}|A_{j+1}(\tau)|=(k_j^\beta-k_{j+1}^\beta)|A_{j+1}(\tau)|&\leq \frac{c_n}{|B_8(\bar{0})\setminus A_j(\tau)|}\int_{A_j(\tau)\setminus A_{j+1}(\tau)}|\nabla v^\beta(y,\tau)|\d y
\\
&\leq \frac{\tilde{c}_n}{\alpha}\int_{A_j(\tau)\setminus A_{j+1}(\tau)}|\nabla v^\beta(y,\tau)|\d y.
\end{align*}
Integrating the estimate over the time interval $(t_o+k_o^{3-m-p}, \tau_o+2k_o^{3-m-p})$ and using H\"older's inequality and \eqref{midsommar} we obtain
\begin{align*}
\frac{k_j^\beta}{2}|A_{j+1}| &\leq \frac{\tilde{c}_n}{\alpha}\int_{A_j\setminus A_{j+1}}|\nabla v^\beta(y,\tau)|\d y
\\
&\leq \frac{\tilde{c}_n}{\alpha}\Big[\int_{A_j\setminus A_{j+1}}|\nabla v^\beta(y,\tau)|^p\d y\Big]^\frac{1}{p} |A_j\setminus A_{j+1}|^{1-\frac{1}{p}}
\\
&\leq \frac{c k_j^\beta}{\alpha \delta^\frac{1}{p}}|Q_{\tau_o}|^\frac{1}{p}|A_j\setminus A_{j+1}|^\frac{p-1}{p},
\end{align*}
where $c$ depends on $m,n,p,C_0, C_1$. Hence,
\begin{align*}
|A_{j+1}|^\frac{p}{p-1}\leq \frac{c}{\alpha^\frac{p}{p-1}\delta^\frac{1}{p-1}}|Q_{\tau_o}|^\frac{1}{p-1}|A_j\setminus A_{j+1}|=\gamma |Q_{\tau_o}|^\frac{1}{p-1}|A_j\setminus A_{j+1}|,
\end{align*}
where $\gamma:=c\alpha^{-\frac{p}{p-1}}\delta^{-\frac{1}{p-1}}$. Adding this equation for $j\in \{0,\dots, j_0-1\}$ where $j_0\in \N$ and noting that $|A_j|$ is decreasing in $j$ we have
\begin{align*}
j_0|A_{j_0}|^\frac{p}{p-1}\leq \gamma |Q_{\tau_o}|^\frac{1}{p-1}|\Big(\sum_{j=0}^{j_0-1}(|A_j|-|A_{j+1}|)\Big)\leq \gamma |Q_{\tau_o}|^\frac{p}{p-1}|.
\end{align*}
Taking into account the definition of $A_j$, this means that
\begin{align*}
|Q_{\tau_o}\cap \{v<k_{j_0}\}|\leq \Big(\frac{\gamma}{j_0}\Big)^\frac{p-1}{p}|Q_{\tau_o}|.
\end{align*}
Recalling that $\delta$ is already determined in terms of $m,n,p,C_0,C_1,\alpha$, this estimate shows that any $\nu>0$ we may choose $j_0=j_0(m,n,p,C_0,C_1,\alpha,\nu)\in \N$ such that
\begin{align*}
|Q_{\tau_o}\cap \{v<k_{j_0}\}|\leq \nu |Q_{\tau_o}|.
\end{align*}
Let $j_*\in [j_0,\infty)$ be the smallest real number for which $(2^{j_*})^\frac{3-m-p}{\beta}$ is an integer. Then $j_*$ only depends on $m,n,p,C_0,C_1,\alpha,\nu$ and
\begin{align}\label{gggg}
|Q_{\tau_o}\cap \{v<k_{j_*}\}|\leq \nu |Q_{\tau_o}|,
\end{align}
where we have extended the definition of $k_j$ in \eqref{k_j-def} to all real numbers.

\noindent \textbf{Step 3: Segmenting the cylinder.}
For $i$ belonging to $\{0,\dots (2^{j_*})^\frac{3-m-p}{\beta} - 1\} $ We define the subcylinders
\begin{align*}
Q_i=B_8(\bar{0})\times \big(\tau_o +k_o^{3-m-p} + ik_{j_*}^{3-m-p}, \tau_o +k_o^{3-m-p} + (i+1)k_{j_*}^{3-m-p}\big),
\end{align*}
which is a parition of $Q_{\tau_o}$ (discarding only a set of measure zero). Thus, \eqref{gggg} implies that for at least one of the subcylinders we must have
\begin{align*}
|Q_i\cap \{v<k_{j_*}\}|\leq \nu |Q_i|.
\end{align*}
Since $v$ satisfies the energy estimates \eqref{en_est_for_v}, we may apply Lemma \ref{lem:GenDeGio} to $Q_i$ with $\rho=8$, $\theta=8^{-p}k_{j_*}^{3-m-p}$, $K=k_{j_*}$ and $a=\frac{1}{2}$. Now $c_g=c\delta$ for a $c$ only depending on $m,n,p,C_1,C_0$ and also $c_e$ only depends on these parameters. Plugging in everything into \eqref{meas_cond_for_v} we see that there is a constant $c$ depending only on $m,n,p,C_1,C_0$, such that if $\nu\leq \nu_o:= c\delta^\frac{n}{p}$ then
\begin{align}\label{nypotatis}
v\geq \frac{1}{2}k_{j_*} \textrm{ in } B_4(\bar{0})\times \big(\tau_o+k_o^{3-m-p} + (i+1-2^{-p})k_{j_*}^{3-m-p}, \tau_o+k_o^{3-m-p} + (i+1)k_{j_*}^{3-m-p}\big).
\end{align}
Fixing $j_0:=j_0(m,n,p,C_0,C_1,\alpha,\nu_o)$, we obtain by the definitions of $\nu_o$ and $\delta$ that the corresponding $j_*$ ultimately depends only on $m,n,p,C_0,C_1,\alpha$, and that \eqref{nypotatis} is indeed valid. Hence, there is a $\tau_1 \in (\tau_o+k_o^{3-m-p}, \tau_o+2k_o^{3-m-p})$ such that for a.e. $y\in B_4(\bar{0})$,
\begin{align}\label{doors}
v(y,\tau_1)\geq \frac{1}{2}k_{j_*}=\frac{k_o}{2^{\frac{j_*}{\beta}+1}} =\frac{\epsilon}{2^{\frac{j_*}{\beta}+1}} e^\frac{\tau_o}{3-m-p}=\sigma_o e^\frac{\tau_o}{3-m-p},
\end{align}
where $\sigma_o=\sigma_o(m,n,p,C_0,C_1,\alpha)$.

\noindent \textbf{Step 4: Returning to the original coordinates.}
By the definition of $v$, \eqref{doors} says that for a.e. $x\in B_{4\rho}(\bar{0})$
\begin{align*}
u(x,t_1)\geq \sigma_o M e^\frac{\tau_o-\tau_1}{3-m-p}=:M_o,
\end{align*}
where $t_1:=\delta M^{3-m-p}\rho^p(1-e^{-\tau_1})$. We want to apply Lemma \ref{lem:VarGenDeGio} with $K=M_o$, $a=\frac{1}{2}$ and $\theta = c 2^{-n-2}M_o^{3-m-p}$, where $c$ is the constant from the assumption \eqref{assumption:meas}. With these choices the assumption in Lemma \ref{lem:VarGenDeGio} is automatically true since it becomes the statement $|Q\cap \{u < M_o\}|\leq |Q|$ for a certain cylinder $Q$. As a consequence, Lemma \ref{lem:VarGenDeGio} implies that
\begin{align}\label{thunder}
u\geq \frac{1}{2}M_o,
\end{align}
in $B_{2\rho}(\bar{0})\times(t_1,t_1+ c 2^{-n-2}M_o^{3-m-p} (4\rho)^p)$. In order to complete the proof, it is sufficient that
\begin{align*}
t_1+ c 2^{-n-2}M_o^{3-m-p} (4\rho)^p) = \delta M^{3-m-p}\rho^p.
\end{align*}
Using the definition of $t_1$ we see that this is equivalent to
\begin{align*}
\tau_o= \ln\Big(\frac{2^{n+2}\delta}{c 4^p\sigma_o^{3-m-p}}\Big),
\end{align*}
where $c$ is the constant from assumption \eqref{assumption:meas}. The right hand side depends only on $m,n,p,C_0,C_1,\alpha$. Hence, with this choice of $\tau_o$, \eqref{thunder} and the upper bound for $\tau_1$ imply that
\begin{align*}
u\geq \frac{1}{2}M_o = \frac{\sigma_o}{2} e^\frac{\tau_o-\tau_1}{3-m-p} M > \frac{\sigma_o}{2} e^{-\frac{2k_o^{3-m-p}}{3-m-p}} M =: \eta M.
\end{align*}
in $B_{2\rho}(\bar{0})\times(t_1, \delta M^{3-m-p}\rho^p)$. Note that $\eta$ only depends on $m,n,p,C_0,C_1,\alpha$. From the upper bound for $\tau_1$ it also follows that
\begin{align*}
t_1 =\delta M^{3-m-p}\rho^p(1-e^{-\tau_1})< \delta M^{3-m-p}\rho^p(1-e^{-\tau_0-2k_o^{3-m-p}}),
\end{align*}
so the claim of the theorem is true if we take
\begin{align*}
\varepsilon= e^{-\tau_0-2k_o^{3-m-p}},
\end{align*}
and the right-hand side clearly only depends only on $m,n,p,C_0,C_1,\alpha$.
\end{proof}

\section{Local Boundedness}
We prove that in the range \eqref{nicerange} all weak solutions are locally bounded. We use a De Giorgi iteration combining the energy estimates obtained in Lemma \ref{lem:Energy_Est} with a Sobolev embedding.
\begin{theo}\label{theo:local_bdd}
Let $u$ be a weak solution in the sense of Definition \ref{weakdef} and suppose that the parameters $m$ and $p$ satisfy \eqref{nicerange}. Then $u$ is locally bounded and for any cylinder of the form $Q_{\rho,2\tau}(z_o)$ contained in $\Omega_T$ and any $\sigma \in (0,1)$ we have the explicit bound
\begin{align*}
\esssup_{Q_{\sigma \rho, \sigma\tau}(z_o)}u\leq c\Big[\big((1-\sigma)^p \tau\big)^{-\frac{n+p}{p}}\iint_{Q_{\rho, \tau}(z_o)}u^{\beta+1}\d x \d t\Big]^\frac{p}{p\beta n+(\beta+1)(p-n)} + \Big(\frac{\tau}{\rho^p}\Big)^\frac{1}{3-m-p},
\end{align*}
where $c$ is a constant depending only on $m,n,p, C_0, C_1$.
\end{theo}
\begin{proof}[Proof]
Suppose that $Q_{\rho,\tau}(z_o)\subset \Omega_T$. Define sequences
\begin{align*}
\rho_j:= \sigma \rho +\frac{(1-\sigma)}{2^j}\rho, \hspace{15mm} \tau_j&:= \sigma\tau +\frac{(1-\sigma)}{2^j}\tau,
\hspace{15mm} k_j:=k(1-2^{-j})^\frac{2}{\beta+1},
\end{align*}
where $k>0$ is a number to be fixed later. We also define the cylinders $Q_j:=Q_{\rho_j,\tau_j}(z_o)=B_j\times T_j$. Choose functions $\varphi_j\in C^\infty(Q_j;[0,1])$ vanishing near the parabolic boundary of $Q_j$ and satisfying $\phi_j=1$ on $Q_{j+1}$ and for which
\begin{align*}
|\nabla \varphi_j|\leq \frac{2^{j+2}}{(1-\sigma)\rho}, \hspace{10mm} |\partial_t \varphi_j|\leq \frac{2^{j+2}}{(1-\sigma)\tau}.
\end{align*}
Furthermore, we define the sequence
\begin{align*}
Y_j:=\iint_{Q_j}\big(u^\frac{\beta+1}{2}-k_j^\frac{\beta+1}{2}\big)_+^2\d x \d t.
\end{align*}
Note that $Y_j$ is finite for every $j$ since $u\in L^{\beta+1}(\Omega_T)$. Define the auxiliary parameters
\begin{align*}
M:=\tfrac{\beta+1}{\beta}, \hspace{10mm} q:= p(\tfrac{1}{M}+\tfrac{1}{n})=\tfrac{p}{M}(\tfrac{n+M}{n}).
\end{align*}
A straightforward calculation shows that \eqref{nicerange-rephrased} (and hence \eqref{nicerange}) guarantees that $q>1$. Thus, we may use H\"older's inequality to estimate
\begin{align}\label{Y-jplusone_estim1}
Y_{j+1}\leq \Big[\iint_{Q_{j+1}} \big(u^\frac{\beta+1}{2}-k_{j+1}^\frac{\beta+1}{2}\big)_+^{2q}\d x \d t\Big]^\frac{1}{q}|Q_{j+1}\cap \{u>k_{j+1}\}|^\frac{1}{q'}.
\end{align}
We will use the shorthand notation
\begin{align*}
\phi:= \big(u^\frac{\beta+1}{2}-k_{j+1}^\frac{\beta+1}{2}\big)_+^\frac{2}{M}\leq (u^\beta-k^\beta_{j+1})_+.
\end{align*}
The upper bound, is a consequence of the definition of $M$ and the fact that $\tfrac{2}{M}>1$.
In the following calculation we express the integral on the right-hand side of \eqref{Y-jplusone_estim1} in terms of $\phi$ and split the integral into space and time variables.
We apply H\"older's inequality to the integral over the space variables, and then estimate one of the resulting factors upwards by the essential supremum over time. After this, we introduce the cut-off function $\varphi_j$ which allows us to apply the Gagliardo-Nirenberg inequality. We also apply Lemma \ref{estimates:boundary_terms} (i). Thus, we obtain two factors which both are bounded by the right-hand side of the energy estimate \eqref{caccioppoli}. All in all, we have
\begin{align}\label{longcalc}
\iint_{Q_{j+1}} &\phi^{p(\frac{n+M}{n})} \d x \d t = \int_{T_{j+1}}\int_{B_{j+1}} \phi^p \phi^\frac{pM}{n}\d x \d t
\\
\notag &\leq \int_{T_{j+1}}\Big[ \int_{B_{j+1}} \phi^{p^*}\d x \Big]^\frac{p}{p^*} \Big[\int_{B_{j+1}} \phi^M \d x\Big]^\frac{p}{n}\d t
\\
\notag &\leq \Big[\esssup_{T_{j+1}}\int_{B_{j+1}}\phi^M \d x\Big]^\frac{p}{n}\int_{T_{j+1}}\Big[ \int_{B_{j+1}} \phi^{p^*}\d x \Big]^\frac{p}{p^*} \d t
\\
\notag &\leq \Big[\esssup_{T_{j+1}}\int_{B_{j+1}}\big(u^\frac{\beta+1}{2}-k_{j+1}^\frac{\beta+1}{2}\big)_+^2 \d x\Big]^\frac{p}{n}\int_{T_{j+1}}\Big[ \int_{B_{j+1}} (u^\beta-k^\beta_{j+1})_+^{p^*}\d x \Big]^\frac{p}{p^*} \d t
\\
\notag &\leq \Big[\esssup_{T_{j}}\int_{B_{j}}\big(u^\frac{\beta+1}{2}-k_{j+1}^\frac{\beta+1}{2}\big)_+^2\varphi_j^p \d x\Big]^\frac{p}{n}\int_{T_{j}}\Big[ \int_{B_{j}} ((u^\beta-k^\beta_{j+1})_+\varphi_j)^{p^*}\d x \Big]^\frac{p}{p^*} \d t
\\
\notag &\leq c \Big[\esssup_{T_{j}}\int_{B_{j}} \b[u,k_j]\chi_{\{u>k_j\}}\varphi_j^p \d x\Big]^\frac{p}{n}\int_{T_{j}} \int_{B_{j}} |\nabla [(u^\beta-k^\beta_{j+1})_+\varphi_j]|^p\d x \d t
\\
\notag &\leq c\Big[ \iint_{Q_j} (u^\beta-k_{j+1}^\beta)^p_+ |\nabla \varphi_j|^p + \big(u^\frac{\beta+1}{2}-k_{j+1}^\frac{\beta+1}{2}\big)_+^2\varphi_j^{p-1}|\partial_t \varphi_j| \d x \d t \Big]^\frac{n+p}{n}.
\end{align}
The constant $c$ depends only on $m,n,p,C_0,C_1$. In the set where $u>k_{j+1}$ we can estimate
\begin{align*}
\frac{(u^\beta-k_{j+1}^\beta)^p_+}{\big(u^\frac{\beta+1}{2}-k_j^\frac{\beta+1}{2}\big)_+^2} = u^{m+p-3} \frac{(1-(k_{j+1}/u)^\beta)^p_+}{\big(1- (k_j/u)^\frac{\beta+1}{2}\big)_+^2} &\leq \frac{u^{m+p-3}}{\big(1- (k_j/k_{j+1})^\frac{\beta+1}{2}\big)^2} 
\\
&\leq \frac{c k^{m+p-3}}{\big(1- (k_j/k_{j+1})^\frac{\beta+1}{2}\big)^2}< c 2^{j} k^{m+p-3},
\end{align*}
where the constant $c$ only depends on $m,p$. In the second last step we used $m+p<3$, and the fact that $k_{j+1}$ is comparable in size to $k$. Applying the previous estimate to the first term to the last line of \eqref{longcalc} and noting that in the second term we can replace $k_{j+1}$ by $k_j$ we obtain
\begin{align*}
\iint_{Q_{j+1}} &\phi^{p(\frac{n+M}{n})} \d x \d t \leq c\Big[ \iint_{Q_j} \big[ 2^{j} k^{m+p-3} |\nabla \varphi_j|^p + \varphi_j^{p-1}|\partial_t \varphi_j| \big] \big(u^\frac{\beta+1}{2}-k_j^\frac{\beta+1}{2}\big)_+^2 \d x \d t \Big]^\frac{n+p}{n}.
\end{align*}
Combining this estimate with the bounds for $\varphi_j$ and its derivatives leads to
\begin{align*}
\iint_{Q_{j+1}} \phi^{p(\frac{n+M}{n})} \d x \d t \leq c  \Big( \frac{2^{(p+1)j}}{(1-\sigma)^p\tau}\Big[\Big(\frac{\tau}{\rho^p}\Big)k^{m+p-3}+1\Big] Y_j \Big)^\frac{n+p}{n}.
\end{align*}
From the last expression we see that if $k\geq \big(\frac{\tau}{\rho^p}\big)^\frac{1}{3-m-p}$ then 
\begin{align}\label{phi-integ-est}
\iint_{Q_{j+1}} \phi^{p(\frac{n+M}{n})} \d x \d t \leq c  \Big( \frac{2^{(p+1)j}}{(1-\sigma)^p\tau} Y_j \Big)^\frac{n+p}{n}.
\end{align}
Observe now that
\begin{align}\label{Q_j-measure-est}
|Q_j\cap \{u>k_{j+1}\}|k^{\beta+1}2^{-2(j+1)} &= |Q_j\cap \{u>k_{j+1}\}|(k_{j+1}^\frac{\beta+1}{2}-k_j^\frac{\beta+1}{2})^2
\\
\notag &\leq \iint_{Q_j\cap \{u>k_{j+1}\}}(u^\frac{\beta+1}{2}-k_j^\frac{\beta+1}{2})^2\d x \d t \leq Y_j.
\end{align}
Using \eqref{phi-integ-est} and \eqref{Q_j-measure-est} in \eqref{Y-jplusone_estim1} we end up with
\begin{align}\label{iter-estim}
Y_{j+1}\leq C b^j Y_j^{1+\delta},
\end{align}
where
\begin{align*}
b=2^{(\frac{n+p}{n})(p+1)\frac{1}{q}+\frac{2}{q'}}, \hspace{5mm} C= \frac{c k^{-(\beta+1)\frac{1}{q'}}}{\big((1-\sigma)^p\tau\big)^\frac{n+p}{nq}}, \hspace{5mm}\delta=\frac{p}{nq}=\frac{M}{n+M}.
\end{align*}
and $c$ only depends on $m,n,p,C_0,C_1$. We want to show that $Y_j\to 0$. According to Lemma \ref{fastconvg} this is true provided that
\begin{align*}
Y_0\leq C^{-\frac{1}{\delta}}b^{-\frac{1}{\delta^2}}.
\end{align*}
Using the definition of $Y_0$ and the parameters we see that this is equivalent to
\begin{align}\label{kcond}
k\geq c\Big[   \big((1-\sigma)^p\tau\big)^{-\frac{n+p}{p}}     \iint_{Q_{\rho,\tau}(z_o)}u^{\beta+1}\d x \d t\Big]^\frac{p}{p\beta n+(\beta+1)(p-n)},
\end{align}
where $c$ is a constant depending only on $m,n,p,C_0,C_1$. Since
\begin{align*}
\iint_{Q_{\sigma \rho,\sigma\tau}(z_o)} (u^\frac{\beta+1}{2}-k^\frac{\beta+1}{2}\big)_+^2\d x \d t \leq Y_j \to 0,
\end{align*}
this means that $u\leq k$ almost everywhere in $Q_{\sigma \rho,(1+\sigma)\tau}(z_o)$. The only lower bounds for $k$ required in this argument were $k\geq \big(\frac{\tau}{\rho^p}\big)^\frac{1}{3-m-p}$ and \eqref{kcond}, so we have verified the estimate for the essential supremum.
\end{proof}
We end this section by proving that the estimate of Theorem \ref{theo:local_bdd} can be somewhat improved. This result will also be used in the reasoning leading to the Harnack estimate in Section \ref{sec:Harnack-est}. Note first that \eqref{nicerange-rephrased} can be rephrased as
\begin{align*}
(\beta+1)p+ n(m+p-3)>0.
\end{align*}
Thus there exists $r\in (0,\beta+1)$ such that
\begin{align}\label{def:lambda-r}
\lambda_r:=rp + n(m+p-3)>0.
\end{align}
The next theorem shows that there is an upper bound in terms of the $L^r$-norm of $u$. 
\begin{theo}\label{theo:Lr-Linfty}
Let $r\in (0,\beta+1)$ be such that \eqref{def:lambda-r} is valid. Then for any cylinder $Q_{2\rho,2\tau}(z_o)\subset \Omega_T$,
\begin{align}\label{essup-Lr-estim}
\esssup_{Q_{\rho,\tau}(z_o)} u \leq c\Big[ \tau^{-\frac{n+p}{p}}\iint_{Q_{2\rho,2\tau}(z_o)}u^r\d x \d t\Big]^\frac{p}{\lambda_r} + c\Big(\frac{\tau}{\rho^p}\Big)^\frac{1}{3-m-p}
\end{align}
where the constant $c$ depends only on $r$ and the data.
\end{theo}
\begin{proof}[Proof]
Define the increasing sequences
\begin{align*}
\rho_j:=(2-2^{-j})\rho, \hspace{7mm} \tau_j:=(2-2^{-j})\tau.
\end{align*}
Define cylinders $Q_j=Q_{\rho_j,\tau_j}(z_o)$. Applying Theorem \ref{theo:local_bdd} to the cylinder $Q_{j+1}$ with $\sigma=\rho_j/\rho_{j+1}=\tau_j/\tau_{j+1}$ and noting that $1-\sigma> 2^{-(j+2)}$ we end up with
\begin{align*}
\esssup_{Q_j} u &\leq c\Big[ 2^{j(n+p)}\tau_j^{-\frac{n+p}{p}}\iint_{Q_{j+1}}u^{\beta+1}\d x \d t\Big]^\frac{p}{p\beta n+(\beta+1)(p-n)} + \Big(\frac{\tau_j}{\rho_j^p}\Big)^\frac{1}{3-m-p}
\\
& \leq c\Big[ 2^{j(n+p)}\tau^{-\frac{n+p}{p}}\iint_{Q_{j+1}}u^{\beta+1}\d x \d t\Big]^\frac{p}{p\beta n+(\beta+1)(p-n)} + \Big(\frac{2\tau}{\rho^p}\Big)^\frac{1}{3-m-p},
\end{align*}
where in the second step we used the fact that $\rho_j\geq \rho$ and $\tau\leq \tau_j < 2\tau$. Denoting now $M_j:= \esssup_{Q_j} u$ and noting that $u\leq M_{j+1}$ a.e. in $Q_{j+1}$ we see that
\begin{align*}
M_j \leq c M_{j+1}^\frac{p(\beta+1-r)}{p\beta n+(\beta+1)(p-n)}\Big[ 2^{j(n+p)}\tau^{-\frac{n+p}{p}}\iint_{Q_{2\rho,2\tau}(z_o)}u^r\d x \d t\Big]^\frac{p}{p\beta n+(\beta+1)(p-n)} + \Big(\frac{2\tau}{\rho^p}\Big)^\frac{1}{3-m-p}.
\end{align*}
Due to \eqref{def:lambda-r}, the exponent of $M_{j+1}$ lies in the interval $(0,1)$. Applying Young's inequality to increase the exponent of $M_{j+1}$ to $1$ we end up with
\begin{align*}
M_j &\leq \varepsilon M_{j+1}+c(\varepsilon)\Big[ 2^{j(n+p)}\tau^{-\frac{n+p}{p}}\iint_{Q_{2\rho,2\tau}(z_o)}u^r\d x \d t\Big]^\frac{p}{\lambda_r} + \Big(\frac{2\tau}{\rho^p}\Big)^\frac{1}{3-m-p}
\\
&= \varepsilon M_{j+1}+c(\varepsilon) b^j \Big[ \tau^{-\frac{n+p}{p}}\iint_{Q_{2\rho,2\tau}(z_o)}u^r\d x \d t\Big]^\frac{p}{\lambda_r} + \Big(\frac{2\tau}{\rho^p}\Big)^\frac{1}{3-m-p},
\end{align*}
where $b=2^{\frac{p(n+p)}{\lambda_r}}$ and the constant $\varepsilon>0$ can be chosen freely. Iterating the last inequality we obtain
\begin{align*}
M_0 \leq \varepsilon^N M_N + c(\varepsilon)\Big[ \tau^{-\frac{n+p}{p}}\iint_{Q_{2\rho,2\tau}(z_o)}u^r\d x \d t\Big]^\frac{p}{\lambda_r}\sum^{N-1}_{j=0}(\varepsilon b)^j + \Big(\frac{2\tau}{\rho^p}\Big)^\frac{1}{3-m-p}\sum^{N-1}_{j=0}\varepsilon^j,
\end{align*}
for $N\geq 1$. Choosing $\varepsilon=\frac{1}{2b}$ we see that both sums on the right-hand side converge as $N\to \infty$. Since $M_N$ is bounded from above by the essential supremum of $u$ over $Q_{2\rho,2\tau}(z_o)$, the term $\varepsilon^N M_N$ vanishes in the limit and we end up with \eqref{essup-Lr-estim}.
\end{proof}

\section{H\"older continuity}
In this setion we consider only $m$ and $p$ in the supercritical range \eqref{supercritical}. We show that in this case weak solutions are locally H\"older continuous. The starting point of the argument is a De Giorgi type lemma providing a sufficient condition for the reduction of the oscillation from above. First we introduce some notation. For $0 <\mu_+<\infty$ we denote
\begin{align}\label{def:theta}
\theta=\varepsilon\mu_+^{3-m-p},
\end{align}
where $\varepsilon \in (0,1)$. A sufficiently small value of $\varepsilon$ will be chosen later in this section. Initially it is important that our results work for all $\varepsilon \in (0,1)$.
 
\begin{lem}\label{lem:De_Giorgi_above}
Let $u$ be a weak solution to \eqref{general} in the sense of Definition \ref{weakdef}. Suppose that we are given a number $0<\mu_+<\infty$, and let $\theta$ be chosen as in \eqref{def:theta}. Moreover, suppose $Q_{\rho,\theta\rho^p}(z_o) \subset \Omega_T$ is a parabolic cylinder satisfying
$$
\esssup_{Q_{\rho,\theta \rho^p}(z_o)} u \leq \mu_+.
$$
then there exists a constant $\nu_o$ depending only on $m,n,p, C_0, C_1$ such that if
\begin{align}
\label{assumption_level_1}
|Q_{\rho,\theta\rho^p}(z_o)\cap \{u^\beta > \mu_+^\beta/2 \}|\leq \nu_o \varepsilon^\frac{n}{p} |Q_{\rho,\theta\rho^p}(z_o)|
\end{align}
then
$$
u^\beta \leq \frac{3}{4}\mu_+^\beta
$$
a.e.\ in $Q_{\rho/2,\theta(\rho/2)^p}(z_o)$.
\end{lem}
\begin{proof}[Proof]
Define sequences of numbers and sets as follows:
\begin{align*}
\rho_j:= \frac{1}{2}\big(1+\frac{1}{2^j}\big)\rho, \hspace{5mm} k_j^\beta:=\big(1-\frac{1}{4}-\frac{1}{2^{j+2}}\big)\mu_+^\beta, \hspace{5mm} Q_j&:=Q_{\rho_j,\theta\rho_j^p}(z_o),
\\
A_j:=Q_j\cap \{u>k_j\}, \hspace{5mm} Y_j&:=\frac{|A_j|}{|Q_j|}.
\end{align*}
We can now choose functions $\varphi_j\in C^\infty(Q_j;[0,1])$ vanishing near the parabolic boundary of $Q_j$ and satisfying $\phi_j=1$ on $Q_j$ and for which
\begin{align*}
|\nabla \varphi_j|\leq \rho^{-1} 2^{j+2}, \hspace{10mm} |\partial_t \varphi_j|\leq c_p \theta^{-1}\rho^{-p} 2^{jp}.
\end{align*}
Note that in the set where $u>k_{j+1}$ we have
\begin{align}\label{shtevensh}
u^\beta-k_j^\beta>k_{j+1}^\beta-k_j^\beta=\frac{\mu_+^\beta}{2^{j+3}}.
\end{align}
This observation and H\"older's inequality show that
\begin{align}\label{A_jplusone-est}
\frac{\mu_+^{\beta p}}{2^{(j+3)p}}|A_{j+1}|&\leq \iint_{A_{j+1}}(u^\beta-k_j^\beta)_+^p \d x \d t
\\
&\leq \Big[ \iint_{A_{j+1}}(u^\beta-k_j^\beta)_+^{p\frac{(n+p)}{n}} \d x \d t\Big]^\frac{n}{n+p}|A_{j+1}|^\frac{p}{n+p}.
\end{align}
The integral in the last expression can be estimated using H\"older's inequality and \eqref{shtevensh} as
\begin{align*}
&\iint_{A_{j+1}}(u^\beta-k_j^\beta)_+^{p\frac{(n+p)}{n}} \d x \d t \leq \int_{T_{j+1}}\int_{B_{j+1}}(u^\beta-k_j^\beta)_+^{p\frac{p}{n}}\chi_{A_{j+1}} (u^\beta-k_j^\beta)_+^p \d x \d t
\\
&\leq \int_{T_{j+1}}\Big[\int_{B_{j+1}}(u^\beta-k_j^\beta)_+^p\chi_{A_{j+1}}\d x\Big]^\frac{p}{n}\Big[\int_{B_{j+1}}(u^\beta-k_j^\beta)_+^{p^*}\d x \Big]^\frac{p}{p^*}\d t
\\
&\leq c2^{j(2-p)\frac{p}{n}}\mu_+^{\beta(p-2)\frac{p}{n}}\int_{T_{j+1}}\Big[\int_{B_{j+1}}(u^\beta-k_j^\beta)_+^2\d x\Big]^\frac{p}{n}\Big[\int_{B_{j+1}}(u^\beta-k_j^\beta)_+^{p^*}\d x \Big]^\frac{p}{p^*}\d t
\\
&\leq c2^{j(2-p)\frac{p}{n}}\mu_+^{\beta(p-2)\frac{p}{n}}\Big[\esssup_{T_{j+1}}\int_{B_{j+1}}(u^\beta-k_j^\beta)_+^2\d x\Big]^\frac{p}{n}\int_{T_{j+1}}\Big[\int_{B_{j+1}}(u^\beta-k_j^\beta)_+^{p^*}\d x \Big]^\frac{p}{p^*}\d t,
\end{align*}
where in the last step we have estimated one of the integrals over space upwards by taking the essential supremum in time. Note that by Lemma \eqref{estimates:boundary_terms} (ii) we have
\begin{align*}
\tfrac1c (u^\beta-k_j^\beta)_+^2\leq (u^{\beta-1}+k_j^{\beta-1})\b[u,k_j]\leq 2\mu_+^{\beta-1}\b[u,k_j].
\end{align*}
Using this observation and introducing the cut-off functions $\varphi_j$ puts us into a position to apply Sobolev inequality and the energy estimate \eqref{caccioppoli} as follows.
\begin{align*}
&\iint_{A_{j+1}}(u^\beta-k_j^\beta)_+^{p\frac{(n+p)}{n}} \d x \d t
\\
&\quad \leq c2^{j(2-p)\frac{p}{n}}\mu_+^{\frac{p}{n}(m+p-3)}\Big[\esssup_{T_j}\int_{B_j}\b[u,k_j]^+\varphi_j^p\d x\Big]^\frac{p}{n}\int_{T_j}\Big[\int_{B_j}[(u^\beta-k_j^\beta)_+\varphi_j]^{p^*}\d x \Big]^\frac{p}{p^*}\d t
\\
&\quad \leq c2^{j(2-p)\frac{p}{n}}\mu_+^{\frac{p}{n}(m+p-3)}\Big[\esssup_{T_j}\int_{B_j}\b[u,k_j]^+\varphi_j^p\d x\Big]^\frac{p}{n}\int_{T_j} \int_{B_j}|\nabla [(u^\beta-k_j^\beta)_+\varphi_j]|^p\d x \d t
\\
&\quad \leq c2^{j(2-p)\frac{p}{n}}\mu_+^{\frac{p}{n}(m+p-3)}\Big[\iint_{A_j}(u^\beta-k_j^\beta)_+^p|\nabla \varphi_j|^p + \b[u,k_j]\varphi_j^{p-1}|\partial_t\varphi_j|\d x \d t\Big]^\frac{n+p}{n}.
\end{align*}
The second term in the last integral can be estimated using Lemma \ref{estimates:boundary_terms} (iii) and the bound for $|\partial_t\varphi_j|$ as
\begin{align*}
\b[u,k_j]\varphi_j^{p-1}|\partial_t\varphi_j|\leq c \mu_+^{\beta+1} \theta^{-1}\rho^{-p}2^{jp} = c\varepsilon^{-1}\mu_+^{\beta p}\rho^{-p} 2^{jp}.
\end{align*}
Using this and the bound for $|\nabla \varphi_j|$ and $u$ we see that
\begin{align*}
\iint_{A_{j+1}}(u^\beta-k_j^\beta)_+^{p\frac{(n+p)}{n}} \d x \d t
\leq c2^{j\frac{p}{n}(n+2)}\mu_+^{\frac{p}{n}(m+p-3)}(\varepsilon^{-1}\rho^{-p}\mu_+^{\beta p}|A_j|)^\frac{n+p}{n}.
\end{align*}
Combining this estimate with \eqref{A_jplusone-est} and the observation that $|A_{j+1}|\leq |A_j|$ we have
\begin{align*}
|A_{j+1}|&\leq c\varepsilon^{-1} 2^{jp[1+\frac{n+2}{n+p}]}\mu_+^{\frac{p}{n+p}(m+p-3)}\rho^{-p}|A_j|^{1+\frac{p}{n+p}}
\\
&= c\varepsilon^{-\frac{n}{n+p}} b^j\theta^{-\frac{p}{n+p}}\rho^{-p}|A_j|^{1+\frac{p}{n+p}},
\end{align*}
where $c$ and $b$ only depend on $m,n,p, C_0,C_1$. Dividing the last expression by $|Q_j|$ and noting that $|Q_j|$ is proportional to $\theta \rho^{n+p}$ we obtain
\begin{align*}
Y_{j+1}\leq c\varepsilon^{-\frac{n}{n+p}}b^j Y_j^{1+\frac{p}{n+p}}.
\end{align*}
Setting $\delta:=\frac{p}{n+p}$ we see that Lemma \ref{fastconvg} guarantees that $Y_j\to 0$ provided that
\begin{align*}
\frac{|Q_{\rho,\theta\rho^p}(z_o)\cap \{ u^\beta>\mu_+^\beta/2\}|}{|Q_{\rho,\theta\rho^p}(z_o)|} = Y_0 \leq (c\varepsilon^{-\frac{n}{n+p}})^{-\frac{1}{\delta}} b^{-\frac{1}{\delta^2}}= \varepsilon^\frac{n}{p}\nu_o,
\end{align*}
where $\nu_o= c^{-\frac{1}{\delta}}b^{-\frac{1}{\delta^2}}$ only depends on $m,n,p, C_0,C_1$. Since $|Q_j|$ is bounded from above, this also means that $|A_j|\to 0$. Furthermore, since
\begin{align*}
Q_{\rho/2,\theta(\rho/2)^p}(z_o)\cap \{u^\beta >\tfrac{3}{4}\mu_+^\beta\}\subset A_j,
\end{align*}
for all $j$, the measure of the set on the left hand side must be zero.
\end{proof}
\subsection{Reduction of the oscillation}
We are now ready to prove the reduction of the oscillation in the case $\mu_-=0$.
If the condition of the De Giorgi lemma holds, then we have a reduction of the oscillation from above. Suppose now that the condition in the De Giorgi lemma fails, i.e.
\begin{align*}
|Q_{\rho,\theta\rho^p}(z_o)\cap \{u^\beta > \mu_+^\beta/2 \}|> \nu_o \varepsilon^\frac{n}{p}|Q_{\rho,\theta\rho^p}(z_o)|.
\end{align*}
Then there is a set $\Delta \subset (t_o-\theta\rho^p,t_o)$ of positive measure such that
\begin{align*}
|\{x\in B_\rho(x_o)\,|\, u(x,\tau)>2^{-\frac{1}{\beta}}\mu_+\}|>\nu_o\varepsilon^\frac{n}{p} |B_\rho(x_o)|,
\end{align*}
for all $\tau\in \Delta$. Provided that $\bar{Q}_{2\rho,\theta\rho^p}\subset \Omega\times[0,T)$, the $L^1$-Harnack inequality of Theorem \ref{harnack} for the time-continuous representative of $u$ shows that for $\tau\in \Delta$,
\begin{align}\label{harnapplic}
\nu_o \varepsilon^\frac{n}{p} c_n\rho^n 2^{-\frac{1}{\beta}}\mu_+ &< 2^{-\frac{1}{\beta}}\mu_+|\{x\in B_\rho(x_o)\,|\, u(x,\tau)>2^{-\frac{1}{\beta}}\mu_+\}|\leq \int_{B_\rho(x_o)} u(x,\tau)\d x
\\
&\leq \gamma \inf_{t\in (t_o-\theta\rho^p,t_o)} \int_{B_{2\rho}(x_o)} u(x,t)\d x + \gamma \Big(\frac{\theta \rho^p}{\rho^\lambda}\Big)^\frac{1}{3-m-p}
\end{align}
By the definition of $\lambda$ and $\theta$ we see that
\begin{align*}
\Big(\frac{\theta \rho^p}{\rho^\lambda}\Big)^\frac{1}{3-m-p}=\varepsilon^\frac{1}{3-m-p}\mu_+\rho^n.
\end{align*}
Moving this term to the right-hand side of \eqref{harnapplic} we obtain
\begin{align}\label{jytkytsunami}
\varepsilon^\frac{n}{p}(c - \gamma\varepsilon^\kappa)\mu_+\rho^n \leq \gamma \inf_{t\in (t_o-\theta\rho^p,t_o)} \int_{B_{2\rho}(x_o)} u(x,t)\d x,
\end{align}
where $c=c(m,n,p, C_0, C_1)$ and
\begin{align*}
\kappa=\frac{1}{3-m-p}-\frac{n}{p},
\end{align*}
is a positive number by \eqref{supercritical}. If we now choose
\begin{align*}
\varepsilon:= \Big(\frac{c}{2\gamma}\Big)^\frac{1}{\kappa},
\end{align*}
which clearly only depends on $m,n,p, C_0, C_1$ we also see from \eqref{jytkytsunami} that
\begin{align}\label{nunnuka}
C\mu_+\rho^n \leq \inf_{t\in (t_o-\theta\rho^p,t_o)} \int_{B_{2\rho}(x_o)} u(x,t)\d x,
\end{align}
for a constant $C=C(m,n,p, C_0, C_1)\leq 1$. Take now $\zeta>0$ and note that\
\begin{align*}
\int_{B_{2\rho}(x_o)} u(x,t)\d x &= \int_{B_{2\rho}(x_o)\cap \{u(x,t)\geq \zeta\mu_+\}} u(x,t)\d x + \int_{B_{2\rho}(x_o)\cap \{u(x,t)<\zeta\mu_+\}} u(x,t)\d x
\\
&\leq \mu_+|B_{2\rho}(x_o)\cap \{u(x,t)\geq \zeta \mu_+\}| + \zeta\mu_+ |B_{2\rho}(x_o)|.
\end{align*}
With the choice $\zeta:=C 2^{-(n+1)}/c_n$ where $C$ is the constant from \eqref{nunnuka}, the last estimate and \eqref{nunnuka} show that
\begin{align}\label{sss}
|B_{2\rho}(x_o)\cap \{u(x,t)\geq \zeta\mu_+\}|\geq \alpha |B_{2\rho}(x_o)|
\end{align}
for all $t\in (t_o-\theta\rho^p,t_o)$ for a constant $\alpha$ depending only on $m,n,p, C_0,C_1$. Suppose now that $Q_{32\rho, \theta \rho^p}(z_o)\subset \Omega_T$. This puts us in a position to apply Theorem \ref{theo:exppos} for a sufficiently small $M$. Namely, taking $M= \min\{\zeta,(\varepsilon/2^p)^\frac{1}{3-m-p}\}\mu_+$, we see that \eqref{sss} is still valid with $\zeta \mu_+$ replaced by $M$ and furthermore that
\begin{align*}
B_{32\rho}(x_o)\times(t_o-\delta M^{3-m-p}(2\rho)^p, t_o)\subset Q_{32\rho, \theta \rho^p}(z_o)\subset \Omega_T,
\end{align*}
where $\delta\in (0,1)$ is the constant from Theorem \ref{theo:exppos}. Hence, we may apply Theorem \ref{theo:exppos} with $s= t_o-\delta M^{3-m-p}(2\rho)^p$ and $\rho$ replaced by $2\rho$ to conclude that there is a $\xi\in (0,1)$ and $\tilde{\varepsilon}<\varepsilon$ depending only on $m,n,p,C_0,C_1$ such that
\begin{align}\label{red_from_below}
u\geq \xi \mu_+ \textrm{ in } B_{4\rho}(x_o)\times (t_o-\tilde{\varepsilon}\mu_+^{3-m-p}\rho^p,t_o),
\end{align}
which is the reduction of the oscillation from below.
Combining the previous reasoning and Lemma \ref{lem:De_Giorgi_above}, we have shown the following.
\begin{lem}\label{lem:mu-plus_red_osc}
There are constants $\varepsilon, \gamma, \eta \in (0,1)$ depending only on $m,n,p,C_0,C_1$ such that for any weak solution $u$ and number $\mu_+>0$ satisfying the conditions $Q_{32\rho, \varepsilon\mu_+^{3-m-p}\rho^p}(z_o)\subset \Omega_T$ and $u\leq \mu_+$ on $Q_{\rho, \varepsilon\mu_+^{3-m-p}\rho^p}(z_o)$, we have
\begin{align}\label{osc-red1}
\essosc_{Q_{\frac{\rho}{2},\gamma\varepsilon\mu_+^{3-m-p}\rho^p}(z_o)} u \leq \eta \mu_+.
\end{align}
Furthermore, one of the following condition must hold in the cylinder $Q_{\frac{\rho}{2},\gamma\varepsilon\mu_+^{3-m-p}\rho^p}(z_o)$:
\begin{align}\label{gandauli}
&(i)\quad \esssup_{Q_{\frac{\rho}{2},\gamma\varepsilon\mu_+^{3-m-p}\rho^p}(z_o)} u\leq \big(\frac{1+\eta}{2}\big)\mu_+, \textrm{ or }
\\
\notag &(ii)\quad \essinf_{Q_{\frac{\rho}{2},\gamma\varepsilon\mu_+^{3-m-p}\rho^p}(z_o)} u\geq \big(\frac{1-\eta}{2}\big)\mu_+.
\end{align}
\end{lem}
\begin{proof}[Proof]
By Lemma \ref{lem:De_Giorgi_above} and the previous reasoning, \eqref{osc-red1} is valid with $\gamma=\min\{ 2^{-p}, \frac{\tilde{\varepsilon}}{\varepsilon}\}$ and $\eta=\max \{ (3/4)^\frac{1}{\beta}, 1-\xi\}$, where $\tilde{\varepsilon}$ and $\xi$ are the constants appearing in \eqref{red_from_below}. Furthermore, if \eqref{gandauli} $(i)$ fails, \eqref{osc-red1} shows that we must have
\begin{align*}
\essinf_{Q_{\frac{\rho}{2},\gamma\varepsilon\mu_+^{3-m-p}\rho^p}(z_o)} u\geq \big(\frac{1+\eta}{2}\big)\mu_+ -\eta\mu_+= \big(\frac{1-\eta}{2}\big)\mu_+,
\end{align*}
so that \eqref{gandauli} (ii) holds.
\end{proof}

\begin{lem}\label{lem:osc-red}
There are constants $c$ and $\nu$ depending only on $m,n,p,C_0,C_1$ such that for any weak solution $u$ and number $\mu_+>0$ for which $Q_{32\rho, \varepsilon\mu_+^{3-m-p}\rho^p}(z_o)\subset \Omega_T$ and $u\leq \mu_+$ on $Q_{\rho, \varepsilon\mu_+^{3-m-p}\rho^p}(z_o)$, we have
\begin{align}\label{r-rho-osc-red}
\essosc_{Q_{r,\varepsilon\mu_+^{3-m-p}r^p}(z_o)} u \leq c \mu_+\Big(\frac{r}{\rho}\Big)^\nu,
\end{align}
for all $0< r \leq \rho$. Here, $\varepsilon$ is the constant from Lemma \ref{lem:mu-plus_red_osc}.
\end{lem}
\begin{proof}[Proof]
Denote $C:= 2\max\{2,\gamma^{-\frac{1}{p}}\}$ where $\gamma$ is the constant from Lemma \ref{lem:mu-plus_red_osc} and define
\begin{align*}
\mu_+^j:=\Big(\frac{1+\eta}{2}\Big)^j\mu_+,\hspace{5mm} \rho_j:=\rho/C^j.
\end{align*}
With these choices,
\begin{align*}
Q_{\rho_1,\varepsilon(\mu^1_+)^{3-m-p}\rho_1^p}(z_o)\subset Q_{\frac{\rho}{2},\gamma\varepsilon\mu_+^{3-m-p}\rho^p}(z_o),
\end{align*}
and Lemma \ref{lem:mu-plus_red_osc} guarantees that
\begin{align*}
\essosc_{Q_{\rho_1,\varepsilon(\mu^1_+)^{3-m-p}\rho_1^p}(z_o)} u\leq \eta \mu_+
\end{align*}
Furthermore, if we are in the case \eqref{gandauli} (i), we have $u\leq \mu_1$ on $Q_{\rho_1,\varepsilon(\mu^1_+)^{3-m-p}\rho_1^p}(z_o)$ and we may apply Lemma \ref{lem:mu-plus_red_osc} to this subcylinder instead to conclude that
\begin{align*}
\essosc_{Q_{\rho_2,\varepsilon(\mu^1_+)^{3-m-p}\rho_2^p}(z_o)} u\leq \eta \mu^1_+
\end{align*}
Also, Lemma \ref{lem:mu-plus_red_osc} guarantees that one of the conditions of \eqref{gandauli} holds with $\rho$ replaced by $\rho_1$ and $\mu_+$ replaced by $\mu^1_+$. If condition (i) is true, we are again in a position to continue the iteration. Continuing in this way, we see that as long as we stay in case (i) at every step of the iteration we have
\begin{align}\label{rutilus}
\essosc_{Q_{\rho_j,\varepsilon(\mu^j_+)^{3-m-p}\rho_j^p}(z_o)} u &\leq \eta \mu^{j-1}_+,
\\
\label{perca}\esssup_{Q_{\rho_{j-1}, \varepsilon(\mu^{j-1}_+)^{3-m-p}\rho_{j-1}^p}(z_o)} u &\leq \mu^{j-1}_+
\end{align}
Either this estimate holds for every $j \in \N$, or there is a $k\in \N$ such that \eqref{rutilus} holds for all $j\in \{1,\dots, k\}$ and
\begin{align}\label{lower_bnd_u}
\essinf_{Q_{\frac{\rho_{k-1}}{2},\gamma\varepsilon(\mu^{k-1}_+)^{3-m-p}\rho_{k-1}^p}(z_o)} u\geq \big(\frac{1-\eta}{2}\big)\mu^{k-1}_+ = \big(\frac{1-\eta}{1+\eta}\big)\mu^k_+.
\end{align}
We assume for now the existence of such a $k$ and investigate its consequences. In the end we will show that the estimate \eqref{r-rho-osc-red} holds whether $k$ exists or not. Since
\begin{align*}
\frac{\rho_{k-1}}{2}=\tfrac{C}{2}\rho_k\geq 2\rho_k, \hspace{5mm} \gamma\varepsilon(\mu^{k-1}_+)^{3-m-p}\rho_{k-1}^p &=\gamma C^p \big( \tfrac{2}{1+\eta}\big)^{3-m-p}\varepsilon(\mu^k_+)^{3-m-p}\rho_k^p
\\
&\geq 2^p \big( \tfrac{2}{1+\eta}\big)^{3-m-p}\varepsilon(\mu^k_+)^{3-m-p}\rho_k^p
\\
&> \varepsilon (\mu^k_+)^{3-m-p}(2\rho_k)^p,
\end{align*}
it follows from \eqref{lower_bnd_u} and \eqref{perca} with $j=k$ that
\begin{align*}
\big(\frac{1-\eta}{1+\eta}\big)\mu^k_+ \leq u \leq \frac{2}{1+\eta}\mu_+^k,\hspace{7mm} \textrm{in } Q_{2\rho_k, \varepsilon (\mu^k_+)^{3-m-p} (2\rho_k)^p}(z_o).
\end{align*}
Up to a translation in the time variable this is exactly the situation of Lemma \ref{re-scaling} with $M=\mu_+^k$.
By translation we may assume that $t_o=0$. Lemma \ref{re-scaling} shows that the function
\begin{align*}
v(x,t)=(\mu_+^k)^{-1}u\big(x, (\mu_+^k)^{3-m-p}t\big), \hspace{3mm}(x,t)\in Q_{2\rho_k, \varepsilon(2\rho_k)^p}(x_o,0).
\end{align*}
solves an equation of parabolic $p$-Laplace type, where the constants in the structure conditions only depend on $m,n,p,C_0,C_1$. Applying Lemma \ref{lem:p-parabolic_hold_cont} to $v$ then shows that for all $(x,t), (y,s)\in Q_{\rho_k, \varepsilon(\rho_k)^p}(x_o,0)$,
\begin{align*}
|v(x,t)-v(y,s)|\leq c\Big[\frac{|x-y|+|t-s|^\frac{1}{p}}{\rho_k}\Big]^{\nu_o},
\end{align*}
where the constants $c$ and $\nu_o$ only depend on $m,n,p,C_0,C_1$. Since the fraction in the last estimate is bounded from above by $2+\varepsilon^\frac{1}{p}$ we see that for any $0<\nu\leq\nu_o$ we have
\begin{align*}
|v(x,t)-v(y,s)|\leq c(2+\varepsilon^\frac{1}{p})^{\nu_o-\nu}\Big[\frac{|x-y|+|t-s|^\frac{1}{p}}{\rho_k}\Big]^\nu < c(2+\varepsilon^\frac{1}{p})^{\nu_o}\Big[\frac{|x-y|+|t-s|^\frac{1}{p}}{\rho_k}\Big]^\nu,
\end{align*}
for all $(x,t), (y,s)\in Q_{\rho_k, \varepsilon(\rho_k)^p}(x_o,0)$.
For the original function $u$ this translates into
\begin{align}\label{u-hold-est}
|u(x,t)-u(y,s)|\leq c\mu^k_+\Big[\frac{|x-y|+(\mu^k_+)^\frac{m+p-3}{p}|t-s|^\frac{1}{p}}{\rho_k}\Big]^\nu,
\end{align}
for all $(x,t),(y,s) \in Q_{\rho_k, \varepsilon (\mu^k_+)^{3-m-p} \rho_k^p}(z_o)$ and $0<\nu\leq\nu_o$. The constant $c$ still depends only on $m,n,p,C_0,C_1$. Now we are ready to prove \eqref{r-rho-osc-red}. For this, take $0< r \leq \rho$. Pick $j\in \N_0$ such that
\begin{align*}
\Big(\frac{1+\eta}{2}\Big)^{(j+1)\frac{(3-m-p)}{p}}\frac{\rho}{C^{j+1}} < r\leq \Big(\frac{1+\eta}{2}\Big)^{j\frac{(3-m-p)}{p}}\frac{\rho}{C^j}.
\end{align*}
From the left inequality we can deduce that
\begin{align*}
\ln \Big[\frac{r}{\rho}\Big] > \ln \Big[\frac{1}{C}\Big(\frac{1+\eta}{2}\Big)^\frac{3-m-p}{p}\Big] + j \ln \Big[\frac{1}{C}\Big(\frac{1+\eta}{2}\Big)^\frac{3-m-p}{p}\Big],
\end{align*}
and with some further manipulations that
\begin{align}\label{j-estim}
j> -1 -b \ln \Big[\frac{r}{\rho}\Big] ,
\end{align}
for some $b>0$ depending only on $m,n,C_0,C_1$.
Note that $r<\rho_j$ and
\begin{align*}
\varepsilon\mu_+^{3-m-p}r^p \leq \varepsilon (\mu^j_+)^{3-m-p}\rho_j^p,
\end{align*}
so $Q_{r, \varepsilon\mu_+^{3-m-p}r^p}(z_o)\subset Q_{\rho_j, \varepsilon (\mu^j_+)^{3-m-p}\rho_j^p}(z_o)$. If $j\leq k$ (or if $k$ does not exist, which means that \eqref{rutilus} is valid for all $j$) then \eqref{rutilus} and \eqref{j-estim} imply that
\begin{align*}
\essosc_{ Q_{r, \varepsilon\mu_+^{3-m-p}r^p}(z_o) } u\leq \eta \mu^{j-1}_+ = \frac{2\eta}{1+\eta}\mu_+ \Big[\frac{1+\eta}{2}\Big]^j< \eta\mu_+\Big[\frac{2}{1+\eta}\Big]^2 \Big[\frac{2}{1+\eta}\Big]^{b\ln[r/\rho]} = c \mu_+\Big(\frac{r}{\rho}\Big)^{\nu_1},
\end{align*}
for some positive constants $c$ and $\nu_1$ depending only on $m,n,p,C_0,C_1$. Suppose now instead that $j>k$. Then $Q_{r, \varepsilon\mu_+^{3-m-p}r^p}(z_o)\subset Q_{\rho_k, \varepsilon (\mu^k_+)^{3-m-p}\rho_k^p}(z_o)$ so from \eqref{u-hold-est} we see that
\begin{align*}
\essosc_{ Q_{r, \varepsilon\mu_+^{3-m-p}r^p}(z_o) } u &\leq c\mu^k_+\Big[\frac{2r+(\mu^k_+)^\frac{m+p-3}{p}\varepsilon^\frac{1}{p}\mu_+^\frac{(3-m-p)}{p}r }{\rho_k}\Big]^\nu
\\
&= c \Big(\frac{1+\eta}{2}\Big)^k\mu_+ \Big[\frac{2r+(\frac{2}{1+\eta})^{k\frac{(3-m-p)}{p}}\varepsilon^\frac{1}{p}r }{\rho/C^k}\Big]^\nu
\\
&\leq c\Big[\Big(\frac{2}{1+\eta}\Big)^{\frac{\nu(3-m-p)}{p}-1}C^\nu\Big]^k \mu_+\Big(\frac{r}{\rho}\Big)^\nu.
\end{align*}
Observe now that the expression inside the square brackets can be made smaller than or equal to one by taking $\nu\leq \nu_2$ where the upper bound $\nu_2$ depends only on $C$ and $\eta$ and hence only on $m,n,p,C_0,C_1$. Taking now $\nu:=\min\{\nu_1,\nu_2\}$ we finally have verified that \eqref{r-rho-osc-red} holds in all cases.
\end{proof}
\subsection{H\"older continuity} Using Lemma \ref{osc-red1} we can now easily prove the local H\"older continuity.
\begin{theo}
Let $u$ be a weak solution in the sense of Definition \ref{weakdef}. Let $m$ and $p$ be in the supercritical range \eqref{supercritical}. Then $u$ is locally H\"older continuous in $\Omega_T$ and the H\"older exponent depends only on $m,n,p,C_0,C_1$.
\end{theo}
\begin{proof}[Proof]
Let $z_o\in \Omega_T$. Pick $R>0$ such that the $(n+1)$-dimensional closed ball $\bar{B}^{n+1}_{2R}(z_o)$ centered at $z_o$ is contained in $\Omega_T$ and define
\begin{align*}
\mu_+=\esssup_{ \bar{B}^{n+1}_{2R}(z_o) } u <\infty.
\end{align*}
The number $\mu_+$ is finite since the range \eqref{supercritical} is contained in the range \eqref{nicerange} which according to Theorem \ref{theo:local_bdd} guarantees local boundedness. By picking a suitable representative of $u$, we may assume that $\mu_+$ is the actual supremum of $u$ on the ball $\bar{B}^{n+1}_{2R}(z_o)$. We can now choose $\rho>0$ so small that for all $z\in \bar{B}^{n+1}_{R}(z_o)$, we have
\begin{align*}
Q_{\rho,\varepsilon\mu_+^{3-m-p}\rho^p}(z)\subset \bar{B}^{n+1}_{2R}(z_o),\hspace{10mm} Q_{32\rho,\varepsilon\mu_+^{3-m-p}\rho^p}(z)\subset \Omega_T.
\end{align*}
From the first condition it follows that $u\leq \mu_+$ in every cylinder $Q_{\rho,\varepsilon\mu_+^{3-m-p}\rho^p}(z)$ where $z\in \bar{B}^{n+1}_{R}(z_o)$. Thus, according to Lemma \ref{lem:osc-red},
\begin{align}\label{aaaaa}
\essosc_{Q_{r,\varepsilon\mu_+^{3-m-p}r^p}(z)} u \leq c \mu_+\Big(\frac{r}{\rho}\Big)^\nu,
\end{align}
for every $r\in (0,\rho)$ and $z\in \bar{B}^{n+1}_{R}(z_o)$. If in the above estimate we had the oscillation rather than the essential oscillation we could now apply \eqref{aaaaa} to any pair of points that are sufficiently close to each other. Since this is not case, we must first exclude a set of measure zero so that the different types of oscillation coincide. In order to ensure that we are only disregarding a set of measure zero, this should be done only for a countable number of cylinders. We now make this idea precise. For every $(z,r)$ in the countable set $[\bar{B}^{n+1}_{R}(z_o)\cap \Q^{n+1} ] \times [(0,\rho)\cap \Q]$, there is a set $N_r^z\subset Q_{r,\varepsilon\mu_+^{3-m-p}r^p}(z)$ of measure zero such that for all $(y,s)\in Q_{r,\varepsilon\mu_+^{3-m-p}r^p}(z)\setminus N^z_r$,
\begin{align*}
\essinf_{ Q_{r,\varepsilon\mu_+^{3-m-p}r^p}(z) } u \leq u(y,s) \leq \esssup_{ Q_{r,\varepsilon\mu_+^{3-m-p}r^p}(z) } u.
\end{align*}
Define $N=\cup_{(z,r)} N^z_r$, and suppose that $z_1,z_2 \in B^{n+1}_{R}(z_o)\setminus N$. We may also assume that $t_1\leq t_2$. Suppose first that $z_1\in Q_{\rho,\varepsilon\mu_+^{3-m-p}\rho^p}(z_2)\cup (B_\rho(x_2)\times\{t_2\})$. Then there is a sequence of numbers $(z^j)\subset B^{n+1}_{R}(z_o)\cap \Q^{n+1}$ such that $z^j\to z_2$, $t^j\geq t_2$, and $z_1 \in Q_{\rho,\varepsilon\mu_+^{3-m-p}\rho^p}(z^j)$ for all $j\in \N$. Define
\begin{align*}
\hat{r}_j&:= |x_1-x^j|<\rho,\hspace{7mm} \tilde{r}_j:= \Big(\frac{|t_1-t^j|}{\varepsilon\mu_+^{3-m-p}}\Big)^\frac{1}{p}<\rho,
\\
\hat{r}&:= |x_1-x_2|<\rho,\hspace{7mm} \tilde{r}:= \Big(\frac{|t_1-t_2|}{\varepsilon\mu_+^{3-m-p}}\Big)^\frac{1}{p}<\rho.
\end{align*}
Take now $r_j\in (\max\{\hat{r}_j,\tilde{r}_j\}, \max\{\hat{r}_j,\tilde{r}_j\}+\frac{1}{j})\cap \Q$ such that $r_j<\rho$. Then $r_j$ converges to $\max\{\hat{r},\tilde{r}\}=:r$. Moreover, $z_1 \in Q_{r_j,\varepsilon\mu_+^{3-m-p}r^p_j}(z^j)\setminus N$ and also $z_2$ belongs to this set for large $j$ so
\begin{align*}
|u(z_1)-u(z_2)|\leq \essosc_{Q_{r_j,\varepsilon\mu_+^{3-m-p}r^p_j}(z^j)}u &\leq c \mu_+\Big(\frac{r_j}{\rho}\Big)^\nu
\\
&\xrightarrow[j\to \infty]{} c \mu_+\Big(\frac{r}{\rho}\Big)^\nu
\\
&\leq c \mu_+\rho^{-\nu} (|x_1-x_2| + \Big(\frac{|t_1-t_2|}{\varepsilon\mu_+^{3-m-p}}\Big)^\frac{1}{p})^\nu
\\
&\leq C |z_1-z_2|^\frac{\nu}{p},
\end{align*}
where the constant $C$ depends on the data and $\rho,\mu_+$. Suppose now instead that $z_1$ does not belong to the set $Q_{\rho,\varepsilon\mu_+^{3-m-p}\rho^p}(z_2)\cup (B_\rho(x_2)\times\{t_2\})$. Then
\begin{align*}
|u(z_1)-u(z_2)|\leq \frac{\max\{u(z_1),u(z_2)\}}{|z_1-z_2|^\frac{\nu}{p}}|z_1-z_2|^\frac{\nu}{p}\leq \mu_+ \min\{ \rho, \varepsilon\mu_+^{3-m-p}\rho^p\}^{-\frac{\nu}{p}} |z_1-z_2|^\frac{\nu}{p}.
\end{align*}
Thus, we have verified that for all $z_1,z_2\in \bar{B}^{n+1}_{R}(z_o)\setminus N$,
\begin{align}\label{holder-est}
|u(z_1)-u(z_2)|\leq C |z_1-z_2|^\frac{\nu}{p},
\end{align}
for a constant $C=C(m,n,p,C_0,C_1,\mu_+,R)$. (Note that $\rho$ depends only on $R$, the data and $\mu_+$.) Since the set $N$ has measure zero, we can re-define $u$ at every point of $N$ as the unique limit guaranteed by \eqref{holder-est} when approaching the point through the set $\bar{B}^{n+1}_{R}(z_o)\setminus N$. In this way we obtain a representative of $u$ which satisfies \eqref{holder-est} for all points $z_1,z_2 \in \bar{B}^{n+1}_{R}(z_o)$.
\end{proof}

\section{Harnack estimates}\label{sec:Harnack-est}
We conclude this paper considering the Harnack inequality for solutions of parabolic singular supercritical equations. Such results were proved in \cite{DiGiaVe1} for equations of parabolic $p$-Laplace and porous medium type. For doubly nonlinear equations see \cite{FoSoVe3} under more restrictive assumptions. Our method is based on the pattern scheme of \cite{DuMoVe}.
\

Let us state and prove some lemmata.

\begin{lem}[Measure-to-point estimate]\label{lem:MTP-est}
Let $u\ge 0$ be a weak solution of \eqref{general}. Suppose that $B_{16\rho}(x_o)\times [t_o,t_o + M^{3-m-p}\rho^p] \subset \Omega_T$. Let $\mu\in(0,1]$ and suppose that 
\begin{align}\label{assumption-hehu}
|B_\rho(x_o)\cap \{u(\cdot, t_o)\geq M\}|\geq \mu |B_\rho(x_o)|.
\end{align}
Then there exist constants $\xi,\tau\in (0,1)$ depending only on the data and $\mu$, such that 
\begin{align*}
u\geq \xi M, \textrm{ in } B_{2\rho}(x_o)\times [t_o + \frac{\tau}{2} M^{3-m-p}\rho^p, t_o + \tau M^{3-m-p}\rho^p].
\end{align*}
Moreover, $\tau$ can be chosen arbitrarily small by decreasing $\xi$.
\end{lem}
\begin{proof}[Proof]
Assumption \eqref{assumption-hehu} and the fact that $B_\rho(x_o)\times [t_o,t_o+ M^{3-m-p}\rho^p]$ is contained in the domain $\Omega_T$ allow us to apply Lemma \ref{snails} to conclude that there exists $\epsilon(\mu)$ such that 
\begin{align}\label{pellini}
|B_\rho(x_o)\cap \{ u(\cdot,t)\geq \epsilon M\}|\geq \frac{\mu}{2}|B_\rho(x_o)|,
\end{align}
for all $t \in (t_o, t_o + \delta M^{3-m-p}\rho^p)$. Here, $\delta =\delta(data, \mu) \in (0,1)$ is the constant from Lemma \ref{snails}. In order to facilitate the latter part of the proof we note that we may instead use $\delta =\delta(data, \frac{\mu}{2})$ which by the construction in the proof of Lemma \ref{snails} is a smaller number. Note that \eqref{pellini} remains valid if we replace $M$ by any $\theta M$, where $\theta \in (0,1]$. Since $B_{16\rho}(x_o)\times [t_o,t_o+M^{3-m-p}\rho^p]$ is contained in the domain, we may apply Theorem \ref{theo:exppos} with $M$ replaced by $\epsilon\theta M$, $\alpha=\frac{\mu}{2}$ and considering all $s$ in $(t_o, t_o + \delta M^{3-m-p}\rho^p)$ for which $s+\delta (\epsilon \theta M)^{3-m-p}\rho^p)\leq t_o +M^{3-m-p}\rho^p$. Thus, we obtain
\begin{align*}
u\geq \eta \epsilon \theta M \textrm{ in } B_{2\rho}(x_o)\times (t_o+(1-\varepsilon)\delta (\epsilon \theta M)^{3-m-p}\rho^p, t_o + \delta M^{3-m-p}\rho^p).
\end{align*}
Here, $\eta$ and $\varepsilon$ only depend on the data and $\mu$. For any $\tau \in (0, \delta)$ we may thus first choose $\theta$ so small that 
\begin{align*}
(1-\varepsilon)\delta (\epsilon \theta )^{3-m-p} < \tau/2
\end{align*}
and then choose $\xi = \eta \epsilon\theta$.
\end{proof}

We now prove an alternative form of the reduction of the oscillation which will be convenient in the sequel.
\begin{lem}[Estimates of H\"older regularity]
Let $u$ be a weak solution of \eqref{general} in $\Omega_{T}$ in the supercritical range. Then for any $S>0$ there exist constants $\bar C >0$ and $\bar \alpha>0$ depending only on $S$ and the data, such that if
$Q_{32R, k^{3-m-p}R^p}(z_o)\subset \Omega_T$ for some $k, R>0$ then
\begin{equation}
\label{oscp<2}
\sup_{Q_{R, k^{3-m-p}R^p}(z_o)}u\le S\, k\quad \Rightarrow\quad \osc_{Q_{r, k^{3-m-p}r^p}} u  \le \bar C\, k\, \big(\frac{r}{R}\big)^{\bar\alpha},\quad r\le R.
\end{equation}
\end{lem}
\begin{proof}[Proof]
Let $\varepsilon$ be the constant from Lemma \ref{lem:mu-plus_red_osc} and define the re-scaled function
\begin{align*}
v(x,t)= S^{-1}u(x,t_o+\varepsilon^{-1}t), \hspace{7mm}(x,t)\in Q_{32R, \varepsilon k^{3-m-p} R^p}(x_o,0).
\end{align*}
Then $v$ satisfies an equation of type \eqref{general}, where the constants appearing in the structure conditions depend only on $S$ and the data from the original problem. Furthermore,  
\begin{align*}
\sup_{Q_{R, \varepsilon k^{3-m-p} R^p}(x_o,0)} \leq k,
\end{align*}
so Lemma \ref{lem:osc-red} implies that for all $r\in (0,R]$, 
\begin{align*}
\osc_{Q_{r,\varepsilon k^{3-m-p} r^p}(x_o,0)} v \leq \tilde{c} k \big(\frac{r}{R}\big)^{\bar\alpha},
\end{align*}
where $\tilde{c}$ and $\bar{\alpha}$ only depend on $S$ and the data of the original problem. Expressing this estimate in terms of $u$ and the original coordinates we obtain the desired estimate with $\bar C = \tilde{c} S$.
\end{proof}

We will also use the following version of the expansion of positivity. 
\begin{lem}[Expansion of positivity] \label{epossing}
There exists $\bar \lambda>p/(3-m-p)$ and, for any $\mu>0$, $c(\mu), \gamma_{1}(\mu), \gamma_{2}(\mu)\in \ (0, 1)$ depending only on $\mu$ and the data, such that if $u\ge 0$ is a solution in $B_{16R}(\bar{0})\times[0, k^{3-m-p}R^p]$ then
\begin{align}\label{exposingu}
&|B_{r}(\bar{0}) \cap \{u(\cdot, 0)\ge k\}|\ge \mu |B_r(\bar{0})| 
\\
\notag &\Rightarrow\quad \inf_{B_{\rho}}u\big(\cdot, k^{3-m-p}\, r^{p}\, \big(\gamma_{1}(\mu)+\gamma_{2}(\mu)\big(1-(r/\rho)^{\bar\lambda(3-m-p)-p}\big)\big)\ge c(\mu)\, k \Big(\frac{r}{\rho}\Big)^{\bar\lambda},
\end{align}
whenever $r< \rho\le R$. Here, $\gamma_1(\mu)$ and $\gamma_2(\mu)$ are so small that  $\gamma_1(\mu) + \gamma_2(\mu) \leq 1 $, which guarantees that the time level is contained in the interval $k^{3-m-p}R^p$. Moreover, the $\gamma_{i}(\mu)$ can be chosen arbitrarily small by lowering $c(\mu)$.
\end{lem}

\begin{proof}[Proof]
Suppose that the measure condition of \eqref{exposingu} holds. Then, by Lemma \ref{lem:MTP-est}, we have 
\begin{align}\label{claim1-reformulated}
u\geq \xi(\mu) k, \textrm{ in } B_{2r}(\bar{0})\times [\frac{\tau(\mu)}{2} k^{3-m-p}r^p,  + \tau(\mu) k^{3-m-p}r^p].
\end{align}
Denote $\xi_1:=\xi(1)$ and note that, since $m+p<3$, we can suppose that 
\begin{align}\label{b_1-bound}
b_1:= 2^{p}\, \xi_1^{3-m-p}\leq \tfrac{1}{2}.
\end{align} 
Consider first the case $2r\leq R$. We may now define
\begin{align*}
\rho_j &:=2^j r, \textrm{ for all } j\in \N \textrm{ such that } \rho_j \leq R,
\\
\tau_1 &:= \tau(1)\leq \tfrac{1}{7}.
\end{align*}
Note that we are considering the case where at least $\rho_1$ is defined. The bound on $\tau_1$ can be obtained due to Lemma \ref{lem:MTP-est}. This might require shrinking $\xi_1$, but this does not violate the bound on $b_1$.
We define recursively
\begin{equation}
\label{tn}
t_0= \frac{\tau(\mu)}{2} k^{3-m-p}r^p, \qquad t_{j+1} = t_j + \frac{ \tau_1}{2}(\xi(\mu) k \xi_1^j)^{3-m-p}\rho^p_{j+1}.
\end{equation}
From \eqref{claim1-reformulated} it follows that $|B_{r}(\bar{0})\cap \{u(\cdot,t_0)\ge \xi(\mu)k\}|=|B_{r}(\bar{0})|$. Hence, we may apply Lemma \ref{lem:MTP-est} with $\mu=1$ repeatedly and obtain
\begin{align}\label{hedeenand}
u\ge \xi(\mu)\xi_1^j k \qquad \text{in }\quad B_{\rho_{j+1}}\times \big[t_j, t_j+\frac{\tau_1}{2}\, (\xi(\mu) k\, \xi_1^{j-1})^{3-m-p}\, \rho_{j}^{p}\big]
\end{align}
for all integers $j\geq 1$ such that $\rho_j \leq R$, provided that the end time of the cylinder in \eqref{hedeenand} does not exceed $k^{3-m-p}R^p$. In fact, this cannot happen, since an explicit calculation shows that for all integers $N\geq 1$,
\begin{align}\label{t_N-ests}
t_N  &= \frac{\tau(\mu)}{2} k^{3-m-p}r^p + \frac{\tau_1}{2}k^{3-m-p}\xi(\mu)^{3-m-p}r^p2^p\sum^{N-1}_{j=0}b_1^j 
\\
\notag &\leq k^{3-m-p}R^p\frac{\tau_1}{2}\Big(1+2^p\frac{1-b_1^N}{1-b_1}\Big)
\\
\notag &\leq k^{3-m-p}R^p 5 \tau_1,
\end{align}
where in the first step we used the fact that $\xi(\mu)\leq \xi_1<1$ and $\tau(\mu)\leq \tau_1$. Thus, we have
\begin{align*}
t_N+\frac{\tau_1}{2}\, (\xi(\mu) k\, \xi_1^{N-1})^{3-m-p}\, \rho_{N}^{p} \leq t_N + k^{3-m-p} R^p 2\tau_1 \leq k^{3-m-p}R^p 7 \tau_1 \leq k^{3-m-p}R^p,
\end{align*}
which means that the cylinders are all contained in the domain of $u$. From \eqref{b_1-bound} we infer $t_j + \frac{\tau_1}{2}\, (\xi(\mu)\, k \xi_1^{j-1})^{3-m-p}\, \rho_{j}^{p} \geq t_{j+1}$, and thus \eqref{hedeenand} implies that
\[
u\ge \xi(\mu)\Big(\frac{r}{\rho_j}\Big)^{\bar{\lambda}} k \qquad \text{in }\quad B_{\rho_{j+1}}\times [t_j, t_{j+1}],
\]
where $\bar\lambda=-\log_{2}\xi_1 >p/(3-m-p)$.
Using the first line of \eqref{t_N-ests} we can re-write $t_N$ as  
\begin{align*}
t_N &= k^{3-m-p} r^p\Big[ \frac{\tau(\mu)}{2} + \frac{2^{p-1}\tau_1 \xi(\mu)^{3-m-p}}{(1-b_1)}(1-b_1^N)\Big]
\\
&= k^{3-m-p} r^p\Big[ \gamma_1(\mu) + \gamma_2(\mu)\Big(1-\Big( \frac{r}{\rho_N}\Big)^{\bar{\lambda}(3-m-p)-p}\Big)\Big].
\end{align*}
For an arbitrary $\rho \in [r,R]$ we now choose the smallest integer $N$ such that $\rho\leq 2^{N+1} r$. But this means that 
\begin{align*}
\rho_N=2^N r\leq \rho\leq R.
\end{align*}
Thus, we may conclude that 
\begin{align*}
u\geq \xi(\mu)\Big(\frac{r}{\rho_N}\Big)^{\bar{\lambda}} k \geq \xi(\mu)\Big(\frac{r}{\rho}\Big)^{\bar{\lambda}} k = c(\mu) \Big(\frac{r}{\rho}\Big)^{\bar{\lambda}} k \quad \textrm{ in } B_\rho\times [t_N, t_{N+1}].
\end{align*}
It now suffices to note that since $\rho_N\leq \rho \leq \rho_{N+1}$,
\begin{align*}
[t_N,t_{N+1}] \ni k^{3-m-p} r^p\Big[ \gamma_1(\mu) + \gamma_2(\mu)\Big(1-\Big( \frac{r}{\rho}\Big)^{\bar{\lambda}(3-m-p)-p}\Big)\Big].
\end{align*}
By the definitions it is clear that $\gamma_1(\mu)$ and  $\gamma_2(\mu)$ can be made arbitrarily small by lowering $c(\mu)$. It only remains to consider the case that $2r>R$. But in this case a bound of the correct form follows already from \eqref{claim1-reformulated} since $r<\rho<2r$.
\end{proof}

Since we are considering the super-critical range, Theorem \ref{theo:Lr-Linfty} holds with $r=1$. Combining this result with the $L^1$- Harnack estimate of Theorem \ref{harnack}, we immediately obtain the following lemma.
\begin{lem}\label{sl1}
Let $u$ be a solution to \eqref{general} for some $m, p$ satisfying \eqref{supercritical} and suppose that $\bar Q_{4\rho, 2\tau}(z_o) \subset \Omega\times[0,T)$. Then
\begin{equation}\label{l1har}
\sup_{Q_{\rho,\tau}(z_o)} u \leq c\tau^{-\frac{n}{\lambda}} \Big[\inf_{t\in [t_o-2\tau, t_o]}\int_{B_{4\rho}(x_o)}u(x, t)\d x\Big]^{\frac{p}{\lambda}}+ c\Big(\frac{\tau}{\rho^p}\Big)^{\frac{1}{3-m-p}},
\end{equation}
where $\lambda=p+n(m+p-3)$ and the constant $c$ only depends on $m,n,p,C_0,C_1$.
\end{lem}
Here we are able to use the actual infimum and supremum rather than their essential equivalents, since we are considering the continuous representative of $u$. Similar results have been shown previously in \cite[Appendix A]{DiGiaVe1} for the $p$-Laplacian with $p<2$  and in \cite{FoSoVe3} for singular doubly nonlinear equations under more restrictive assumptions. 

Now we are ready prove the final result of this paper. For simplicity, we have opted to formulate and prove the theorem for a cylinder centered at the origin, but obviously the result is translation invariant. Note that since an infimum can only increase when passing to a smaller set, we could replace the ball in the right estimate in \eqref{singharnack} by $B_{R/4}(\bar 0)$, so that the supremum and infimum are taken over the same ball.

\begin{theo}[Harnack inequality]
Let $u\ge 0$ solve \eqref{general} for some $m, p$ satisfying \eqref{supercritical}, in a domain containing $B_{34R}(\bar 0) \times [- T, T]$. Suppose that $u(0, 0)>0$ and 
\begin{equation}
\label{tass}
4\, R^{p}\, \sup_{B_{2R}(\bar{0})}u(\cdot, 0)^{3-m-p}\le T.
\end{equation}
Then there exist constants $\bar C\ge 1$, $\bar\theta>0$ depending only on the data such that 
\begin{align}
\label{singharnack}
\bar C^{-1}\, &\sup_{B_{R/4}(\bar 0)} u(\cdot, s)\le u(0, 0)\le \bar C\, \inf_{B_{R}(\bar 0)}u(\cdot, t),
\\
\notag &\textrm{for }- \bar\theta\, u(0, 0)^{3-m-p}\, R^{p}\le s, t\le\bar\theta\, u(0, 0)^{3-m-p}\, R^{p}.
\end{align}
\end{theo}
\begin{proof}[Proof]
In the cylinder $B_{34}(\bar 0)\times [-T', T']$, where $T'= T\, R^{-p}\, u(0, 0)^{m+p-3}$, the function 
\begin{align*}
v(x,t)=u(0, 0)^{-1}\, u(R\, x, R^{p}\, u(0, 0)^{3-m-p}\, t), 
\end{align*}
satisfies a doubly singular equation with the same structure conditions as the original equation. With these definitions,  \eqref{tass} implies
\begin{equation}\label{tass2}
1\le M^{3-m-p}:=\sup_{B_{1}(\bar 0)}v(\cdot, 0)^{3-m-p}\le T'/4,
\end{equation}
where the left inequality follows from the fact that $v(\bar{0},0)=1$.
We first prove the $\inf$ bound in \eqref{singharnack}. Let $\bar\lambda > p/(3-m-p)$ be the expansion of positivity exponent, define $\psi(\rho)=(1-\rho)^{\bar\lambda}\, \sup_{\bar{B}_{\rho}}v(\cdot, 0)$ for $\rho\in [0, 1]$ and choose $\rho_{0}\in [0,1]$, $x_o\in \bar{B}_{\rho_{0}}(\bar 0)$ such that
\[
\max_{[0, 1]}\psi=\psi(\rho_{0})=(1-\rho_{0})^{\bar\lambda}\, v_{0},\qquad v_{0}=v(x_o, 0)\ge 1.
\]
Let $\bar\xi\in [0, 1)$ be the unique number such that $(1-\bar\xi)^{-\bar\lambda}=2$. Setting $r=\bar\xi\, (1- \rho_{0})$ we have
\begin{align}
\label{padf}
v_{0}\, r^{\bar\lambda}=  \psi(\rho_0) \bar\xi^{\bar{\lambda}}   \geq \bar\xi^{\bar\lambda}, 
\end{align}
where we used the fact that $\psi(\rho_0)\geq \psi(0)=1$. Furthermore, we may estimate
\begin{align}\label{padf2}
\sup_{\bar B_{r}(x_{0})}v(\cdot, 0) &\leq (1-[\bar \xi(1-\rho_0)+\rho_0])^{-\bar{\lambda}}  (1-[\bar \xi(1-\rho_0)+\rho_0])^{\bar{\lambda}}  \sup_{\bar B_{\xi(1-\rho_0)+\rho_0}(\bar{0})}v(\cdot, 0)  
\\
\notag &= (1-\bar \xi)^{-\bar{\lambda}}(1-\rho_0)^{-\bar{\lambda}}\psi(\bar \xi(1-\rho_0)+\rho_0)
\\
\notag &\leq (1- \bar \xi)^{-\bar{\lambda}}(1-\rho_0)^{-\bar{\lambda}}\psi(\rho_0)
\\
\notag &= (1-\bar\xi)^{-\bar\lambda}\, v_{0}
\\
\notag &=2\, v_{0}.
\end{align}
Let $a:=v_{0}^{3-m-p}\, r^{p}$. By construction $v_{0}\le M$ and by \eqref{tass2}, $B_{r}(x_o)\times [-4\, a, 4\, a]$ is contained in the domain of $v$. Thus we can apply Lemma \ref{sl1} to conclude that
\begin{align}\label{btut}
\sup_{B_{\frac{r}{4}}(x_o)\times [-a, a]}v&\le \frac{c}{a^{\frac{n}{n(m+p-3)+p}}} \Big( \int_{B_{r}(x_o)}v(x, 0)\, dx\Big)^{\frac{p}{n(m +p-3)+p}}+ c\, a^{\frac{1}{3-m-p}}\, r^{\frac{p}{m+p-3}}
\\
\notag &\le  c\,\frac{ (2\, v_{0}\, r^{n})^{\frac{p}{n(m+p-3)+p}}}{(v_{0}^{3-m-p}\, r^{p})^{\frac{n}{n(m+p-3)+p}}}+ c\, v_{0}\le  c\, v_{0},
\end{align}
where we used \eqref{padf2} to bound the integral. The constant $c$ depends only on the data. Since $a=v_{0}^{3-m-p}\, r^{p}$, we can apply \eqref{oscp<2} with $k=v_{0}$, and taking $S$ to be the constant $c$ from the last line of the previous estimate,  in both $B_{r/4}(x_o)\times [-v_0^{3-m-p}(r/4)^p, 0]$ and $B_{r/4}(x_o)\times [v_0^{3-m-p}\rho^p -v_0^{3-m-p}(r/4)^p, v_0^{3-m-p}\rho^p]$ for any $\rho\leq r/4$ to get
\[
{\rm osc}(v, B_{\rho}(x_o)\times [- v_0^{3-m-p}\rho^p, v_0^{3-m-p}\rho^p])\le \bar c\, v_{0}\, (\rho/r)^{\bar\alpha},\qquad \rho\le r/4,
\]
where the constants $\bar c$ and $\bar \alpha$ only depend on the data. This estimate also relies on the fact that $B_{8 r}(x_o)\times[-a,a]$ is contained in $B_{8}\times [-T', T']$, and hence in the domain of $v$. As $v(x_o)=v_o$ we infer that 
\begin{align*}
v \ge v_o/2 \quad \textrm{ in } B_{\bar\eta r}(x_o)\times [-\bar\eta^{p}\, a, \bar\eta^{p}\, a],
\end{align*}
for some suitable $\bar\eta\in \ (0, 1/4)$ depending only on the data. Thus, 
\begin{align*}
|B_r(x_o)\cap \{ v(\cdot,t)\ge v_{0}/2\}|\ge \bar\eta^{n}|B_r(x_o)|,
\end{align*} 
for all $|t|\le  v_{0}^{3-m-p}\, \bar\eta^p\, r^{p}$. For any such time, the cylinder $B_{32}(x_o)\times [t, t + (v_0/2)^{3-m-p}2^p]$ is contained in the domain of $v$, so we may apply Lemma \ref{epossing} with $k=v_0/2$ and $R=\rho=2$. 
Choosing the $\gamma_{i}(\bar\eta^{n})$ so small that $\gamma_{1}(\bar\eta^{n})+\gamma_{2}(\bar\eta^{n})<\bar\eta^p/2$, its conclusion implies, thanks to $B_{2}(x_{0})\supseteq B_{1}$,
\begin{align*}
\inf_{B_{1}} v(\cdot, t+ \gamma_{r} \, v_{0}^{3-m-p}\, r^{p})\ge \bar c\, v_{0}\, r^{\bar\lambda},\qquad \gamma_{r}:=\gamma_{1}(\bar\eta^{n})+\gamma_{2}(\bar\eta^{n})\big(1-(r/2)^{\bar\lambda(3-m-p)-p}\big)<\bar\eta^p/2
\end{align*}
for all $|t|\le \bar\eta^p\, v_{0}^{3-m-p}\, r^{p}$. The latter readily gives $v(x, t)\ge \bar c\, v_{0}\, r^{\bar\lambda}$ for $x\in B_{1}$ and $|t|\le\bar\eta^p\, v_{0}^{3-m-p} \, r^{p}/2$. Finally, observe that since $r\le 1$ and $\bar\lambda\ge p/(3-m-p)$, it holds $v_{0}^{3-m-p} \, r^{p}\ge (v_{0}\, r^{\bar\lambda})^{3-m-p}$, so that \eqref{padf} yields $v(x, t)\ge \bar c\, \bar\xi^{\bar\lambda}=:1/\bar C$ for $x\in B_{1}$ and $|t|\le\bar\eta^p\, \bar\xi^{\bar\lambda(3-m-p)}/2=:\bar\theta$. Expressing this in terms of $u$, we obtain the estimate for the infimum in \eqref{singharnack}.

To prove the bound for the supremum we proceed similarly. Indeed, let $x_{*}\in \bar B_{R}(\bar 0)$ be such that $u(x_{*},0)=\sup_{\bar B_{R}(\bar 0)}u(\cdot, 0)$ and define the rescaled translated function 
\begin{align*}
w(x,t)=u(x_*,0)^{-1}u\big(x_*+Rx, R^p u(x_*,0)^{3-m-p}t\big), \quad (x,t)\in B_{65}(\bar{0})\times [-\tilde{T}, \tilde{T}],
\end{align*}
where $\tilde{T}=R^{-p}u(x_*,0)^{m+p-3}T$. Proceeding as before, we obtain that $w(x,t)\geq \frac{1}{\bar C}$ for $x\in \bar B_1(\bar{0})$ and $|t|\leq \bar \theta$. Writing this estimate in terms of $u$ we see that 
\begin{align*}
u(x,t)\geq \frac{u(x_*,0)}{\bar C}, \quad x\in \bar B_{R}(x_*), \quad |t|\leq R^p u(x_*,0)^{3-m-p}\bar \theta.
\end{align*}
Noting that $\bar 0 \in \bar B_R(x_*)$, and taking into account the definition of $x_*$ we obtain
\begin{align}\label{aather}
\bar C^{-1} \sup_{\bar B_{R}(\bar 0)}u(\cdot, 0) \leq u(\bar 0, 0).
\end{align}
Since $u$ is a solution on $\bar B_R(\bar 0)\times [-H,H]$ with $H=4R^pu(\bar 0,0)^{3-m-p}$, we can combine \eqref{aather} and Lemma \ref{sl1} (with $t_o=H/4$ and $\tau=H/2$) to conclude similarly as in \eqref{btut} that 
\begin{align*}
\sup_{B_{R/4}(\bar 0)\times [-H/4,H/4]} u \leq c\, u(\bar 0, 0),
\end{align*}
which concludes the proof.
\end{proof}

\end{document}